\documentclass{cmslatex}
\usepackage[paperwidth=7in, paperheight=10in, margin=.875in]{geometry}
\usepackage[colorlinks,linkcolor=black,anchorcolor=black,citecolor=black]{hyperref}
\usepackage{amsfonts,amssymb}
\usepackage{amsmath}
\usepackage{graphicx}
\usepackage{cite}
\usepackage{enumerate}
\renewcommand{\theequation}{\arabic{section}.\arabic{equation}}

   \allowdisplaybreaks

\usepackage{cleveref}
\usepackage{xspace}
\usepackage{mathrsfs}

\crefname{equation}{}{}

\newcommand{\M}{\ensuremath{\mathbf{M}}\xspace}
\newcommand{\m}{\ensuremath{\mathbf{m}}\xspace}

\renewcommand{\H}{\ensuremath{\mathbf{H}}\xspace}

\newcommand{\g}{\ensuremath{\mathbf{g}}\xspace}
\newcommand{\Real}{\ensuremath{\mathbb{R}}\xspace}
\newcommand{\x}{\ensuremath{\mathbf{x}}\xspace}
\newcommand{\w}{\ensuremath{\mathbf{w}}\xspace}
\renewcommand{\u}{\ensuremath{\mathbf{u}}\xspace}

\newcommand{\f}{\ensuremath{\mathbf{f}}\xspace}

\renewcommand{\S}{\ensuremath{\mathbf{S}}\xspace}
\newcommand{\B}{\ensuremath{\mathbf{B}}\xspace}

\newcommand{\E}{\ensuremath{\mathbf{E}}\xspace}
\newcommand{\e}{\ensuremath{\mathbf{e}}\xspace}
\newcommand{\F}{\ensuremath{\mathbf{F}}\xspace}
\newcommand{\T}{\ensuremath{\mathbf{T}}\xspace}

\renewcommand{\v}{\ensuremath{\mathbf{v}}\xspace}
\newcommand{\V}{\ensuremath{\mathbf{V}}\xspace}
\newcommand{\Z}{\ensuremath{\mathbf{Z}}\xspace}

\newcommand{\A}{\ensuremath{\mathbf{A}}\xspace}
\newcommand{\I}{\ensuremath{\mathbf{I}}\xspace}
\newcommand{\J}{\ensuremath{\mathbf{J}}\xspace}
\newcommand{\C}{\ensuremath{\mathbf{C}}\xspace}
\newcommand{\D}{\ensuremath{\mathbf{D}}\xspace}

\newcommand{\R}{\ensuremath{\mathbf{R}}\xspace}
\newcommand{\X}{\ensuremath{\mathbf{X}}\xspace}
\newcommand{\Y}{\ensuremath{\mathbf{Y}}\xspace}

\renewcommand{\c}{\ensuremath{\mathbf{c}}\xspace}

\newcommand{\Znumbers}{\ensuremath{\mathbb{Z}}\xspace}

\newcommand{\LL}{\ensuremath{\mathscr{L}}}
\renewcommand{\L}{\ensuremath{\mathcal{L}}}
\newcommand{\PP}{\ensuremath{\mathcal{P}}}

\newcommand{\QQ}{\ensuremath{\mathcal{Q}}}
\renewcommand{\AA}{\ensuremath{\mathcal{A}}}
\newcommand{\bnabla}{\boldsymbol\nabla}

\def\longversion{1}

\begin{document}

\title{Homogenization of the Landau-Lifshitz equation}

 \author{Lena Leitenmaier \thanks{ Department of Mathematics, KTH,
     Royal Institute of Technology,
     Lindstedsv\"agen 25, 100 44 Stockholm, Sweden (lenalei@kth.se).}
          \and Olof Runborg \thanks{Department of Mathematics, KTH, Royal Institute of Technology (olofr@kth.se).}}

 \pagestyle{myheadings} \markboth{Homogenization of the
   Landau-Lifshitz equation}{L. Leitenmaier and
   O. Runborg}

\maketitle

\begin{abstract}
  In this paper, we consider homogenization of the Landau-Lifshitz
  equation with a highly oscillatory material coefficient with
  period $\varepsilon$ modeling a ferromagnetic composite. We derive
  equations for the homogenized solution to the problem and the
  corresponding correctors and obtain estimates for the difference
  between the exact and homogenized solution as well as corrected
  approximations to the solution. Convergence rates in $\varepsilon$
  over times $O(\varepsilon^\sigma)$ with $0\leq \sigma\leq 2$ are
  given in the Sobolev norm $H^q$, where $q$ is limited by the
  regularity of the solution to the detailed Landau-Lifshitz
  equation and the homogenized equation.  The rates depend on $q$,
  $\sigma$ and the the number of correctors.
\end{abstract}



\section{Introduction}

The governing equation in micromagnetics is the Landau-Lifshitz equation
\cite{LandauLifshitz35, Gilbert55,Brown63,Aharoni96},
\begin{equation}\label{eq:problem}
  \partial_t \m^\varepsilon = - \m^\varepsilon \times \H(\m^\varepsilon) - \alpha \m^\varepsilon \times \m^\varepsilon \times \H(\m^\varepsilon)\,,
\end{equation}
where $\m^\varepsilon$ is the magnetization vector,
$\H(\m^\varepsilon)$ the so-called effective field and
$\alpha$ a positive damping constant. The first term on
the right-hand side here is a precession term, while the second one
is damping, with the damping parameter $\alpha$ determining the
strength of the effect.
The Landau-Lifshitz equation is important for describing
magnetic materials and processes in applications
like recording devices, discrete storage media,
and magnetic sensors.

In this paper we consider a simplified
version of the 
Landau-Lifshitz equation, where we assume
that $\H(\m^\varepsilon)$ only consists of the exchange interaction
contribution, which in many cases is the term dominating the effective
field,
\[\H(\m^\varepsilon) = \bnabla \cdot (a^\varepsilon(x) \bnabla \m^\varepsilon)\,.\]
We assume that $a^\varepsilon(x) = a(x/\varepsilon)$ is a smooth,
periodic, highly oscillatory material coefficient. This could for instance
be seen as a simple model for a magnetic multilayer
\cite{Hartmann00}, a ferromagnetic composite, consisting of thin layers
of two different materials with different interaction behavior, with
$a^\varepsilon$ indicating the current material. The size of two of
the layers then corresponds to $ \varepsilon$. The most straight
forward coefficient describing such a setup would have rather low
regularity. However, to make the problem more suitable for
mathematical treatment we suppose that $a\in C^\infty$.

Numerical simulations of the Landau-Lifschitz equation
are of considerable interest in applications, \cite{Prohl01,Garcia07}.
For the case when the material changes rapidly, as
above with $\varepsilon\ll 1$, the computational cost of simulations
becomes very high, since the $\varepsilon$-scales must be well resolved
by the numerical approximation. For such problems, multiscale
methods like the heterogeneous multiscale methods (HMM)
\cite{hmm4} and equation free methods \cite{kevrekidis2003}
become more efficient.
These are
inspired by homogenization theory \cite{Bensoussan_Lions_Papanicolaou_1,Cioranescu_Donato_1}.
In the framework of HMM, one combines the approximation of
a coarse macroscale model,
similar to a homogenized equation, with
simulations of  the original detailed equation
\cref{eq:problem}.
The simulations of \cref{eq:problem} are, however, restricted to
small boxes in
space and short time intervals, which reduces the computational cost.
The motivation behind our choice of focus here is to do error analysis
of HMM methods for magnetization dynamics. Such analysis relies
on homogenization theory, and the behavior of
solutions to \cref{eq:problem} over short times, as $\varepsilon\to 0$.
See \cite{arjmand1,arjmand2} for examples of
HMM methods in the context of magnetization dynamics.

There are several articles dealing with the homogenization of
(\ref{eq:problem}) and related problems. In particular, a similar
problem was considered in \cite{hamdache} and recently in
\cite{alouges2019stochastic}, where the authors use two-scale
convergence techniques to analyze \cref{eq:problem} with a
stochastic material coefficient $a^\varepsilon$, which can be seen
as a model for so-called spring magnets, a special type of
ferromagnetic composites.  The corresponding stationary problem was
studied in \cite{AlougesDiFratta15}.
Furthermore, in \cite{highcontrast}, a high contrast
composite medium is considered using two-scale convergence.
In \cite{multilayer} homogenization for ferromagnetic
multilayers in the presence of surface energies is studied, using 
a material coefficient
to describe the magnetic field associated with the exchange energy.
In all of these papers, the authors show convergence for
weak solutions and do not focus on convergence rates in $\varepsilon$,
which is of prime importance for HMM error analysis.
In contrast, our goal is to study how classical solutions to
\cref{eq:problem} can be approximated by the homogenized solution
and associated correction terms.  We note that
while existence of weak solutions to \cref{eq:problem} is shown in
\cite{weakexistence},
existence of classical solutions is only known for short times
and/or for small initial data gradients, see for example
\cite{strong_sol3,strong_sol2,strong_sol4,melcher, weakstrong}. In particular,
in \cite{strong_sol3,strong_sol2}, the authors prove
local existence and global existence given that the gradient of the
initial data is sufficiently small.  In \cite{strong_sol4},
existence of arbitrarily regular solutions with respect to space and
time up to an arbitrary final time is shown on bounded 3D domains,
assuming that the initial data is small enough and has high enough
regularity.  Although these works do not consider exactly the same
Landau-Lifshitz problem as us --- they do not include a varying
material coefficient $a^\varepsilon(x)$ and use slightly different
norms --- we will in this paper assume existence of regular
solutions to \cref{eq:problem} and the corresponding homogenized
equation and focus on convergence rates.
\if\longversion1
Moreover, in
  \Cref{appendix1} and \Cref{appendix2} we generalize the energy
  estimates in \cite{melcher} to the problem considered here.
\fi

In the main result of this paper we analyze the difference between
the solution $\m^\varepsilon$ of \cref{eq:problem} and the
homogenized solution with arbitrary many correction terms. We
provide rates for the convergence in terms of $\varepsilon$ in
Sobolev norms for dimensions $n =1,2,3$.  The rates that we obtain
depend on the length of the time interval considered, and are
centered on short times of length $O(\varepsilon^\sigma)$ with
$0\leq \sigma<2$. These short times are of main relevance for HMM
analysis. Note that the temporal oscillation period in
$\m^\varepsilon$ is of order $\varepsilon^2$ meaning that the times
considered are still relatively long, and include an infinite number
of oscillations in time as $\varepsilon\to 0$.  The approach we use
to achieve this is based on asymptotic multiscale expansions,
together with careful estimates of the corrector terms, inspired by
\cite{assyr}, which used a similar strategy to derive estimates for
the wave equation over long time. Unlike that paper and the ones
mentioned earlier, we include a fast time variable
$\tau=t/\varepsilon^2$ in the multiscale expansions to capture the
precise behavior of the initial transient of the solution.  Our main
assumption, besides existence and regularity
of all solutions, is an $L^\infty$ bound on
$\bnabla\m^\varepsilon$, uniformly in $\varepsilon$.
We note that such a uniform bound is easy to check in $L^2$, and
that it also true
in $L^\infty$ for the homogenized solution with correction terms.

This paper is organized as follows: in the next section we introduce
the notation used in this paper as well as some useful identities.
Section 3 contains the main result of the paper and outlines the
steps that are required to obtain it.
 In Section 4, we motivate our choice of
homogenized equation corresponding to \cref{eq:problem} as well as the form of
the related correctors. We obtain linear partial differential
equations describing the evolution of these correctors. In Section
5, we then show several properties of Bochner-Sobolev
norms that simplify dealing with the multiscale character of the
problem. Section 6 is devoted to a stability estimate for the error
introduced when approximating the solution $\m^\varepsilon$ to
\cref{eq:problem} by the solution to a perturbed version of the
original problem. We then
derive specific bounds for the correctors and the corresponding
approximation to $\m^\varepsilon$ in Section 7.

\section{Preliminaries}

Throughout this paper, the problems are set on
a domain $\Omega \subset \Real^n$,
with $n=1,2,3$ and periodic boundary conditions.
Moreover, for the fast variations we define also
$Y$ as the $n$-dimensional
unit cell, $Y = [0, 1]^n$.

In this section, we introduce notation for working with vector
functions $\v(x, t): \Omega \times \Real \mapsto \Real^3$ and their
gradients. We moreover introduce suitable norms
\if\longversion1
and useful identities
\fi
for working with multiscale problems and matrix valued functions. 


\subsection{Basic notation and differential operators.}
Let $\m:\Omega \times \Real \mapsto \mathbb{S}^2 \subset \Real^3$ denote the magnetization
vector, which is a function of time $t$ and space $x\in\Real^n$.
The components of $\m$ will be called $m^{(j)}$, hence
$\m=[m^{(1)},m^{(2)},m^{(3)}]^T$.  In accordance with standard
notation in the area we denote its Jacobian matrix by $\bnabla
\m$. We consider this as an element in $\Real^{3\times n}$, such that
\if\longversion1
\[
\bnabla \m :=
\begin{bmatrix}
\left(\nabla m^{(1)}\right)^T\\
\left(\nabla m^{(2)}\right)^T\\
\left(\nabla m^{(3)}\right)^T
\end{bmatrix}.
\]
\else
\[
\bnabla \m :=
\begin{bmatrix}
\left(\nabla m^{(1)}\right)^T &
\left(\nabla m^{(2)}\right)^T &
\left(\nabla m^{(3)}\right)^T
\end{bmatrix}^T.
\]
\fi

Suppose that $\A:\Real^n\mapsto \Real^{n\times n}$ gives a symmetric positive definite matrix, uniformly in $x$. Then we define $L$
for a function $u: \Real\times \Real^n \to \Real$ and the corresponding vector-operator $\L$ according to
\[
L u := \nabla \cdot (\A(x) \nabla u), \qquad \L \m :=
  \begin{bmatrix}
    L m^{(1)} &
    L m^{(2)} &    L m^{(3)}
  \end{bmatrix}^T.
\]
In general, all linear operators returning scalars are to be applied
element-wise to vector-valued functions if not explicitly stated otherwise.
As a convention, the cross product and scalar product between a vector-valued function
$\v \in \Real^3$ and $\B$ are done column-wise, and the divergence-operator is applied row-wise.

Moreover, consider $\B, \C \in \Real^{3 \times n}$
with elements $b_{ij}$ and $c_{ij}$ where $i$ and $j$ denote row and column, then we define
\begin{align*}
  \B : \C := \sum_{i=1}^3 \sum_{j=1}^n b_{ij} c_{ij}\,.
\end{align*}
Finally, note that the operator ${\mathcal L}\m$ could also
be defined as ${\mathcal L}\m = \bnabla\cdot (\bnabla\m \A)$ using the
notation introduced above. This is equivalent to the component-wise
definition.


\subsection{Function spaces and norms.}\label{sec:norms}
In the following, we denote by $C(I)$ the space of continuous
functions on an interval $I$ and by $C^\infty(\Omega)$ the space of smooth functions on $\Omega$.
By $H^q(\Omega)$ we denote the standard {periodic} Sobolev
spaces on $\Omega$, with norm $\|\cdot\|_{H^q}$,
$$
 \|v\|^2_{H^{q}} = \sum_{|\beta| \le q} \int_\Omega  |\partial_x^\beta v(x)|^2 dx\,.
$$
Moreover, by $H^{q,p}(\Omega;Y)$ we denote the
{periodic} Bochner--Sobolev
spaces on $\Omega\times Y$, with norm $\|\cdot\|_{H^{q,p}}$,
defined as
\begin{align*}
  \|v\|^2_{H^{q, p}} = \sum_{|\beta| \le q}
  \int_\Omega ||\partial^\beta_x v(x,\cdot)||^2_{H^p(Y)}dx=
  \sum_{|\beta| \le q, |\gamma| \le p} \int_\Omega \int_Y |\partial_x^\beta \partial_y^\gamma v(x,y)|^2 dy dx\,
\end{align*}
and we define the multiscale-norm
\[\| v\|_{H_\varepsilon^q} := \sum_{j=0}^q \varepsilon^j
  \|v\|_{H^j}\,,\] where we assume $0 < \varepsilon \le 1$.  All
previous norm definitions are analogous for vector valued functions.
Furthermore, let $|\cdot|$ denote the norm on $\Real^{3 \times n}$,
for a matrix-valued function $\B \in \Real^{3 \times n}$,
\if\longversion1
\begin{align*}
  |\B|^2 := \B: \B,
\end{align*}
\fi
and the corresponding $L^2$-norm on $\Omega \mapsto\Real^{3\times n}$ is given by
\begin{align*}
  \|\B\|_{L^2}^2 = \int_\Omega |\B|^2 dx = \int_\Omega \B : \B dx\,.
\end{align*}

\if\longversion1
\subsection{Useful identities.}
In the following, we will make frequent use of the standard triple product identities,
stating that it holds for three vector-valued functions $\u, \v, \w \in \Real^3$  that
\begin{align}
  \label{eq:first_triple_product_identity}
  \u \cdot(\v \times \w) &= \v \cdot(\w \times \u), \\
  \label{eq:second_triple_product_identity}
  \u \times \v \times \w &= (\u \cdot \w) \v - (\u \cdot \v) \w\,.
\end{align}
From \cref{eq:first_triple_product_identity} it follows directly that
\begin{align}\label{eq:first_tpi_gradient}
  \bnabla \u : (\v \times \bnabla \u) = \sum_{j=1}^n \partial_{x_j} \u \cdot (\v \times \partial_{x_j} \u) =
\sum_{j=1}^n \v \cdot (\partial_{x_j} \u \times \partial_{x_j} \u) = \boldsymbol 0\,.
\end{align}
In accordance to integration by parts for scalar functions, it holds
given a periodic matrix-valued function $\B \in H^1(\Omega;\Real^{3 \times n})$, a
periodic vector-valued $\v \in H^1(\Omega;\Real^{3})$ and  a periodic real valued function $f \in H^1(\Omega;\Real)$ that
\begin{align*}
  \int_\Omega \v \cdot (\bnabla \cdot \B) dx = - \int_\Omega \B : \bnabla \v dx,
\qquad \text{and} \qquad
  \int_\Omega f \bnabla \cdot \B dx = - \int_\Omega \B \bnabla f dx.
\end{align*}
Moreover, integration by parts implies that
\begin{align}\label{eq:int_m_lm_0}
  \int_\Omega \v \times \L \v dx  &=  0 \, \qquad \text{and} \qquad
  \int_\Omega \u \cdot \L \v dx = - \int_\Omega a \bnabla \u : \bnabla \v dx\,.
\end{align}
\fi
\section{Main results}\label{sec:main}

Assume that $\m^\varepsilon$ is a classical solution to the
Landau-Lifshitz equation on a domain $\Omega \subset \Real^n$,
$n = 1, 2, 3$ with periodic boundary conditions,
\begin{subequations} \label{eq:main_prob}
\begin{align}
  \partial_t \m^\varepsilon(x, t) &= - \m^\varepsilon(x,t) \times \L \m^\varepsilon(x, t) - \alpha
\m^\varepsilon(x,t) \times \m^\varepsilon(x,t) \times \L \m^\varepsilon(x, t) \,, \\
  \m^\varepsilon(x, 0) &= \m_\mathrm{init}(x)\,,
\end{align}
\end{subequations}
 where
$\L \m^\varepsilon := \bnabla \cdot (a^\varepsilon \bnabla \m)$ and $a^\varepsilon(x) := a(x/\varepsilon)$ is a
highly oscillatory, scalar material coefficient.
Moreover, let $\m_0$ satisfy the homogenized equation corresponding to \cref{eq:main_prob} on $\Omega$, which is derived in \Cref{sec:homogenization},
\begin{subequations} \label{eq:main_hom}
\begin{align}
  \partial_t \m_0(x, t) &= - \m_0(x,t) \times \bar \L \m_0(x, t) - \alpha
\m_0(x,t) \times \m_0(x,t) \times \bar \L \m_0(x, t) \,, \\
  \m_0(x, 0) &=  \m_\mathrm{init}(x)\,,
\end{align}
\end{subequations}
where
$\bar\L \m_0 := \bnabla \cdot (\bnabla \m_0 \A^H)$ and
$\A^H \in \Real^{n \times n}$ is the constant homogenized
coefficient matrix.
 Let furthermore
 $\tilde \m_J^\varepsilon$ be a corrected approximation to
 $\m^\varepsilon$, defined as
\begin{align}\label{eq:main_m_tilde_expanded}
  \tilde \m_J^\varepsilon(x, t) = \m_0(x, t) + \sum_{j=1}^J \varepsilon^j \m_j(x, x/\varepsilon, t, t/\varepsilon^2)\,,
\end{align}
where $\m_j$ are higher order correctors obtained by solving linear equations as given in \cref{eq:higher_order_pde}.
Our main goal in this paper then is to investigate the difference
in terms of $\varepsilon$
between $\m^\varepsilon$ and $\m_0$
as well as between $\m^\varepsilon$ and $\tilde \m^\varepsilon_J$.
We assume that the homogenized solution $\m_0$ exists up to time $T$.
For $\m^\varepsilon$ and
the error estimates we mainly consider shorter time intervals
$t\in[0,T^\varepsilon]$, where
\begin{equation}\label{eq:Tepsdef}
   T^\varepsilon:=\varepsilon^\sigma T,\qquad 0\leq \sigma\leq 2.
\end{equation}
We make the following precise assumptions. 
\begin{itemize}
\item [(A1)] The material coefficient function $a(x)$ is in
  $C^\infty(\Omega)$ and such that
  $a_\mathrm{min}\leq a(x)\leq a_\mathrm{max}$ for constants
  $a_\mathrm{min}, a_\mathrm{max} > 0$.
\item [(A2)] The initial data $\m_\mathrm{init}(x)$ satisfies $| \m_\mathrm{init}(x)| \equiv 1$, constant in space.
 Note that the Landau-Lifshiz equation is norm preserving,
\begin{equation}
  \label{eq:norm_pres}
  \tfrac12
  \partial_t |\m|^2 =
\m\cdot \partial_t \m=
\m\cdot\left(\m\times {\mathcal L}\m - \alpha \m\times
  \m\times {\mathcal L}\m\right)=0.
\end{equation}
Hence, this assumption implies that $|\m^\varepsilon(x, t)| \equiv 1$ and
$|\m_0(x, t)| \equiv 1$ for all time.
\item [(A3)] The damping coefficient $\alpha$ and the oscillation period
$\varepsilon$ are small, $0 < \alpha \le 1$ and $0 < \varepsilon < 1$.
\item [(A4)] The solution $\m^\varepsilon$ is such that
  \[\m^\varepsilon \in C^1([0, T^\varepsilon]; H^{s+1}(\Omega)), \qquad \text{for some } s \ge 1,\]
  and there is a constant $M$ independent of $\varepsilon$ such that
  \[\|\bnabla \m^\varepsilon(\cdot, t)\|_{L^\infty} \le M\,, \qquad 0 \le t \le T^\varepsilon\,.\]
\item [(A5)] The homogenized solution $\m_0$ is such that, for some $r\geq 5$,
  \begin{equation}
    \label{eq:m0regularity}
    \partial_t^k \m_0 \in C([0, T]; H^{r-2k}(\Omega)), \qquad 0 \le 2k \le r\,,
  \end{equation}
  which implies that
  \[\|\bnabla \m_0(\cdot, t)\|_{L^\infty} \le C\,, \qquad 0 \le t \le T\,.\]
\end{itemize}

We then obtain the following result.
\begin{theorem}\label{thm:main}
  Assume that $\m^\varepsilon$ is a classical solution to
  \cref{eq:main_prob}, $\m_0$ is a classical solution to
  \cref{eq:main_hom} and that the assumptions (A1)-(A5) are
  satisfied. Let $\tilde \m_J^\varepsilon$ be the corrected
  approximation to $\m^\varepsilon$ as given by
  \cref{eq:main_m_tilde_expanded} and consider the
   final time $T^\varepsilon$
  in \cref{eq:Tepsdef} with $\sigma$ satsifying
  \begin{equation}
  \label{eq:Teps}
\begin{cases}
   0 \le \sigma \le 2, & J \le 2, \\
   1 - \frac{1}{J-2} \le \sigma \le 2, & J \ge 3.
 \end{cases}
\end{equation}
Moreover, let $q_J = \min(s, r - 3 - \max(2, J))$.
Then we have results for three different cases:
 \begin{itemize}
  \item Fixed time, $\sigma=0$:
  \begin{align}\label{eq:main_l2_m0}
  \|\m^\varepsilon(\cdot, t) - \tilde\m^\varepsilon_J(\cdot, t)\|_{L^2} \le C \varepsilon\,,
  \qquad
  \|\m^\varepsilon(\cdot, t) - \tilde\m^\varepsilon_J(\cdot, t)\|_{H^1} \le C\,,
  \end{align}
  for $0 \le t \le T$ and $0 \le J \le 2$,
  provided $r\geq 6$ for the $H^1$ case.

 \item Short time, $0 < \sigma \le 1$:
  \begin{align}\label{eq:main_mJ_1}
  \|\m^\varepsilon(\cdot, t) - \tilde \m_J^\varepsilon(\cdot, t)\|_{H^q} \le C
    \begin{cases}
    \varepsilon^{1+ \sigma/2-q} , & J = 1, \\
    \varepsilon^{2- (1 - \sigma )(J-1)-\sigma/2-q} , & J \ge 2,
  \end{cases}
  \end{align}
  for $0 \le t \le T^\varepsilon$,
  provided $q \le q_J$.

 \item Very short time, $1 < \sigma \le 2$:
\begin{align}\label{eq:main_mJ}
  \|\m^\varepsilon(\cdot, t) - \tilde \m_J^\varepsilon(\cdot, t)\|_{H^q} \le C
    \varepsilon^{2+ (\sigma - 1)(J-1) - \rho(\sigma, q, r, J)-q},
\end{align}
for $0 \le t \le T^\varepsilon$,
provided $q \le q_J$ and $J\geq 1$. Here $\rho(\sigma, q, r, J) := \max(0,~ \tfrac12\sigma - (\sigma-1)(r - 3 - J - q))$.
\end{itemize}
In all cases, the constant $C$ is independent of $\varepsilon$ and $t$ but
depends on $M$ and $T$.
\end{theorem}

For fixed final times of order $\mathcal{O}(1)$
this theorem  shows the expected strong $L^2$ convergence rate of
$\varepsilon$ for $\m_0$
and also for the higher order approximations
$\tilde\m^\varepsilon_1=\m_0+\varepsilon\m_1$
and $\tilde\m^\varepsilon_2=\m_0+\varepsilon\m_1+\varepsilon^2\m_2$. Moreover, the errors with
$\m_0$, $\tilde\m^\varepsilon_1$ and $\tilde\m^\varepsilon_2$
have bounded $H^1$-norms, suggesting weak $H^1$ convergence for
these three approximations.

For the short and very short time cases where $\sigma>0$ we note
that since the temporal oscillation period in the problem is of
order $\varepsilon^2$, as is shown in \Cref{sec:homogenization},
final times with $0<\sigma <2$ are still relatively long, and
include an infinite number of oscillations in time as
$\varepsilon\to 0$.

The second bullet in the theorem shows that for times from
$\mathcal{O}(\varepsilon)$ and up to
$\mathcal{O}(1)$,$0 < \sigma \le 1$, one gets strong convergence of
the $L^2$ and $H^1$-norms when considering the corrected
approximation $\tilde\m^\varepsilon_1$,
\[\|\m^\varepsilon-\tilde\m^\varepsilon_1\|_{L^2} \le C \varepsilon^{1 + \sigma/2}, \qquad \|\m^\varepsilon-\tilde\m^\varepsilon_1\|_{H^1} \le \varepsilon^{\sigma/2}.\]
However, one does not get better
approximations by including more correctors.

For final times shorter than $\mathcal{O}(\varepsilon)$, on the
other hand, one gets better approximations by including
more correctors, as \cref{eq:main_mJ} shows. This is especially
relevant since these are the times that are most interesting in the
context of HMM. For these short times, the regularity of $\m_0$ determines which
converge rate one obtains. In particular, if
\[r \ge J + 3 + q + \left\lceil \frac{\sigma}{2 (\sigma - 1)}
  \right\rceil,\] the penalty term $\rho$ in \cref{eq:main_mJ}
becomes zero and one obtains the optimal estimate for short times. The
longer the time considered, which means the closer $\sigma$ is to one,
the higher is the required regularity.
In particular, if
$\partial_t^k \m_0 \in C([0, T]; H^\infty(\Omega))$, $k \ge 0$, we
get
\[\|\m^\varepsilon(\cdot, t) - \tilde \m_{J}^\varepsilon\|_{H^q} \le C \varepsilon^{2 + (\sigma - 1) (J-1) -q}, \qquad \sigma > 1, ~J > 0.\]
This entails for example the following bounds for $\sigma = 3/2$ and $\sigma = 2$,
\[\sigma = 3/2: ~\|\m^\varepsilon - \tilde \m_{J}^\varepsilon\|_{H^q} \le C \varepsilon^{0.5 J + 1.5 - q}, \qquad \qquad \sigma = 2: ~\|\m^\varepsilon - \tilde \m_{J}^\varepsilon\|_{H^q} \le C \varepsilon^{J + 1 - q}\,.\]
Choosing $J$ high enough, we can obtain any convergence rate we want
for these errors.

Note that the first corrected approximation is of the form
$$
   \tilde{\m}_1(x,t) =
   \m_0(x,t)+\varepsilon
   \nabla \m_0(x,t) \boldsymbol\chi(x/\varepsilon)
   +\varepsilon \v(x,x/\varepsilon,t,t/\varepsilon^2).
$$
The part $\nabla\m_0 \boldsymbol\chi$ is familiar from
homogenization of elliptic operators
with $\boldsymbol\chi$ being the solution of the cell problem
\cref{eq:cell_problem}.
The second
part $\v$ is special for \cref{eq:main_prob}.
It satisfies the linear PDE \cref{eq:v_total} and oscillates both in
time and space, with the time variations decaying
exponentially. 
See \Cref{sec:homogenization}.

\subsection{Proof of Theorem 3.1}

We begin with a preliminary estimate,
based on \Cref{thm:error_norm}, which we subsequently improve to obtain the results in \Cref{thm:main}.
In \Cref{thm:m_pert} we show that the approximation
$\tilde \m_J^\varepsilon$, \cref{eq:main_m_tilde_expanded},
satisfies a perturbed version of \cref{eq:main_prob},
\begin{subequations} \label{eq:main_m_tilde}
\begin{align}
  \partial_t \tilde \m_J^\varepsilon(x, t) &= - \tilde \m_J^\varepsilon(x, t) \times \L \tilde \m_J^\varepsilon(x, t) - \alpha
\tilde \m_J^\varepsilon(x, t) \times \tilde \m_J^\varepsilon(x, t) \times \L \tilde \m_J^\varepsilon(x, t) + \boldsymbol \eta^\varepsilon_J\,,\\
\tilde \m_J^\varepsilon(x, 0) &= \m_\mathrm{init}(x)\,,
\end{align}
\end{subequations}
and that
the norm of the residual $\boldsymbol \eta^\varepsilon_J$ can be bounded as
\begin{equation}\label{eq:main_etabound}
  \|\boldsymbol \eta^\varepsilon_J(\cdot, t)\|_{H^{q}_\varepsilon} \le  C\varepsilon^{1+(\sigma-1)(J-2)},\qquad
  0\leq t\leq T^\varepsilon,
\end{equation}
if we include at least two correctors in the expansion, $J\geq 2$,
and if $0 \le q \le r - 2 -J$. Furthermore,
using \cref{eq:m_tilde_length_alt} after
\Cref{lemma:m_tilde_length},
we show that
\begin{equation}\label{eq:main_gradmbound}
 \|\bnabla |\tilde{\m}_J^\varepsilon(\cdot,  t)|^2\|_{H^q_\varepsilon}  \le C \varepsilon^{2 + (\sigma-1)(J-2)},  \qquad 0\leq t\leq T^\varepsilon,
\end{equation}
under the same conditions. 
This last estimate can be seen as a measure for how rapidly the length of
$\tilde \m^\varepsilon_J$ changes.
\Cref{thm:error_norm} now says that
 the error $\e_J := \m^\varepsilon - \tilde \m_J^\varepsilon$
satisfies
\begin{equation}\label{eq:main_ebound}
    \|\e_J(\cdot, t)\|^2_{H^q} \le C  \frac{t}{\varepsilon^{2q}} \sup_{0 \le s \le t} \left(\|\nabla|\tilde \m_J^\varepsilon(\cdot, s)|^2 \|_{H^{q}_\varepsilon}^2 +  \|\boldsymbol \eta_J^\varepsilon(\cdot, s)\|_{H^q_\varepsilon}^2\right),
\quad 0 \le t \le T^\varepsilon
\,,
\end{equation}
when $q \leq s$ and
\[
  \|\tilde{\m}^\varepsilon_J(\cdot,  t)\|_{W^{k,\infty}} \le  C\varepsilon^{\min(0,1-k)},\qquad 0\leq k \leq q+1,
\]
uniformly for  $t\in [0, T^\varepsilon]$. The latter estimates
are true by \Cref{thm:mJnorms} when $0 \le q \le r-3-J$. Therefore,
combining \cref{eq:main_etabound},  \cref{eq:main_gradmbound}, \cref{eq:main_ebound},
and \cref{eq:Teps}
we get
\begin{equation}
  \label{eq:eJeps_bound}
  \|\e_J(\cdot, t)\|_{H^q} \le C \varepsilon^{2+(\sigma - 1)(J-1)-\sigma/2-q} \,, \quad 0 \le t \le T^\varepsilon\,,
\end{equation}
as long as $0\leq q\leq \min(s,r-3-J)$ and $J\geq 2$.
This completes the preliminary estimate. 

To improve the estimate, we consider the difference between
$\tilde{\m}^\varepsilon_{J}$ and higher order corrections $\tilde{\m}^\varepsilon_{J'}$ with $J'> J$ and $J'\geq 2$.
We write, using
\Cref{lemma:multiscaleest},
\begin{align}\label{eq:mj_error}
  \|\e_{J}\|_{H^q} &\leq
    \|\m^\varepsilon - \tilde\m_{J'}^\varepsilon\|
    +\|\tilde\m_{J'}^\varepsilon- \tilde\m_{J}^\varepsilon\|_{H^q}\leq
    \|\e_{J'}\|_{H^q}+
\sum_{j=J+1}^{J'} \varepsilon^j   \left\|\m_j(\cdot, \cdot/\varepsilon, t, t/\varepsilon^2)\right\|_{H^q}
 \nonumber  \\
  &\le \|\e_{J'}\|_{H^q}  + C \sum_{j=J+1}^{J'} \varepsilon^{j-q} \|\m_j(\cdot, \cdot, t, t/\varepsilon^2)\|_{H^{q, q+2}}.
\end{align}
We then need to use
\Cref{thm:mjestimate}, where it is shown that the norms of the first two
correctors,
$\m_1$ and $\m_2$, are uniformly bounded in
$\tau$, while higher order correctors grow algebraically.  In
particular, it holds for all $p \ge 0$ and $j \le r $ that
\begin{align}\label{eq:corrector-bound1}
  \|\m_j(\cdot, \cdot, t, \tau)\|_{H^{r-j,p}} \le C ( 1 + \tau^{\max(0, j-2)}) \le C \varepsilon^{(\sigma-2) \max(0, j-2)},
\end{align}
for $0 \le t \le T^\varepsilon$ and
$0 \le \tau \le \varepsilon^{-2} T^\varepsilon$.
Entering \cref{eq:eJeps_bound} and \cref{eq:corrector-bound1} in
\cref{eq:mj_error} then shows that
\begin{equation}
  \label{eq:inter_bound}
  \|\e_{J}\|_{H^q}
  \leq
  C\varepsilon^{
  2+(\sigma - 1)(J'-1)-\sigma/2-q}+
  C\sum_{j=J+1}^{J'} \varepsilon^{j+(\sigma-2) \max(0, j-2)-q},
\end{equation}
when $q\leq \min(s,r-3-J')$.

We are now ready to show the final estimates as given in
\Cref{thm:main}. For the first case, where $\sigma=0$, we take
$0\leq J< J'=2$ and $0\leq q\leq 1$. Then \cref{eq:inter_bound}
gives us
$$
  \|\e_{J}^\varepsilon\|_{H^q}
  \leq
  C\varepsilon^{1-q}+
  C\sum_{j=J+1}^{2} \varepsilon^{j-q}\leq C \varepsilon^{1-q},
$$
when $q\leq \min(s,r-5)$, which is automatically satisfied for $q=0$
by (A4) and (A5) but requires $r\geq 6$ for $q=1$.  The result for
$J=2$ follows directly from \Cref{eq:eJeps_bound}.

For the second case, where $0<\sigma\leq 1$, we cannot improve
the preliminary estimate \cref{eq:eJeps_bound} using \cref{eq:inter_bound}
when $J\geq 2$. However, for
$J=1$ and $J'=2$, \cref{eq:inter_bound} gives
$$
  \|\e_{1}\|_{H^q}
  \leq
  C\varepsilon^{1+ \sigma/2-q}+
  C\varepsilon^{2-q}\leq \varepsilon^{1+ \sigma/2-q}.
  $$
  This is valid as long as $q\leq \min(s,r-3-\max(2,J))=q_J$.

Finally, for the third case in \Cref{thm:main}, where $1<\sigma\leq 2$, we
only consider \cref{eq:inter_bound} with $J\geq 1$. Then
$j+(\sigma-2)(j-2)=2+(\sigma-1)(j-2)$ and we get
\begin{align*}
  \|\e_{J}\|_{H^q}
  &\leq
  C\varepsilon^{2+(\sigma - 1)(J'-1)-\sigma/2-q}+
  C\sum_{j=J+1}^{J'} \varepsilon^{j+(\sigma-2)(j-2)-q}\\
  &\leq
  C\varepsilon^{2+(\sigma-1)(J'-1)-\sigma/2-q}+
  C\varepsilon^{2+(\sigma-1)(J -1)-q}
  \\
  &\leq
  C \varepsilon^{2+\min\bigl[(\sigma-1)(J'-1)-\sigma/2,\ (\sigma-1)(J-1)\bigr]-q}\\
  &=
  C \varepsilon^{2+(\sigma-1)(J-1)-\max\bigl[\sigma/2-(\sigma-1)(J'-J),\ 0\bigr]-q},
\end{align*}
where the possible choices of $J'$ are limited by the restrictions
$q\leq \min(s,r-3-J')$, $J'\geq 2$ and $J'>J$.  When
$q=r-3-\max(2,J)$ we can therefore not choose $J'$ such that we get
an improvement. Hence \cref{eq:main_mJ} is the same as the
preliminary estimate \cref{eq:eJeps_bound} in that case.  It thus
only remains to prove the case $q< r-3-\max(2,J)$.  We are then
allowed to take $J'=r-3-q>\max(2,J)$ and get
\begin{align*}
  \|\e_{J}\|_{H^q}\leq
  C \varepsilon^{2+(\sigma-1)(J-1)-\max(\sigma/2-(\sigma-1)(r-3-q-J),0)-q}.
\end{align*}
The theorem is proved.

\section{Homogenization}\label{sec:homogenization}

In this section we derive differential equations for the homogenized
solution $\m_0$ to \cref{eq:main_prob} and the corresponding
correction terms.  We aim to motivate our choice of equations but do
not include any proofs in this section.  Precise energy estimates
will be done in \Cref{sec:application}.

\subsection{Multiscale expansion.}
We consider the Landau-Lifshitz   \cref{eq:main_prob}
and assume
that we are looking for an asymptotic solution to \cref{eq:main_prob} of the form
\[
\m^\varepsilon(x, t) = \m\left(x, x/\varepsilon, t, t/\varepsilon^2;\varepsilon\right)
\]
for a suitable function $\m(x, y, t, \tau)$. Numerical experiments
suggest that this is the form that is required for our problem. One
example for this is shown in \cref{fig:num_ex} where one can clearly
observe oscillations in space on an $\varepsilon$-scale and
oscillations in time on an $\varepsilon^2$ scale when taking the
difference between $\m^\varepsilon$ satisfying \cref{eq:main_prob}
and the suggested $\m_0$.

\begin{figure}[htb]
  \centering
  \includegraphics[width=\textwidth]{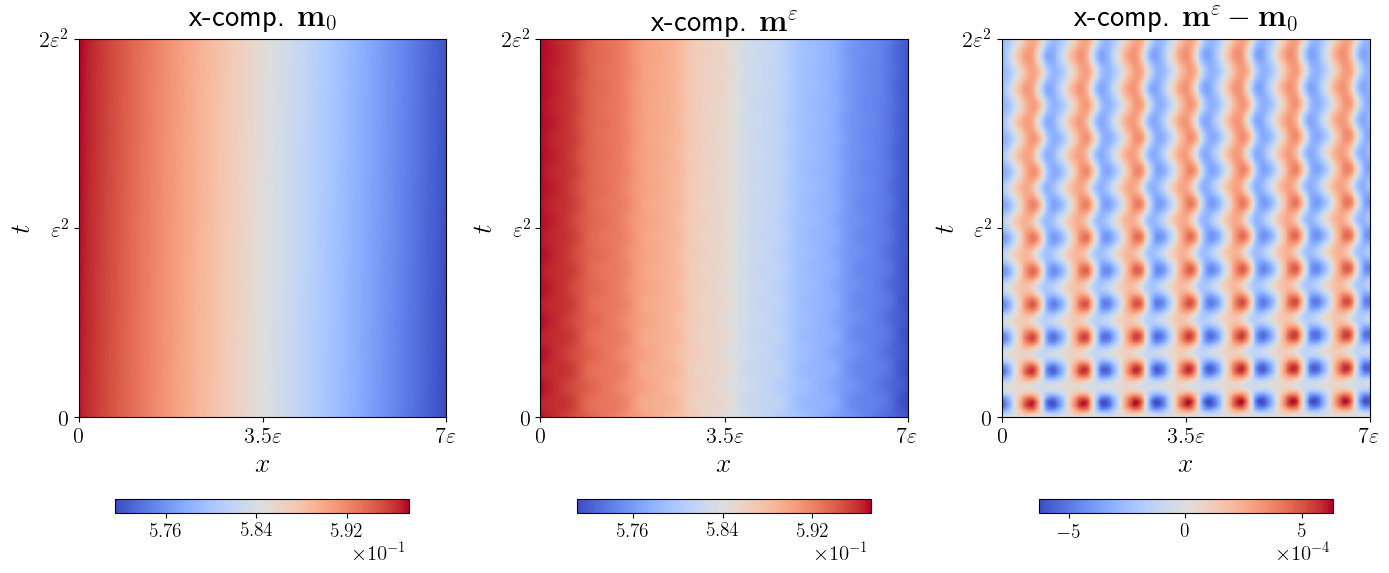}
  \caption{Numerical example: $x$-component of the solution
    $\m^\varepsilon$ to \cref{eq:main_prob} in 1D and the
    corresponding $\m_0$ according to $\cref{eq:main_hom}$ when
    choosing $a^\varepsilon(x) = 1 + 0.5\sin(2\pi x/\varepsilon)$,
    $\varepsilon=1/70$, $\alpha = 0.02$  and
    initial data
    $\m_\mathrm{init}(x) = \m_{nn}(x) / |\m_{nn}(x)|$
    where  $\m_{nn}(x) = 0.5 + [\exp(-0.1\cos(2\pi(x-0.2))),
    \exp(-0.2\cos(2\pi x)),
    \exp(-0.1\cos(2\pi(x-0.8)))]^T$ on a subset
    $[0, 7\varepsilon]$ of the domain
    $\Omega = [0, 1]$ and a short time interval
    $0 \le t \le 2\varepsilon^2$. \label{fig:num_ex}}
\end{figure}
Taking derivatives of $\m^\varepsilon(x, t)$, one obtains
\begin{align*}
  \nabla \m^\varepsilon(x, t) &= \nabla_x \m(x, y, t, \tau;\varepsilon) + \frac{1}{\varepsilon} \nabla_y \m(x, y, t, \tau;\varepsilon)\,, \\
\partial_t \m^\varepsilon(x, t) &= \partial_t \m(x, y, t,
  \tau;\varepsilon) + \frac{1}{\varepsilon^2} \partial_\tau \m(x, y,
  t, \tau;\varepsilon)\,,
\end{align*}
where $y := \frac{x}{\varepsilon}$ is
the fast variable in space and $\tau := \frac{t}{\varepsilon^2}$ the
fast variable in time.  The differential operator $\L$ can
accordingly be rewritten in the form
\begin{align*}
  \L = \L_0 + \frac{1}{\varepsilon} \L_1 + \frac{1}{\varepsilon^2} \L_2,
\end{align*}
where $\L_0, \L_1$ and $\L_2$ are the vector-operators corresponding to the scalar operators
\begin{align*}\label{eq:operators}
  L_0 &:= \nabla_x \cdot(a(y) \nabla_x)\,, 
\qquad
  L_1 := \nabla_x \cdot(a(y) \nabla_y) + \nabla_y \cdot (a(y) \nabla_x)\,, 
\qquad
  L_2 := \nabla_y \cdot(a(y) \nabla_y) \,.
\end{align*}
We are looking for an asymptotic expansion for $\m$,
\begin{equation}\label{eq:asymptotic_expansion}
  \m(x,y,t,\tau;\varepsilon) = \m_0(x, t) + \sum_{j=1}^\infty \varepsilon^j \m_j(x, y, t, \tau),
\end{equation}
where we assume that $\m_0 = \m_0(x, t)$ only depends on the slow
variables, $x$ and $t$, and that the correctors $\m_j$, $j = 1, 2, \ldots$ are 1-periodic in $y$.

Before we consider an expanded version of the differential equation
\cref{eq:main_prob}, we start by introducing suitable notation that
will help us to keep track of terms of the same structure throughout the rest of this paper.
First, we let $\m_{-1}(x, t) := 0$ and define
\begin{align}\label{eq:V_notation}
  \V_j :=
  \L_2 \m_{j} + \Z_{j-1},\quad j\geq 1,
  \quad\text{and}
  \quad
\Z_j :=
\begin{cases}
 \L_1 \m_{0}, & j=0, \\
 \L_0 \m_{j-1} + \L_1 \m_{j}, & j\geq 1.
  \end{cases}
\end{align}
Furthermore, let for $j \ge 1$,
\begin{align}\label{eq:T_notation}
  \T_{j} := \sum_{k=1}^j \m_{j-k} \times \V_{k} =  \m_0 \times \V_j + \mathbf{R}_{j-1} \,,
\quad 
\mathbf{R}_j :=
  \begin{cases}
 0, & j=0, \\
  \sum_{k=1}^{j} \m_{j+1-k} \times \V_{k}, & j\geq 1,
  \end{cases}
\end{align}
and finally
\begin{align}\label{eq:S_notation}
  \mathbf{S}_j :=
  \begin{cases}
 0, & j=0, \\
  \sum_{k=1}^{j} \m_{j+1-k} \times \T_{k}, & j\geq 1.
  \end{cases}
\end{align}
Note that in all of these quantities, $j$ indicates the highest index of all
$\m_j$ that are part of the quantity.

Consider now the expanded version of $\L \m^\varepsilon$, which becomes
\begin{align}\label{eq:Lm_expansion}
  \L \m^\varepsilon = \frac{1}{\varepsilon^2} \L_2 \m_0 + \frac{1}{\varepsilon} ( \L_1 \m_0 + \L_2 \m_1)
+ \sum_{j=0}^\infty \varepsilon^j (\L_0 \m_j + \L_1 \m_{j+1} + \L_2 \m_{j+2})
=: \sum_{j=1}^\infty \varepsilon^{j-2} \V_j\,,
\end{align}
entailing that the precession term in \cref{eq:main_prob} expands to
\begin{align*}\label{eq:scale_der}
  \m^\varepsilon(x, t) \times \L \m^\varepsilon(x, t) &= \sum_{j=0}^\infty \varepsilon^j \m_j \times
\sum_{k=1}^\infty \varepsilon^{k-2} \V_{k} = \sum_{j=1}^\infty \varepsilon^{j-2} \sum_{k=1}^{j} \m_{j-k} \times  \V_{k} = \sum_{j=1}^\infty \varepsilon^{j-2} \T_j,
\end{align*}
and the damping term takes the form
\begin{align*}
  \m^\varepsilon \times \m^\varepsilon \times \L \m^\varepsilon
  &=  \sum_{\ell=0}^\infty \varepsilon^\ell \m_\ell \times \sum_{j=1}^\infty \varepsilon^{j-2} \T_j =  \sum_{\ell=0}^\infty \varepsilon^{\ell-2}  \sum_{j=1}^\ell \m_{\ell-j} \times  \T_{j}\,.
\end{align*}
For the time derivative of $\m^\varepsilon$, it moreover holds that
\[\partial_t \m^\varepsilon = \sum_{j=0}^\infty \varepsilon^j \partial_t \m_j + \varepsilon^{j-2} \partial_\tau \m_j
\,.\]
We  can then formally
rewrite the differential equation \cref{eq:main_prob} as
\begin{align*}
  \sum_{j=1}^\infty \varepsilon^{j-2} (\partial_t \m_{j-2} + \partial_\tau \m_j) =
- \sum_{j=1}^\infty \varepsilon^{j-2} \T_j - \alpha  \sum_{j=0}^\infty \varepsilon^{j-2}  \sum_{k=1}^j \m_{j-k} \times  \T_{k},
\end{align*}
which implies that at scale $\varepsilon^{j-2}$ and for $j\geq 1$, it holds that
\begin{align}\label{eq:higher_order_m_eq}
  \partial_t \m_{j-2}+\partial_\tau \m_{j}
  &= - \T_{j} - \alpha  \sum_{k=1}^{j} \m_{j-k} \times  \T_{k} \,. 
\end{align}
Note that as $\m_0(x, t)$ is independent of $y$ and $\tau$, both $\partial_\tau \m_0(x, t) = 0$ and $\L_2 \m_0(x, t) = 0$.
Based on \cref{eq:higher_order_m_eq}, it is now possible to
show that all the correctors $\m_j$,
$j \ge 1$, satisfy linear
differential equations of a similar structure as the one for $\m_0$.
Since it holds that
\begin{align*}
  \T_j &= \m_0 \times \V_j + \R_{j-1}
=\m_0 \times \L_2\m_j+\m_0 \times \Z_{j-1} + \R_{j-1}, \\
\sum_{k=1}^{j} \m_{j-k} \times  \T_{k}
&=\m_0\times\m_0 \times \L_2\m_j+\m_0\times\m_0 \times \Z_{j-1} + \m_0\times\R_{j-1}+\S_{j-1},
\end{align*}
where $\mathbf{R}_{j-1}$, $\mathbf{S}_{j-1}$ and $\mathbf{Z}_{j-1}$
only contain lower order $\m_k$ with $k \leq j-1$, it follows that
$\m_j$, with $j \ge 1$, satisfies the linear differential equation
\begin{align} \label{eq:higher_order_pde}
  \partial_\tau \m_{j} = - \m_0 \times \L_2 \m_j - \alpha \m_0 \times \m_0 \times \L_2 \m_j + \F_{j-1}
= \LL \m_j + \F_{j}\,,
\end{align}
where the linear operator $\LL$ is defined
such that
 \begin{align} \label{eq:main_ll}
   \LL \m_j := - \m_0 \times \L_2 \m_j - \alpha \m_0 \times \m_0 \times \L_2 \m_j, 
 \end{align}
and
all terms involving only $\m_k$ with $k < j$ are contained in $\F_{j}$,
defined according to
\begin{align}\label{eq:forces}
  \F_j :=  -\mathbf{R}_{j-1}-
  \m_0\times\Z_{j-1}-
   \alpha \left(\m_0 \times \mathbf{R}_{j-1}+  \m_0\times\m_0\times\Z_{j-1}
   +\S_{j-1}\right)
    -  \partial_t \m_{j-2},
\end{align}
for $j \ge 1$.

\subsection{Derivation homogenized equation.}\label{sec:homogeq_deriv}
In order to derive a homogenized equation corresponding to
\cref{eq:main_prob}, we now take a closer look at the differential
equations for $\m_1$ and $\m_2$.
As by definition $\R_0=\S_0=\m_{-1} := 0$,
 \begin{align}\label{eq:forces1}
   \F_1 &=
   -
  \m_0\times\Z_{0}-
   \alpha \m_0\times\m_0\times\Z_{0},
 \end{align}
where $\Z_0 = \L_1 \m_0$, it holds according to
\cref{eq:higher_order_pde} at
scale $\varepsilon^{-1}$ that
\begin{equation}
  \label{eq:m1_eq}
  \partial_\tau \m_1 = - \m_0 \times \V_1 - \alpha \m_0 \times \m_0  \times \V_1\,,
\end{equation}
since $\V_1=\L_2 \m_1 + \L_1 \m_0$.
To find a solution for this equation, we assume that $\m_1$ takes the form
\begin{equation}
  \label{eq:m1}
  \m_1(x, y, t, \tau) = \nabla_x \m_0 \boldsymbol\chi(y) + \v(x, y, t, \tau)\,,
\end{equation}
where $\boldsymbol\chi(y)$ is the solution to the {cell problem}
\begin{equation}
  \label{eq:cell_problem}
  \nabla_y \cdot (a(y) \nabla_y \boldsymbol\chi(y)) = - \nabla_y a(y) \,.
\end{equation}
Note that \cref{eq:cell_problem} only determines $\boldsymbol\chi$
up to a constant. In accordance with standard practice in the
literature
\cite{Bensoussan_Lions_Papanicolaou_1,Cioranescu_Donato_1}, we
assume in the following that this constant is chosen such that
$\boldsymbol\chi(y)$ has zero average.  As, by the definition of
$\boldsymbol\chi(y)$ and the assumption \cref{eq:m1},
\begin{align}\label{eq:L2_m1_L1_m0_L2_v}
  \V_1=\L_2 \m_1 + \L_1 \m_0 = \L_2 \v + \L_2 (\nabla_x \m_0 \boldsymbol\chi ) + \L_1 \m_0 = \L_2 \v\,,
\end{align}
it follows from \cref{eq:m1_eq} that
\begin{equation}
  \label{eq:v_total}
  \partial_\tau \v = - \m_0 \times \L_2 \v - \alpha \m_0 \times \m_0 \times \L_2 \v\,
  = \LL \v\,.
\end{equation}
This is a linear differential equation with the same structure as
\cref{eq:higher_order_pde}, but with forcing $\F = 0$.  At the
initial time, $\tau = 0$, we set $\m_1(x, y, t, 0) = 0$ and hence
have $\v(\tau = 0, y) = - \bnabla_x \m_0 \boldsymbol\chi(y)$.  Note that $\m_1$
is biggest term in $\m^\varepsilon - \m_0$ and therefore determines
the right figure in \cref{fig:num_ex}: there we can observe
oscillations around zero on a scale of approximately $\varepsilon$ smaller then
the variations in the homogenized solution.  For short times, we clearly
observe oscillations in both time and space while the oscillations
in time reduce as $t$ increases. This indicates that the $\v$-part
of $\m_1$ gets damped away with time, while
$\bnabla_x \m_0 \boldsymbol\chi(y)$, which does not depend on
$t/\varepsilon^2$ but oscillates in space, is preserved. This matches
with the results for $\v$ and $\m_1$ in
\Cref{sec:correction_estimates}.

On the $\varepsilon^0$-scale, we have
\begin{equation}\label{eq:eps0_total}
   \partial_\tau \m_2 =\LL \m_2 + \F_2,
\end{equation}
where the expression for $\F_2$ given by \cref{eq:forces} becomes
\begin{align}\label{eq:F2}
\F_2 =& - \R_1 -\m_0 \times \Z_1
- \alpha [\m_0 \times \R_1+\m_0 \times \m_0 \times \Z_1
+  \S_1]-\partial_t\m_0\,,
\end{align}
and the relation \cref{eq:L2_m1_L1_m0_L2_v} gives the simplification
\begin{equation}
  \label{eq:R1_S1}
  \R_1 =\m_1 \times \L_2 \v, \qquad \S_1 = \m_1 \times\m_0\times \L_2 \v.
\end{equation}
To obtain a homogenized equation, \cref{eq:eps0_total} is averaged over one period $Y$ in $y$.
Then all terms which are derivatives with respect
to $y$ of $y$-periodic terms cancel, and
since  $\m_0$ does not depend on $y$ we get
\begin{equation}\label{eq:m2averaged1}
   \partial_\tau \int_Y \m_2dy =\int_Y \F_2dy
   =
   -\partial_t\m_0 -\m_0 \times \int_Y \Z_1 dy
- \alpha \m_0 \times \m_0 \times \int_Y\Z_1 dy-\E_1,
\end{equation}
where
\begin{equation}
  \label{eq:E1}
 \E_1 :=\int_Y \R_1 +\alpha [\m_0 \times \R_1+  \S_1]dy.
\end{equation}
Furthermore,
\begin{align*}
 \int_Y \Z_1 dy
&= \int_Y \L_0\m_0 + \L_1\m_1 dy =
\int_Y   \bnabla_x \cdot a(y) \bnabla_x  \m_0+\bnabla_x \cdot \left( a(y) \bnabla_y
(\nabla_x \m_0 \boldsymbol\chi + \v)\right) dy\nonumber \\
&= \int_Y \bnabla_x \cdot \left(a(y) \bnabla_x \m_0 (\I+\bnabla_y \boldsymbol\chi\right) dy
+\int_Y \L_1\v dy.
\end{align*}
We therefore define the constant homogenized material coefficient matrix $\A^H \in \Real^{n \times n}$ as
\begin{align*}
  \A^H := \int_Y a(y) \left(\I +  \bnabla_y \boldsymbol\chi \right) dy\,
\end{align*}
and let $\bar L u := \nabla_x \cdot (\A^H \nabla_x u )$ for any
scalar function $u: \Real^n \times \Real \mapsto \Real$, with the
corresponding vector-operator being denoted $\bar \L$.  Plugging
this into \cref{eq:m2averaged1}, we get
\begin{align}\label{eq:AF2}
  \partial_\tau \int \m_2 dy =&  -\partial_t \m_0 -\m_0 \times \bar \L \m_0 - \alpha \m_0 \times \m_0 \times \bar \L \m_0 - \E_1 -\E_2,
\end{align}
where
\begin{equation}
  \label{eq:E2}
 \E_2 :=
 \m_0 \times \int_Y \L_1\v dy + \alpha \m_0 \times \m_0 \times \int_Y \L_1\v dy.
\end{equation}
As we will see in \Cref{sec:correction_estimates}, $\v$ oscillates and decays exponentially in $\tau$,
which means that so do $\R_1$ and $\S_1$ by \cref{eq:R1_S1}.
Therefore, if we average over a fixed interval in the fast time variable, the contributions of $\E_1$ and $\E_2$ will become negligible
as the interval size increases, while $\m_0$ is unaffected.
We therefore define $\m_0$ such that it satisfies
\begin{align}\label{eq:m0_eq}
  \partial_t \m_0 =&  -\m_0 \times \bar \L \m_0 - \alpha \m_0 \times \m_0 \times \bar \L \m_0\,.
\end{align}
In contrast to the differential equations for $\m_j$, $j \ge 1$, this is a nonlinear
differential equation with a matrix-valued coefficient in the operator $\bar \L$.

\section{Sobolev norm estimates}\label{sec:utility}

The proofs in the following sections rely frequently on properties
of the considered Bochner-Sobolev and multiscale norms as well as
several bilinear Sobolev estimates. In this section, we therefore
prove lemmas providing the required properties, making it possible
to keep the subsequent sections mostly focused on specific estimates
for the solution to the Landau-Lifshitz equation \cref{eq:main_prob}, corresponding
homogenized solution and correctors.

If not stated otherwise, the estimates in this section apply to
functions in arbitrary dimensions, not necessarily on $\Omega$ as
considered previously. All the lemmas that are stated for scalar
functions analogously apply to vector valued functions, with either
scalar or cross products instead of products of scalar functions.
Throughout this section, we suppose $0 < \varepsilon < 1$ in accordance with (A3).

In several of the subsequent estimates we use the Sobolev
inequality which states that when $f\in H^2(D)$ and
$D \subseteq \Real^n$, for dimension $n \le 3$, then
\begin{align}
  \label{eq:sob}
\sup_{x\in D} |f(x)| \le C \|f\|_{H^2(D)}\,.
\end{align}


\subsection{Multiscale norms.}
In the present paper,
a function $u(x,y)$
in
the Bochner-Sobolev space
$H^{q,p}(\Omega;Y)$ is often used to describe multiscale
phenomena, where
the $x$- and $y$-variables represent the slow and fast scales respectively.
For such functions we have the following lemma.
\begin{lemma}\label{lemma:multiscaleest}
Suppose $f^\varepsilon(x) := u(x,x/\varepsilon)$ and $n \le 3$. Then
\begin{align}
   \|f^\varepsilon\|_{H^q} \leq \frac{C}{\varepsilon^{q}}\|u\|_{H^{q, q+2}}\,,\qquad
   \|f^\varepsilon\|_{W^{q,\infty}} \leq \frac{C}{\varepsilon^{q}}\|u\|_{H^{q+2, q+2}}\,,
\end{align}
whenever the norms are bounded.
\end{lemma}
\begin{proof}
  Using \cref{eq:sob} and the definition of the norms, we find that
\begin{align*}
   \|f^\varepsilon\|^2_{H^q}
   &\le
   \sum_{\substack{|\alpha| \le q \\ \gamma \le \alpha}}
\binom{\alpha}{\gamma}
  \int \varepsilon^{-2|\gamma|}|\partial_y^\gamma\partial_x^{\alpha-\gamma} u(x,x/\varepsilon)|^2 dx
  \\
&\le C\varepsilon^{-2q}
   \sum_{\substack{|\alpha| \le q \\ \gamma \le \alpha}}
\|\partial_y^\gamma\partial_x^{\alpha-\gamma} u\|^2_{H^{0, 2}}
\leq C\varepsilon^{-2q}
  \|u\|^2_{H^{q, q+2}},
\end{align*}
and accordingly,
  \begin{align*}
   \|f^\varepsilon\|^2_{W^{k,\infty}}
   &\le
 C\varepsilon^{-2q}
   \sum_{\substack{|\alpha| \le q \\ \gamma \le \alpha}}
\sup_{x,y} |\partial_y^\gamma\partial_x^{\alpha-\gamma} u(x,y)|^2
\\&\leq
 C\varepsilon^{-2q}
   \sum_{\substack{|\alpha| \le q \\ \gamma \le \alpha}}
\|\partial_y^\gamma\partial_x^{\alpha-\gamma} u\|^2_{H^{2, 2}}
\leq C\varepsilon^{-2q}
\|u\|^2_{H^{q+2, q+2}},
\end{align*}
which shows the lemma.
\end{proof}

The weighted multiscale norm $\|\cdot\|_{H^q_\varepsilon}$ has the following properties that we will use:
\begin{lemma}\label{lemma:H_eps}
  Consider $f \in H^q(\Omega)$ such that
  for $0 \le j \le q$ and some constant $c \in \Real$,
    \[\|f\|_{H^j} \le C_j \varepsilon^{c-j}\,,\]
    then it follows that
    \begin{equation}\label{eq:H_eps_bound}
      \|f\|_{H^{q}_\varepsilon} \le C  \varepsilon^c,
    \end{equation}
    where the constants $C, C_j$ are independent of $\varepsilon$. Moreover, given a multi-index $\beta$, it holds for
    $0 \le q \le r - |\beta|$ that
    \begin{equation}\label{eq:H_eps_grad}
     \|\partial^\beta f\|_{H^q_\varepsilon} \le  \varepsilon^{-|\beta|}  \|f\|_{H^{q + |\beta|}_\varepsilon}.
    \end{equation}
\end{lemma}
\begin{proof}
  The first claim, \cref{eq:H_eps_bound} holds, since by the
  definition of $\|\cdot\|_{H^q_\varepsilon}$ and the given
  assumption,
\begin{align*}
  \|g\|_{H^{q}_\varepsilon} = \sum_{j=0}^{q} \varepsilon^j \|g\|_{H^j}
\le  \varepsilon^c \sum_{j=0}^{q} C_j
\le C \varepsilon^c.
\end{align*}
Similarly, we find that
\[\varepsilon^{|\beta|} \|\partial^\beta f\|_{H^{q}_\varepsilon} = \sum_{j=0}^q \varepsilon^{j + |\beta|} \|\partial^\beta f\|_{H^j} \le \sum_{j=0}^q \varepsilon^{j + |\beta|} \|f\|_{H^{j + |\beta|}}
=  \sum_{j=|\beta|}^{q+|\beta|} \varepsilon^j \|f\|_{H^j} \le \|f\|_{H^{q+|\beta|}_\varepsilon},\]
which implies \cref{eq:H_eps_grad}.
\end{proof}


\subsection{Bilinear estimates.}

To obtain estimates for the product of two functions, the following
bilinear Sobolev estimates are useful.

\begin{lemma}\label{lemma:bilinearest}
Let
$f, g \in {C}(\Omega) \cap H^q(\Omega)$. It then holds
that
\begin{equation}
  \label{eq:interpol_ineq}
 \|(\partial^\beta f) (\partial^\gamma g)\|_{L^2} \le C (\|f\|_{L^\infty} \|g\|_{H^q}
+ \|g\|_{L^\infty} \|f\|_{H^q}) \quad \text{for } |\beta| + |\gamma| = q,
\end{equation}
and
  \begin{equation}
    \label{eq:interpol2}
   \|f g\|_{H^q} \le
   C \left(\|f\|_{L^\infty} \|g\|_{H^q} +\|f\|_{H^q} \|g\|_{L^\infty}\right).
  \end{equation}
Let $u\in H^{q_1,\infty}(\Omega;Y)$ and $v\in H^{q_2,\infty}(\Omega;Y)$
where $q_1,q_2\in\Znumbers$.
Let $q_0\leq \min(q_1,q_2)$ and $n\leq 3$. Then,
for all $p\geq 0$,
\begin{equation}\label{eq:prodest}
\|u v\|_{H^{q_0, p}} \leq C
\|u\|_{H^{q_1, p+2}} \|v\|_{H^{q_2, p}},
\end{equation}
if either
\begin{equation}\label{eq:qjcond}
   q_1+q_2\geq \min(3+q_0,5)
   \qquad\text{or}\qquad
   \text{$q_1\geq q_0+2$. }
\end{equation}
The constants $C$ are independent of $f$, $g$, $u$ and $v$.
\end{lemma}
\begin{proof}
The first two statements
\cref{eq:interpol_ineq} and
\cref{eq:interpol2}
are proved for instance in
\cite[Proposition 3.6]{taylor3} and
  \cite[Proposition 3.7]{taylor3}.

  To prove the remaining statement, let $|\alpha|+|\gamma|= q_0$ and
  $|\beta|+|\kappa|= p$. We then start by estimating the same
  quantity in two different ways. First,
  \if\longversion1
\begin{align}\label{eq:prodcase1}
\|(\partial_x^\alpha\partial_y^\beta u)
  (\partial_x^\gamma\partial_y^\kappa v)\|_{H^{0,0}}^2
 &=\int |
   \left(\partial_x^\alpha\partial_y^\beta u(x,y)\right)
\left(\partial_x^\gamma\partial_y^\kappa v(x,y)\right)|^2dxdy \nonumber\\
&\leq \sup_{(x,y)\in\Omega\times Y}|\partial_x^\alpha\partial_y^\beta u(x,y)|^2\int |\partial_x^\gamma\partial_y^\kappa v(x,y)|^2dxdy 
 \nonumber\\
&\leq C\|\partial_x^\alpha\partial_y^\beta u\|^2_{H^{2,2}} \|\partial_x^\gamma\partial_y^\kappa v\|^2_{H^{0,0}}
\leq C\|u\|^2_{H^{|\alpha|+2,p+2}}
\|v\|^2_{H^{q_0,p}}.
\end{align}
\else
\begin{align}\label{eq:prodcase1}
\|(\partial_x^\alpha\partial_y^\beta u)
  (\partial_x^\gamma\partial_y^\kappa v)\|_{H^{0,0}}^2
&\leq \sup_{(x,y)\in\Omega\times Y}|\partial_x^\alpha\partial_y^\beta u(x,y)|^2\int |\partial_x^\gamma\partial_y^\kappa v(x,y)|^2dxdy 
 \nonumber\\
&\leq C\|\partial_x^\alpha\partial_y^\beta u\|^2_{H^{2,2}} \|\partial_x^\gamma\partial_y^\kappa v\|^2_{H^{0,0}}
\leq C\|u\|^2_{H^{|\alpha|+2,p+2}}
\|v\|^2_{H^{q_0,p}}.
\end{align}
   \fi
Second,
\begin{align}\label{eq:prodcase2}
\|(\partial_x^\alpha\partial_y^\beta u)
(\partial_x^\gamma\partial_y^\kappa v)\|_{H^{0,0}}^2
&\leq \int \sup_{y\in\Omega}|\partial_x^\alpha\partial_y^\beta u(x,y)|^2 \sup_{x\in Y} |\partial_x^\gamma\partial_y^\kappa v(x,y)|^2dxdy \nonumber\\
&\leq C\int \|\partial_x^\alpha\partial_y^\beta u(x,\cdot)\|_{H^2(Y)}^2
\|\partial_x^\gamma\partial_y^\kappa v(\cdot,y)\|_{H^2(\Omega)}^2
 dxdy \nonumber\\
  &=
C \|\partial_x^\alpha\partial_y^\beta u\|^2_{H^{0,2}}\|\partial_x^\gamma\partial_y^\kappa v\|^2_{H^{2,0}}
\leq C\|u\|^2_{H^{q_0,p+2}}
\|v\|^2_{H^{|\gamma|+2,p}}.
\end{align}
We then consider the case when $q_1+q_2\geq \min(q_0+3,5)$.
Suppose $q_1\leq q_2$
and assume that $|\alpha| \le q_1 - 2$. Then it follows from \cref{eq:prodcase1}
that
\begin{equation}
  \label{eq:lemma_22b}
  \|(\partial_x^\alpha\partial_y^\beta u) (\partial_x^\gamma\partial_y^\kappa v)\|_{H^{0,0}}
  \le C \|u\|_{H^{q_1, p+2}} \|v\|_{H^{q_2, p}}\,.
\end{equation}
If, on the other hand, $|\alpha| \ge q_1 - 1$, then when $q_0\leq 2$,
\[
  q_2\geq q_0+3-q_1
  \geq q_0+2-\max(0,q_1-1)
  \geq q_0+2-|\alpha|=|\gamma|+2,
\]
while if $q_0\geq 3$,
\[
q_2\geq q_0\geq 3\geq q_1-|\alpha|+2
\geq q_0-|\alpha|+2=|\gamma|+2.
\]
By \cref{eq:prodcase2},
this shows that
\cref{eq:lemma_22b} holds also for
 for $|\alpha| \ge q_1 - 1$.
When $q_2\leq q_1$ we get the same result upon switching the cases and using
\cref{eq:prodcase1}
 for $|\gamma| \ge q_2 - 1$ and
\cref{eq:prodcase2}
for $|\gamma| \le q_2 - 2$.
Finally,
\cref{eq:lemma_22b} follows directly
from \cref{eq:prodcase1}
in the case
when $q_1\geq q_0+2$.

From
the estimates \cref{eq:lemma_22b} we finally have
\begin{align*}
  \|u v\|^2_{H^{q,p}}
&\leq
\sum_{\substack{|\alpha+\gamma| \le q \\ |\beta+\kappa| \le p}}
\binom{\alpha+\gamma}{\alpha}
\binom{\beta+\kappa}{\beta}
\|(\partial_x^\alpha\partial_y^\beta u)(\partial_x^{\gamma}\partial_y^{\kappa} v)\|^2_{H^{0,0}}\\
&\leq C
\sum_{j=0}^q
\sum_{k=0}^p
\|u\|^2_{H^{q_1,p+2}}
\|v\|^2_{H^{q_2,p}}
\leq C \|u\|^2_{H^{q_1,p+2}}
\|v\|^2_{H^{q_2,p}}.
\end{align*}
This proves the lemma.
\end{proof}

The next two results, regarding the cross product of
vector-valued functions, are consequences of
\Cref{lemma:bilinearest}.
\begin{lemma}\label{lemma:prodindex}
Suppose
$\partial^\ell_t\u_m,\partial^\ell_t\v_m\in H^{r-m-2\ell,\infty}(\Omega;Y)$
for $0\leq 2\ell\leq 2k\leq r-j$ and $0\leq m\leq j$.
Then $\partial^k_t(\u_{m}\times \v_{m'})\in H^{r-j-2k,\infty}(\Omega;Y)$
when $m+m'\leq j+2$, and for all $p\geq 0$,
\begin{align*}
\|\partial^k_t(\u_m\times \v_{m'})\|_{H^{r-j-2k, p}} \leq
C\sum_{\ell=0}^k\|\partial^{k-\ell}_t\u_m\|_{H^{r-m-2k+2\ell, p+2}}\|\partial^\ell_t\v_{m'}\|_{H^{r-m'-2\ell, p}}, 
\end{align*}
where $C$ is independent of $\u_m$ and $\v_{m'}$.
\end{lemma}
\begin{proof}
By \cref{eq:prodest} in \Cref{lemma:bilinearest}, where we choose  $q_0=r-j-2k$, $q_1=r-m-2k+2\ell$ and $q_2=r-m'-2\ell$
for $0\leq \ell \leq k$, we get
\begin{align*}
\|\partial^k_t(\u_m\times \v_{m'})\|_{H^{r-j-2k, p}} &\leq
C\sum_{\ell=0}^k\|(\partial^{k-\ell}_t\u_m)\times (\partial^\ell_t\v_{m'})\|_{H^{r-j-2k, p}}
\\
&\leq
C\sum_{\ell=0}^k\|\partial^{k-\ell}_t\u_m\|_{H^{r-m-2k+2\ell, p+2}}\|\partial^\ell_t\v_{m'}\|_{H^{r-m'-2\ell, p}}.
\end{align*}
It is indeed valid to use \Cref{lemma:bilinearest} since
$q_0=q_1-(j-m)-2\ell=q_2-(j-m')-2(k-\ell)\leq \min(q_1,q_2)$ and
\[
q_1 + q_2 = q_0 + r + j - (m+m') \ge q_0 + r - 2 \ge q_0 + 3\,,
\]
satisfying the left condition in \cref{eq:qjcond}.
The proof is complete.
\end{proof}

As a consequence of this lemma,
we get estimates for the time
derivatives of precession and damping term in
the Landau-Lifshitz equation by
taking $\m_0$, the solution to the
homogenized equation \cref{eq:main_hom}, as one of the functions
in \Cref{lemma:prodindex}.

\begin{corollary}\label{lemma:m0bound}
Suppose that $\m_0$ satisfies (A5). For $0\leq 2\ell\leq 2k \leq q\leq r$ and  $\partial_t^\ell\f(\cdot,\cdot,t)\in H^{q-2\ell,p}$
when $0\leq t\leq T$, we have for all $p\geq 0$ and $0 \le t \le T$
\begin{align*}
\|\partial^k_t(\m_0\times \f)\|_{H^{q-2k, p}} &\leq
C \sum_{\ell=0}^k \|\partial^\ell_t\f\|_{H^{q-2\ell, p}}, \\
\|\partial^k_t(\m_0\times\m_0\times \f)\|_{H^{q-2k, p}} &\leq
C \sum_{\ell=0}^k \|\partial^\ell_t\f\|_{H^{q-2\ell, p}},
\end{align*}
where $C$ is independent of $\f$ and $t$.
\end{corollary}
\begin{proof}
The first inequality is obtained by
taking $\u_m=\m_0$, $\v_{m'}=\f$,
$q=r-j$, $m=0$ and $m'=r-q$ in \Cref{lemma:prodindex},
which is a valid choice due to (A5).
The triple product case then follows since
\begin{align*}
\|\partial^k_t(\m_0\times \m_0\times \f)\|_{H^{q-2k, p}} &\leq
C\sum_{\ell=0}^k\|\partial^{k-\ell}_t\m_0\|_{H^{r-2k+2\ell, p}}\|\partial^\ell_t(\m_0\times\f)\|_{H^{q-2\ell, p}}
\\ &\leq
C\sum_{\ell=0}^k\|\partial^\ell_t(\m_0\times\f)\|_{H^{q-2\ell, p}}.
\end{align*}
\end{proof}

Finally, we consider the product of two functions with a maximum norm
bound given for one of them. Then the following bilinear estimate
holds.
\begin{lemma}  \label{lemma:prod_infty_bound}
  Suppose $f \in H^{q}(\Omega)$ and $g \in W^{q, \infty}(\Omega)$. Then
  \begin{equation}\label{eq:prod_infty_bound}
    \|f g\|_{H^q} \le C \sum_{j=0}^q \|g\|_{W^{j, \infty}} \|f\|_{H^{q-j}}, \qquad
    \|f g\|_{H^q_\varepsilon} \le C \sum_{j=0}^q \varepsilon^j \|g\|_{W^{j, \infty}} \|f\|_{H^{q-j}_\varepsilon}.
  \end{equation}
   In particular,  consider
    $h \in C^\infty(Y)$ and let
    $h^\varepsilon = h(x/\varepsilon)$, then it holds for $0 \le j \le q$  that
  \begin{equation}\label{eq:hepsf_prod}
    \| h^\varepsilon f\|_{H^j} \le C \frac{1}{\varepsilon^{j}} \|h\|_{W^{j, \infty}} \|f\|_{H_\varepsilon^j}, \qquad
      \| h^\varepsilon f\|_{H^j_\varepsilon}
      \le C \|h\|_{W^{j, \infty}}\|f\|_{H^j_\varepsilon}.
    \end{equation}
    In all cases, the constant $C$ is independent of $\varepsilon$.
\end{lemma}
\begin{proof}
  Consider first the $\|\cdot\|_{H^q}$-norm of the product. It holds that
  \if\longversion1
  \begin{align*}
     \|f g\|_{H^q}^2 &=  \sum_{\substack{|\alpha| \le q \\ \gamma \le \alpha}}
\binom{\alpha}{\gamma}
\int |\partial^\gamma g \partial^{\alpha-\gamma} f|^2 dx \\
   &\le
 C \sum_{\substack{|\alpha| \le q \\ \gamma \le \alpha}}
\sup |\partial^\gamma g|^2 \int |\partial^{\alpha-\gamma} f|^2 dx 
\le C\sum_{j=0}^q \|g\|_{W^{j, \infty}}^2 \|f\|_{H^{q-j}}^2,
  \end{align*}
  \else
  \begin{align*}
     \|f g\|_{H^q}^2
   &\le
 C \sum_{\substack{|\alpha| \le q \\ \gamma \le \alpha}}
\sup |\partial^\gamma g|^2 \int |\partial^{\alpha-\gamma} f|^2 dx 
\le C\sum_{j=0}^q \|g\|_{W^{j, \infty}}^2 \|f\|_{H^{q-j}}^2,
  \end{align*}
  \fi
which shows the first statement. Consequently, we find
\begin{align*}
  \|f g\|_{H^q_\varepsilon}
 &= \sum_{j=0}^q \varepsilon^j \|f g\|_{H^j}
\le C \sum_{j=0}^q \sum_{i=0}^j \varepsilon^j \|g\|_{W^{i, \infty}} \|f\|_{H^{j-i}}
\le C \sum_{j=0}^q \sum_{i=0}^j \varepsilon^i \|g\|_{W^{i, \infty}} \varepsilon^{j-i}\|f\|_{H^{j-i}} \\
&= C \sum_{i=0}^q \sum_{j=0}^{q-i} \varepsilon^i \|g\|_{W^{i, \infty}} \varepsilon^{j}\|f\|_{H^{j}}
= C \sum_{i=0}^q \varepsilon^i  \|g\|_{W^{i, \infty}} \|f\|_{H^{q-i}_\varepsilon}.
\end{align*}
When given $h\in C^\infty(Y)$,
\if\longversion1
we one can bound
\[\|h^\varepsilon\|_{W^{k, \infty}} \le \frac{\|h\|_{W^{k, \infty}}}{\varepsilon^k}, \qquad k \ge 0, \]
hence
\fi
the $\|\cdot\|_{H^j}$-estimate in \cref{eq:hepsf_prod}
follows from the $\|\cdot\|_{H^q}$-estimate in \cref{eq:prod_infty_bound},
\[\|h^\varepsilon f\|_{H^q} \le C \sum_{j=0}^q \|h^\varepsilon \|_{W^{q-j, \infty}} \|f\|_{H^j}
  \le C \sum_{j=0}^q \frac{\|h\|_{W^{q-j, \infty}}}{\varepsilon^{q-j}} \| f\|_{H^j} = \frac{C}{\varepsilon^q} \|h\|_{W^{q, \infty}} \| f\|_{H^q_\varepsilon}.\]
The $\|\cdot\|_{H^q_\varepsilon}$-estimate then is a direct consequence of \cref{eq:H_eps_bound}.
\end{proof}

\subsection{Norms involving the linear operator $L$.}
Consider now $a^\varepsilon(x) = a(x/\varepsilon)$ such that (A1) holds 
and let
$L = \nabla \cdot (a^\varepsilon \nabla)$, which is the setup we
consider in the rest of this paper.
We then show two results, allowing us to switch between
$H^q_\varepsilon$-norms and $L^2$-norms involving $L$. First we can
estimate $L^p u$ in terms of $\nabla u$.
 \begin{lemma}\label{lemma:nonlinear_utility1}
    Suppose $u \in H^{r}(\Omega)$ and $a \in C^\infty(\Omega)$. Then
    it holds for $2 \le 2k \le r-1-\ell$ and $0 \le q \le r - 2k$
     \begin{align}
      \|L^k u\|_{H^q} &\le C \frac{1}{\varepsilon^{q+2k-1}} \|\nabla u\|_{H^{q+2k-1}_\varepsilon},
     \end{align}
     where the constant $C$ is independent of $\varepsilon$.
  \end{lemma}
  \begin{proof}
Let $\beta$ be a multi-index with
    $|\beta| \le 2k$. Since $a \in C^\infty(Y)$, there exist
    functions $c_\beta(y) \in C^\infty(Y)$, which are either zero or consist of a
    product of $\partial^\gamma a(y), |\gamma| \le |\beta|$, such
    that
    \begin{align*}
      L^k u &= \sum_{1 \le |\beta| \le 2k} \frac{1}{\varepsilon^{2k-|\beta|}} c_\beta^\varepsilon \partial^\beta u ,
    \end{align*}
    where $c^\varepsilon_\beta = c_\beta(x/\varepsilon)$. It thus
    follows by \cref{eq:hepsf_prod} and \cref{eq:H_eps_grad} in
    \Cref{lemma:H_eps} that
    \if\longversion1
    \begin{align*}
      \|L^k u \|_{H^q} &\le C \sum_{1 \le |\beta| \le 2k} \frac{1}{\varepsilon^{q + 2k-|\beta|}} \left\| \partial^\beta u \right\|_{H^q_\varepsilon}
                         \le C \frac{1}{\varepsilon^{q + 2k - 1}} \sum_{0 \le |\nu| \le 2k-1} \varepsilon^{|\nu|} \|\partial^{\nu} \nabla u\|_{H^q_\varepsilon} \\
&\le C \frac{1}{\varepsilon^{q + 2k - 1}}  \|\nabla u\|_{H^{q+2k-1}_\varepsilon} .
    \end{align*}
    \else
    \begin{align*}
      \|L^k u \|_{H^q} &\le C \sum_{1 \le |\beta| \le 2k} \frac{1}{\varepsilon^{q + 2k-|\beta|}} \left\| \partial^\beta u \right\|_{H^q_\varepsilon}
                         \le C \frac{1}{\varepsilon^{q + 2k - 1}} \sum_{0 \le |\nu| \le 2k-1} \varepsilon^{|\nu|} \|\partial^{\nu} \nabla u\|_{H^q_\varepsilon},
    \end{align*}
    which shows the estimate in the lemma.
    \fi
\end{proof}

Second, we have the following multiscale version of
elliptic regularity.
\if\longversion1
(Note that standard elliptic
regularity estimates have constants that depend on
$\varepsilon$.)
\fi
 \begin{lemma}\label{lemma:ms_elliptic_regularity}
Suppose $u\in H^q(Y)$ with $q\geq 2$ and $0<\varepsilon\leq 1$.
Then
 \begin{equation}\label{eq:elliptic_reg}
  \|u\|_{H^{q}} \leq C\left(
    \|u\|_{L^2} +
  \frac{1}{\varepsilon^{q-1}}\|\nabla u\|_{H^{q-2}_\varepsilon}
+
  \begin{cases}
  \|{L}^p u\|_{L^2}
& q=2p,\\
  \|{L}^p u\|_{H^1}
& q=2p+1,
\end{cases}
\quad \right).
\end{equation}
Moreover, let
$\ell \in \{0, 1\}$, then
it holds for $0 \le 2k \le q-1-\ell$ that
\begin{align}\label{eq:nonlinear_utility4}
  \|u\|_{H_\varepsilon^{2k + 1 + \ell}} &\le C
        \begin{cases}
          \varepsilon^{2k+1} \|\sqrt{a^{\varepsilon}} \nabla L^k u\|_{L^2} + \|u\|_{H^{2k}_\varepsilon}, & \ell = 0, \\
          \varepsilon^{2k+2} \|L^{k+1} u\|_{L^2} + \|u\|_{H^{2k+1}_\varepsilon}, & \ell = 1,
        \end{cases}
\end{align}
where the constant $C$ is independent of $\varepsilon$.
\end{lemma}
\begin{proof}
  To show \cref{eq:elliptic_reg} we first prove that
  given a multi-index $\sigma$ with $2 \le |\sigma|\le q$,
 \begin{equation}\label{eq:elliptic_regsimple}
   \|\partial^\sigma u\|_{L^2} \leq C \left(
     \frac{1}{\varepsilon^{|\sigma|-1}}\|\nabla u\|_{H^{|\sigma|-2}_\varepsilon} +
  \begin{cases}
  \|{L}^p u\|_{L^2}
& |\sigma|=2p,\\
  \|{L}^p u\|_{H^1}
& |\sigma|=2p+1,
\end{cases}
\right).
\end{equation}
We start by proving this for
$p=1$ and $|\sigma|=2$. Then we have, with $u_k:=\partial_{x_k}u$,
       \begin{align*}
     \|D^2u\|_{L^2}^2
     &=:
     \sum_{|\sigma|=2}
     \|\partial^\sigma u\|_{L^2}^2=
     \sum_{k=1}^n \int_\Omega
          |\nabla u_k|^2 dx
     \leq C
     \sum_{k=1}^n \int_\Omega
          a^\varepsilon|\nabla u_k|^2 dx
     =
     -\sum_{k=1}^n \int_\Omega
           u_k L u_kdx\\
     &=
       \int_\Omega L u
       \sum_{k=1}^n \partial^2_{x_k}u dx
       -\sum_{k=1}^n \int_\Omega u_k [L u_k-\partial_{x_k}(L u)]
       dx
         \\
     &=
       \int_\Omega L u\, \Delta udx
       +\sum_{k=1}^n \int_\Omega (\nabla u_k) \cdot [a^\varepsilon\nabla u_k-\partial_{x_k}a^\varepsilon\nabla u]
       dx\\
     &=
       \int_\Omega L u\, \Delta udx
       -\sum_{k=1}^n \int_\Omega  \partial_{x_k} a^\varepsilon \nabla u_k
       \cdot \nabla u
       dx.
       \end{align*}
 Application of Cauchy-Schwarz and Young's inequality with a constant hence yields
       \begin{align*}
       \|D^2u\|_{L^2}^2  
           &\leq
           \frac{\gamma}{2}\|\Delta u\|_{L^2}^2
          +\frac{1}{2\gamma} \|Lu\|^2_{L^2}+
\sum_{k=1}^n  \frac{\gamma}{2}\|\nabla u_k\|^2_{L^2}
+         \sum_{k=1}^n\frac{1}{2\gamma} \frac{1}{\varepsilon^2} \|a\|^2_{W^{1,\infty}}
\|\nabla u\|^2_{L^2} \\
           &\leq
           {\gamma}\|D^2u\|_{L^2}^2
          +\frac{1}{2\gamma} \|Lu\|^2_{L^2}
+   \frac{n}{2\gamma\varepsilon^2} \|a\|^2_{W^{1,\infty}}
\|\nabla u\|^2_{L^2},
       \end{align*}
       for any constant $\gamma > 0$.
Thus, by taking $\gamma$ small enough we get
$$
          \|D^2u\|^2_{L^2}\leq C\left(
\|Lu\|^2_{L^2}
+   \frac{1}{\varepsilon^2}\|\nabla u\|^2_{L^2}
  \right),
$$
from which \cref{eq:elliptic_regsimple} for $|\sigma|=2$ follows since $\varepsilon \leq 1$.
Next, we assume that
\cref{eq:elliptic_regsimple} holds for $2\leq |\sigma|\leq 2p$.
Given another multi-index $\alpha$, we then obtain upon applying
\cref{eq:elliptic_regsimple} for $|\sigma|=2p$ and \Cref{lemma:H_eps}, that
    \begin{align*}
  \|\partial^{\sigma+\alpha} u\|_{L^2} &\leq
  C \left(\|L^p\partial^{\alpha}u\|_{L^2}
  +
  \frac{1}{\varepsilon^{2p-1}}\|\partial^{\alpha}\nabla u\|_{H^{2p-2}_\varepsilon}
  \right)\\
  &\leq
  C \left(\|\partial^{\alpha}L^p u\|_{L^2}+
  \|\partial^{\alpha}L^p u-
     L^p \partial^{\alpha}u\|_{L^2}
  +
  \frac{1}{\varepsilon^{2p+|\alpha|-1}}\|\nabla u\|_{H^{2p+|\alpha|-2}_\varepsilon}
  \right).
    \end{align*}
    Expressing $L^p u$ involving some smooth functions
    $c_\beta^\varepsilon(x)=c_\beta(x/\varepsilon)$, as in the proof
    of \Cref{lemma:nonlinear_utility1}, we can write
       \begin{align*}
     \partial^{\alpha}L^p u &=
     \sum_{\substack{1 \le |\beta| \le 2p \\ 0\leq \gamma \leq  \alpha}}
\binom{\alpha}{\gamma}
\frac{1}{\varepsilon^{2p-|\beta|+|\gamma|}}
(\partial^{\gamma}c_\beta^\varepsilon) \partial^{\beta+\alpha-\gamma} u.
       \end{align*}
Therefore, it holds that
       \begin{align*}
         \|\partial^{\alpha}L^p u-
     L^p \partial^{\alpha}u\|_{L^2}
     &\leq
          C\sum_{\substack{1 \le |\beta| \le 2p \\ 1\leq |\gamma| \leq  |\alpha|}}
\frac{1}{\varepsilon^{2p-|\beta|+|\gamma|}}
\|\partial^{\beta+\alpha-\gamma} u\|_{L^2}
 \leq
\frac{C}{\varepsilon^{2p+|\alpha|-1}}
\|\nabla u \|_{H^{2p+|\alpha|-2}_{\varepsilon}},
\end{align*}
       and thus we have in total
       \[\|\partial^{\sigma + \alpha} u\|_{L^2} \le C
         \left( \|\partial^\alpha L^p u\|_{L^2} +  \frac{1}{\varepsilon^{2p+|\alpha|-1}}\|\nabla u\|_{H^{2p+|\alpha|-2}_\varepsilon}
  \right).\]
When $|\alpha|=1$ we then get \cref{eq:elliptic_regsimple} with
$|\sigma|=2p+1$  by noting that
\if\longversion1
\begin{align*}
  \|  \partial^{\alpha}L^p u\|_{L^2}  &\leq    C\|L^{p}u\|_{H^1}.
\end{align*}
\else
$  \|  \partial^{\alpha}L^p u\|_{L^2}  \leq    C\|L^{p}u\|_{H^1}.$.
\fi
On the other hand, when $|\alpha|=2$, we get with one
more application of \cref{eq:elliptic_regsimple} and
\Cref{lemma:nonlinear_utility1},
\begin{align*}
\|
\partial^{\alpha}L^p u\|_{L^2}
&\leq
     C \left(\|L^{p+1}u\|_{L^2}
  +
  \frac{1}{\varepsilon}\|\nabla L^p u\|_{L^{2}}
  \right)
\leq      C \left(\|L^{p+1}u\|_{L^2}
  +
  \frac{1}{\varepsilon^{2p+1}}\|\nabla u\|_{H^{2p}_\varepsilon}
  \right).
\end{align*}
This completes the induction step and proves
\cref{eq:elliptic_regsimple}.
To finally prove \cref{{eq:elliptic_reg}} we use
\cref{eq:elliptic_regsimple} together with \Cref{lemma:nonlinear_utility1}, and note that
for $2\leq |\sigma|\leq q-1$,
\begin{align*}
  \|\partial^\sigma u\|_{L^2} &\leq C \left(
  \frac{1}{\varepsilon^{q-2}}\|\nabla u\|_{H^{q-3}_\varepsilon} +
  \begin{cases}
  \frac{1}{\varepsilon^{2p-1}}\|\nabla u\|_{H^{2p-1}_\varepsilon},
& |\sigma|=2p,\\
  \frac{1}{\varepsilon^{2p}}\|\nabla u\|_{H^{2p}_\varepsilon},
& |\sigma|=2p+1,
  \end{cases}
  \
  \right)
  \leq
  \frac{C}{\varepsilon^{q-2}}\|\nabla u\|_{H^{q-2}_\varepsilon},
\end{align*}
which clearly also holds for $|\sigma|=1$. Hence,
 \begin{equation}
  \|u\|_{H^{q}} \leq C
  \sum_{|\sigma|=0}^{q}
    \|\partial^\sigma u\|_{L^2}
    \leq C
    \left(\|u\|_{L^2}
    +
  \frac{1}{\varepsilon^{q-2}}\|\nabla u\|_{H^{q-2}_\varepsilon}
+
  \sum_{|\sigma|=q}
    \|\partial^\sigma u\|_{L^2}
\right)
\end{equation}
which together with
\cref{eq:elliptic_regsimple}
gives
\cref{eq:elliptic_reg}.

  To finally prove \cref{eq:nonlinear_utility4}, we consider odd and
  even indices in the sum in $\|\cdot\|_{H^{2k+1+\ell}_\varepsilon}$
  separately and use elliptic regularity as given by
  \cref{eq:elliptic_reg}, which results in
  \begin{align*}
    \|u\|_{H_\varepsilon^{2k+1+\ell}} &= \sum_{j=0}^{2k+1+\ell} \varepsilon^{j} \|u\|_{H^j}
    = \sum_{j=0}^{k+\ell} \varepsilon^{2j} \|u\|_{H^{2j}}
    + \sum_{j=0}^{k} \varepsilon^{2j+1} \|u\|_{H^{2j+1}} \\
    &\le C \left(\sum_{j=0}^{k+\ell} \varepsilon^{2j} \|L^j u\|_{L^2}
      + \sum_{j=0}^{k} \varepsilon^{2j+1} \|L^j u\|_{H^1}
      + \varepsilon \|\nabla u\|_{H^{2(k+\ell) - 2}_\varepsilon}
      + \varepsilon\|\nabla u\|_{H^{2k-1}_\varepsilon}\right)\\
    &\le  C \left( \sum_{j=0}^{k+\ell} \varepsilon^{2j}  \|L^j u\|_{L^2}
  + \sum_{j=0}^{k} \varepsilon^{2j+1} \|\nabla L^j u\|_{L^2} + \varepsilon \|\nabla u\|_{H^{2k -1 + \ell}_\varepsilon}
  \right).
  \end{align*}
  Application of
  \Cref{lemma:nonlinear_utility1} to all but the highest order terms
  in each sum together with \Cref{lemma:H_eps} then yields
\begin{align*}
\|u\|_{H_\varepsilon^{2k + 1 + \ell}}
    &\le  C \left(\varepsilon^{2(k+\ell)} \|L^{k+\ell} u\|_{L^2} + \varepsilon^{2k+1} \|\nabla L^k u\|_{L^2}
      + \varepsilon \|\nabla u\|_{H^{2k+\ell-1}_\varepsilon}\right) \\
      &\le  C
        \begin{cases}
          \varepsilon^{2k+1} \|\nabla L^k u\|_{L^2} + \|u\|_{H^{2k}_\varepsilon}, & \ell = 0, \\
          \varepsilon^{2k+2} \|L^{k+1} u\|_{L^2} + \|u\|_{H^{2k+1}_\varepsilon}, & \ell = 1.
        \end{cases}
     \end{align*}
  Using the fact that
  $a_\mathrm{min} \le a \le a_\mathrm{max}$ we then obtain the
  result in the lemma.
\end{proof}

\subsection{Application of $\L$ to a cross product.}
The next lemma is based on ideas from \cite{melcher} but has
to be significantly adapted for the problem considered here. We consider $\varepsilon$-dependent functions $\u$ and $\f$, where we assume
that $\f \in W^{q + 2k-1, \infty}(\Omega)$ such that its
$\|\cdot\|_{W^{j}}$ norm is bounded in terms of $\varepsilon$.  We
show that when applying the operator $\L^k$ to the cross product of
either $\u$ or $\f$ and $\L \u$, one can factor out the
highest order term and obtains a remainder term that is bounded in
terms of the $\|\cdot\|_{H^{q +2k}_\varepsilon}$-norm of the
gradient of $\u$. Again we assume that (A1) is true.

\begin{lemma}\label{lemma:nonlinear_utility3}
  Given $k\ge 0, q \ge 0$, suppose $\u \in H^{q + 2k + 1}(\Omega)$ and
  $\f \in W^{q + 2k-1, \infty}(\Omega)$ such that
  \begin{equation}
    \label{eq:m_tilde_bound_ass1}
  \|\bnabla \u\|_{L^\infty} \le M, \qquad  \|\f \|_{W^{j, \infty}} \le  \tilde M \left(1 + \varepsilon^{1-j}\right), \qquad\ 0 \le j \le q+2k-1 \,,
\end{equation}
for constants $M$ and $\tilde M$ independent of $\varepsilon$.
  Then it holds for
  $\w \in \{\u, \f\}$ that
  \begin{equation*}
    \L^k (\w \times \L \u) = \w \times \L^{k+1} \u + \mathbf{R}_{k, \w}\,,
\qquad \text{where} \qquad
\|\R_{k, \w}\|_{H^q} \le  C \frac{1}{\varepsilon^{q + 2k}} \|\bnabla \u\|_{H^{q + 2k}_\varepsilon}\,,
  \end{equation*}
   for a constant $C$ independent of $\varepsilon$.
\end{lemma}

\begin{proof}
  When $k = 0$, the claim in the lemma is trivially true with $\R_{0, \w} = 0$. Let
  $\w$ be either $\u$ or $\f$, then it holds for $k > 0$ that
  \[
  \begin{split}
\L^k (\w \times \L \u)
& = \L^{k-1} (\w \times \L^2 \u + \L \w \times \L \u +  2 a \sum_{j = 1}^n \partial_{x_j} \w \times \partial_{x_j} \L \u) \\
&= \w \times \L^{k+1} \u + \sum_{\ell=1}^{k} \L^{k-\ell}\left(\L \w \times \L^{\ell} \u + 2 a \sum_{j = 1}^n \partial_{x_j} \w \times \partial_{x_j} \L^{\ell} \u \right),
\end{split}
\]
which implies that $\R_{k, \w}$ in the lemma 
is given by
\begin{align*}
  \mathbf{R}_k
  =: \sum_{\ell=1}^{k} \L^{k-\ell} \mathbf{r}_\ell(\w)
   \quad \text{and} \quad
   \mathbf{r}_\ell(\w) := \L \w \times \L^{\ell} \u + 2 a \sum_{j = 1}^n \partial_{x_j} \w \times \partial_{x_j} \L^{\ell} \u
\end{align*}
In the following, we obtain bounds for $\|\R_{k, \w}\|_{H^q}$, first
for $\w=\u$ and then later for $\w = \f$. For the first
estimate, we use the fact that according to assumption (A1), there
exist functions $c_{\beta, \gamma}(y) \in C^\infty(\Omega)$, similar
to the ones in the proof of \Cref{lemma:nonlinear_utility1}, which
might also be zero, such that
\begin{align*}
  \L \u \times \L^{\ell} \u + 2 a \sum_{j = 1}^n \partial_{x_j} \u \times \partial_{x_j} \L^{\ell} \u
  =  \frac{1}{\varepsilon^{2\ell}} \sum_{\substack{1 \le |\beta|, 1 \le |\gamma| \\ |\beta + \gamma| \le 2 + 2\ell}}
  c_{\beta, \gamma}\left(\frac{x}{\varepsilon}\right)
  \varepsilon^{|\beta| + |\gamma|-2}
  \left(
    \partial^{\gamma}  \u \times \partial^\beta \u \right).
\end{align*}
Furthermore, it is a consequence of the interpolation inequality
\cref{eq:interpol_ineq} that given
multiindices $\beta$ and $\gamma$ with $|\beta| \ge 1$,
$|\gamma| \ge 1$,
\[\|\partial^\gamma \u \times \partial^\beta \u\|_{H^j} \le C \|\bnabla \u\|_{L^\infty} \|\bnabla \u\|_{H^{j + |\beta| + |\gamma| - 2}}, \qquad 0 \le j \le q\,,\]
wherefore we find proceeding as in the proof of \Cref{lemma:H_eps} that
\begin{align}\label{eq:interpol_e}
  \varepsilon^{|\beta| + |\gamma| - 2} \|\partial^\gamma \u \times \partial^\beta  \u\|_{H^q_\varepsilon}
&\le C \|\bnabla \u\|_{L^\infty} \sum_{j=0}^q \varepsilon^{j + |\beta| + |\gamma| - 2} \|\bnabla \u\|_{H^{j + |\beta| + |\gamma| - 2}} \nonumber \\
&\le C \|\bnabla \u\|_{L^\infty} \|\bnabla \u\|_{H^{q + |\beta| + |\gamma| - 2}_\varepsilon}.
\end{align}
Therefore, it follows by \cref{eq:hepsf_prod} in \Cref{lemma:prod_infty_bound}
  and \cref{eq:interpol_e} that
\begin{align*}
  \left \|\mathbf{r}_\ell(\u)\right\|_{H^q}
  &\le \frac{C}{\varepsilon^{q + 2 \ell}}
 \sum_{\substack{1 \le |\beta|, 1 \le |\gamma| \\ |\beta + \gamma| \le 2 + 2\ell}}
  \varepsilon^{|\beta| + |\gamma|-2}
  \| \partial^{\gamma} \u \times  \partial^\beta \u \|_{H^q_\varepsilon}  \le \frac{C}{\varepsilon^{q + 2 \ell}} \|\bnabla \u\|_{L^\infty} \|\bnabla \u\|_{H^{q + 2 \ell}_\varepsilon},
\end{align*}
and we obtain using \Cref{lemma:nonlinear_utility1} and \cref{eq:H_eps_bound} in \Cref{lemma:H_eps}  that
\begin{samepage}
\begin{align*}
  \left\|\L^{k-\ell} \mathbf{r}_\ell(\u)
  \right\|_{H^q}
  &\le \frac{C}{\varepsilon^{q + 2k- 2\ell}} \left\|\mathbf{r}_\ell(\u)
    \right\|_{H^{q + 2k - 2\ell}_\varepsilon} \le \frac{C}{\varepsilon^{q + 2 k}} \|\bnabla \u\|_{L^\infty} \|\bnabla \u\|_{H^{q + 2 k}_\varepsilon}.
\end{align*}
This shows that the norm of $\R_{k, \u}$ can be
bounded as stated in the lemma.
\end{samepage}

In case of $\w = \f$, the estimate is based on
\cref{eq:prod_infty_bound} in \Cref{lemma:prod_infty_bound} and the
fact that \cref{eq:m_tilde_bound_ass1} holds for $\f$. When
applying \Cref{lemma:nonlinear_utility1} and using these bounds, we find that
given $q' = q + 2k -2\ell$ and a multi-index $\gamma$ with $|\gamma| = 1$,
\begin{align*}
  \|\L \f \times \L^\ell \u\|_{H^{q'}_\varepsilon}
  &\le  C \sum_{j=0}^{q'} \varepsilon^j \|\L \f\|_{W^{j, \infty}} \|\L^\ell \u\|_{H^{q' - j}_\varepsilon} \le \frac{C}{\varepsilon} \|\L^\ell \u\|_{H^{q'}_\varepsilon} \le C \varepsilon^{-2 \ell } \|\bnabla \u\|_{H^{q' + 2\ell -1}_\varepsilon}, \\
   \|\partial^\gamma \f \times \partial^\gamma \L^{\ell} \u \|_{H^{q'}_\varepsilon}
  &\le C \sum_{j=0}^{q'} \varepsilon^j \|\partial^\gamma \f\|_{W^{j, \infty}} \|\partial^\gamma \L^{\ell} \u \|_{H^{q' - j}}
  \le C \varepsilon^{-2 \ell} \|\bnabla \u\|_{H_\varepsilon^{q' + 2\ell}}.
\end{align*}
Hence, it holds that
\begin{align*}
  \|\L^{k-\ell}(\L \f \times \L^{\ell} \u)\|_{H^q}
 &\le C \frac{1}{\varepsilon^{q + 2k - 2\ell}} \|\L \f \times \L^\ell \u\|_{H^{q + 2k -2\ell}_\varepsilon}
  \le C \frac{1}{\varepsilon^{q + 2k}} \|\bnabla \u\|_{H^{q + 2k -1}_\varepsilon},
\end{align*}
as well as
\begin{align*}
\left\|\L^{k-\ell} \left(a \sum_{i = 1}^n \partial_{x_{i}} \f \times \partial_{x_{i}} \L^{\ell} \u \right)\right\|_{H^q}
&\le C \frac{1}{\varepsilon^{q + 2k-2\ell}} \sum_{|\gamma| = 1} \|\partial^\gamma \f \times \partial^\gamma \L^{\ell} \u \|_{H^{q + 2k - 2\ell}_\varepsilon} \\
  &\le C \frac{1}{\varepsilon^{q + 2k}} \|\bnabla \u\|_{H^{q + 2k}_\varepsilon}.
\end{align*}
Thus, $\R_{k, \f}$ can be bounded in the same way as $\R_{k, \u}$. This completes the proof.
\end{proof}
\section{Stability estimate} \label{sec:nonlinear}

In this section, we derive a stability estimate for the error
introduced when approximating $\m^\varepsilon$ satisfying the
Landau-Lifshitz equation, \cref{eq:main_prob}, by $\tilde \m^\varepsilon$ that satisfies a perturbed version of the equation,
\begin{equation}
  \label{eq:llg_approx}
  \partial_t \tilde \m^\varepsilon = - \tilde \m^\varepsilon \times \L \tilde \m^\varepsilon -  \alpha \tilde \m^\varepsilon \times \tilde \m^\varepsilon \times \L \tilde \m^\varepsilon- \boldsymbol \eta^\varepsilon \,, \quad 0 \le t \le T^\varepsilon,
\end{equation}
where we recall that $T^\varepsilon=\varepsilon^\sigma T$ some some
$\sigma\in[0,2]$.
In particular, we suppose that the assumptions (A1)-(A4) hold and
that initially,
$\tilde \m^\varepsilon(x, 0) = \m^\varepsilon(x, 0)$. Moreover, we
assume that $\tilde \m^\varepsilon \in  C([0,T^\varepsilon]; W^{q+1, \infty}(\Omega)$ and that there is a constant $\tilde{M}$ such that
\begin{equation}
  \label{eq:m_tilde_bound_ass}
    \|\tilde \m^\varepsilon(\cdot,t)\|_{W^{k, \infty}} \le  \tilde M \left(1 + \frac{1}{\varepsilon^{k-1}}\right), \qquad\ 0 \le k \le q+1,
\end{equation}
for $0 \le t \le T^\varepsilon$, uniformly in $\varepsilon$. Note that this
assumption is chosen such that it fits with the estimates that will
be shown in \Cref{sec:application}.
We can then prove the following stability estimate for the difference between $\m^\varepsilon$ and $\tilde \m^\varepsilon$.
\begin{theorem}\label{thm:error_norm}
  Assume (A1) - (A4) hold and let $q \le s$ as given in (A4).
  Suppose $\tilde \m^\varepsilon \in C^1([0, T^\varepsilon];  W^{q+1, \infty}(\Omega))$ is the
  solution to \cref{eq:llg_approx} such that
  \cref{eq:m_tilde_bound_ass} holds and
  $\boldsymbol \eta^\varepsilon(\cdot, t) \in H^q(\Omega)$ for
  $0 \le t \le T^\varepsilon$.  Then there is a constant $C$ independent of $\varepsilon$ but dependent on $T$ and $a$,
  such that the error $\e := \m^\varepsilon - \tilde \m^\varepsilon$
  satisfies,
\begin{equation}\label{eq:eHq_est}
  \|\e(\cdot, t)\|^2_{H^q} \le C t  \sup_{0 \le \zeta \le t} \frac{1}{\varepsilon^{2q}}\left(\|\boldsymbol \eta^\varepsilon(\cdot, \zeta)\|_{H^q_\varepsilon}^2 + \|\nabla|\tilde \m^\varepsilon(\cdot, \zeta)|^2 \|_{H^{q}_\varepsilon }^2\right), \qquad 0 \le t \le T^\varepsilon\,.
\end{equation}
\end{theorem}
To prove this theorem, we first derive a differential equation for
$\e$. Then an estimate for $\|\e\|_{L^2}$ is shown, since the proof
in that case is somewhat different then for higher order
norms. Finally, we complete the section by using induction to show
that \cref{eq:eHq_est} holds for general $q$.  Note that these
proofs are based on ideas from \cite{melcher}.
For better readability, we drop the superscript $\varepsilon$ for
$\m$ and $\boldsymbol \eta$ in the rest of this section, keeping in
mind that they are $\varepsilon$-dependent. However, we keep the
notation $a^\varepsilon$ to stress that the constants in the
estimates depend on norms of $a$, but not $a^\varepsilon$.

To obtain a differential equation for $\e := \m - \tilde \m$,
let $\m$ and $\tilde \m$ satisfy \cref{eq:main_prob} and
\cref{eq:llg_approx}, respectively. Then $\e$ is the solution to
\begin{align}  \label{eq:dte}
  \partial_t \e &= 
\D_1 + \alpha (\L \e +  \D_2 + \D_3) + \boldsymbol \eta\,,
\end{align}
where $\D_1$ is the difference between the precession terms in \cref{eq:main_prob} and \cref{eq:llg_approx},
\begin{align}
  \D_1 := - \m \times \L \m + \tilde \m \times \L \tilde \m 
&= - \m \times \L \e - \e \times \L \tilde \m,  \label{eq:D11}
\end{align}
and $\D_2$ and $\D_3$ arise when taking the difference of the damping terms,
\begin{align*}
- \m \times \m \times \L \m + \tilde \m \times \tilde \m \times \L \tilde \m
 &= \L \m | \m|^2 -  \L \tilde \m |\tilde \m|^2 +  a^\varepsilon \m |\bnabla \m|^2 \\ &\qquad - a^\varepsilon \tilde \m |\bnabla \tilde \m|^2 +
\bnabla \cdot (a^\varepsilon \tilde \m \cdot \bnabla \tilde \m) \tilde \m
= \L \e +  \D_2 + \D_3
\,,
\end{align*}
where
\begin{subequations}
  \begin{align}
   \D_2 :&= (\e \cdot (\m + \tilde \m)) \L \tilde \m
 + a^\varepsilon \e |\bnabla \m|^2 + a^\varepsilon \tilde \m (\bnabla \e: \bnabla(\m + \tilde \m)), \label{eq:D21}\\
\D_3 :&= \bnabla \cdot (a^\varepsilon \tilde \m \cdot \bnabla \tilde \m) \tilde \m = \frac{1}{2} L |\tilde \m|^2 \tilde \m.
  \end{align}
\end{subequations}
Note that by assumption, $|\m|^2 = 1$, constant in time and space,
but $|\tilde \m|^2$ is not constant, therefore the remainder term involving
only $\tilde \m$, $\D_3$, does not vanish.

\subsection{$L^2$-estimate.}
To obtain an estimate for the change in the norm of the
error $\e$, we multiply \cref{eq:dte} by $\e$ and integrate in
space, which  yields
\begin{align*}%
\frac{1}{2} \partial_t \|\e\|^2_{L^2} &= \int_\Omega \e \cdot \partial_t \e = \int_\Omega \e \cdot \D_1 dx + \alpha \int_\Omega \e \cdot (\L\e +  \D_2 + \D_3) dx  + \int_\Omega \e \cdot \boldsymbol \eta dx
\\ &=
\int_\Omega \e \cdot \D_1 dx - \alpha \int_\Omega a^\varepsilon \bnabla \e : \bnabla \e dx + \alpha \int_\Omega \e \cdot (\D_2 + \D_3) dx  + \int_\Omega \e \cdot \boldsymbol \eta dx \,.
\end{align*}
It thus holds that
\begin{align}
   \frac{1}{2} \partial_t \|\e\|^2_{L^2} + \alpha  \|\sqrt{a^\varepsilon} \bnabla \e\|_{L^2}^2
&= \I_1 +  \alpha (\I_2 + \I_3) + \int_\Omega \e \cdot \boldsymbol \eta dx \label{eq:dte_int},
\end{align}
where we define
for the sake of notation,
\[\I_k := \int_\Omega \e \cdot \D_k dx, \qquad k = 1, 2, 3\,.\]
Our goal in the following then is to derive bounds for the integrals
$\I_k$ that only depend on the $L^2$-norms of $\e$ and
$\sqrt{a^{\varepsilon}} \bnabla \e$, multiplied by a suitable constant that we can
choose such that the terms involving $\sqrt{a^{\varepsilon}} \bnabla \e$ on the
left- and right-hand side cancel. This makes it possible to use Gr\"onwall's
inequality to obtain \cref{eq:eHq_est} for $q = 0$. 
Using the fact that the cross product of a vector by itself is zero,
$\D_1$ can be rewritten as
\begin{equation}
  \D_1 = - \tilde \m \times \L \e - \e \times \L \m =  - \bnabla \cdot  (\e \times a^\varepsilon \bnabla \m + \tilde \m \times a^\varepsilon \bnabla \e) \label{eq:D10}\,.
\end{equation}
Applying integration by parts and the
\if\longversion0
standard scalar triple product identity,
\else
identity
\cref{eq:first_tpi_gradient},
\fi
we then find that due to
orthogonality,
\begin{align*}
  \I_{1} = - \int_\Omega \e \cdot \left[ \bnabla \cdot (\e \times a^\varepsilon \bnabla \m
  + \tilde \m \times a^\varepsilon \bnabla \e)\right] dx
= \int_\Omega a^\varepsilon \bnabla \e : (\e \times \bnabla \m) dx\,.
\end{align*}
Therefore one can bound the first integral as
\begin{equation}
  \label{eq:i1_bound}
   |\I_1| \le  \|\sqrt{a^\varepsilon} \bnabla \e\|_{L^2} \|\e\|_{L^2} \|\sqrt{a^\varepsilon}\bnabla \m\|_\infty \le \frac{\gamma}{2} \|\sqrt{a^\varepsilon} \bnabla \e\|^2_{L^2} + \frac{a_\mathrm{max}M^2}{2\gamma} \|\e\|^2_{L^2}\,.
\end{equation}
For the second integral, we have according to the definition of $\D_2$, \cref{eq:D21},
\begin{align*}
  \I_2 &= \int_\Omega a^\varepsilon |\e|^2 |\bnabla \m|^2 dx
         + \int_\Omega a^\varepsilon (\e \cdot \tilde \m) (\bnabla \e : \bnabla (\m + \tilde \m)) dx
         \\&\qquad - \int_\Omega a^\varepsilon \bnabla (\e (\e \cdot (\m + \tilde \m))) : \bnabla \tilde \m  dx,
\end{align*}
where we used integration by parts on the last term. Applying Cauchy-Schwarz
\if\longversion1
to these integrals yields
\[
\begin{split}
\left|\int_\Omega a^\varepsilon (\e \cdot \tilde \m) (\bnabla \e : \bnabla (\m + \tilde \m)) dx\right|
&\le  \|\sqrt{a^\varepsilon} \bnabla \e\|_{L^2} \|\e\|_{L^2} \|\sqrt{a^\varepsilon}\tilde \m\|_{\infty} \|\bnabla (\m + \tilde \m)\|_{\infty}
\end{split}
\]
and similarly,
\begin{align*}
&\left|\int_\Omega a^\varepsilon \bnabla (\e (\e \cdot (\m + \tilde \m))) : \bnabla \tilde \m  dx\right|
\\ &\hspace{1cm}\le \sqrt{a_\mathrm{max}} \|\tilde \m\|_{L^\infty} \|\m + \tilde \m\|_{W^{1,\infty}} \left(\|\e\|_{L^2}\|\sqrt{a^\varepsilon} \bnabla \e\|_{L^2}
+ \|\e\|^2_{L^2}\right)\,.
\end{align*}
From 
\else
,
\fi
 Young's inequality with a constant together with the bounds
\cref{eq:m_tilde_bound_ass} and using assumption (A4), we thus obtain
for $t\in[0,T^\varepsilon]$,
\begin{align}\label{eq:i2_bound}
  |\I_{2}| \le \frac{\gamma}{2} \|\sqrt{a^\varepsilon} \bnabla \e\|^2_{L^2} + a_\mathrm{max} \left(\frac{1}{2\gamma} + 1\right) (M^2 \tilde M^2 + \tilde M^4) \|\e\|^2_{L^2}, \qquad \text{for all } \gamma > 0.
\end{align}
In order to derive a bound for $\I_3$, note first that
since $\m \cdot \bnabla \m = \boldsymbol 0$, it holds that
\begin{align*}
  \nabla (\e \cdot \tilde \m) &
= (\tilde \m \cdot \bnabla \e - \m \cdot \bnabla \e - \tilde \m \cdot \bnabla \tilde \m )^T
= - (\e \cdot \bnabla \e)^T - \frac{1}{2} \nabla |\tilde \m|^2\,,
\end{align*}
which implies that
\begin{align*}
  \I_3 &= \frac{1}{2} \int_\Omega (\e \cdot \tilde \m) L |\tilde \m|^2 dx
= - \frac{1}{2} \int_\Omega a^\varepsilon \nabla (\e \cdot \tilde \m) \cdot \nabla | \tilde \m|^2 dx \\
&= \frac{1}{2} \int_\Omega a^\varepsilon (\e \cdot \bnabla \e)^T \cdot \nabla | \tilde \m|^2 dx
+ \frac{1}{4} \|\sqrt{a^\varepsilon} \nabla |\tilde \m|^2\|^2_{L^2}\,.
\end{align*}
It then follows that for any $\gamma > 0$,
\begin{align}\label{eq:i3_bound}
|\I_3| \le
\frac{\gamma}{2} \|\sqrt{a^\varepsilon} \bnabla \e\|_{L^2}^2  + C a_\mathrm{max} \left(\frac{1}{2\gamma} \tilde M^2 \|\e\|^2_{L^2} +  \|\nabla |\tilde \m|^2\|_{L^2}^2 \right)\,.
\end{align}
The last integral in \cref{eq:dte_int} can be directly bounded
using Cauchy-Schwarz and Young,
\begin{equation}
  \label{eq:eta_remainder}
  \int_\Omega \e \cdot \boldsymbol \eta \,dx \le C (\|\e\|^2_{L^2} + \|\boldsymbol \eta\|_{L^2}^2)\,.
\end{equation}
Putting \cref{eq:i1_bound}, \cref{eq:i2_bound} and
\cref{eq:i3_bound} into \cref{eq:dte_int} then
yields,
upon choosing $\gamma$ sufficiently small,
\[
 \partial_t \|\e\|^2_{L^2}
\le C \left(\frac{M^2}{\gamma} \|\e\|^2_{L^2} + \|\nabla |\tilde \m|^2\|_{L^2}^2 + \|\boldsymbol \eta\|^2_{L^2}\right)\,, \qquad 0\leq t \leq T^\varepsilon\,,
\]
for some $C$ independent of $\varepsilon$ and $t$.
As $\e(0) = 0$, it follows by Gr\"onwall's inequality that
\begin{equation}
  \label{eq:eL2_est}
\|\e(\cdot, t)\|^2_{L^2} \le c e^{C (M^2/\gamma) T^\varepsilon} \int_0^t \|\boldsymbol \eta(\cdot, s)\|_{L^2}^2 + \|\nabla |\tilde \m(\cdot, s)|^2\|^2_{L^2} ds\,, \qquad 0\leq t \leq T^\varepsilon,
\end{equation}
where the prefactor can be taken independent of
$\varepsilon$
as $T^\varepsilon\leq T$.
This proves the estimate in \Cref{thm:error_norm} for $q = 0$.

\subsection{Higher-order estimates.}
In this section, we show estimates for $\|\e\|_{H^q}$, $q > 0$ to
complete the proof of \Cref{thm:error_norm}. The general structure
of these estimates is similar to the $L^2$-estimate. However,
we include an induction argument to obtain the final result.
Furthermore, bounds for the $H^q$-norms of $\D_2$ are required to
complete the proof. These are given in the
following lemma.

\begin{lemma}\label{lemma:d2_norm}
  Let $\D_2$ be given by \cref{eq:D21} and suppose that
  $\e \in H^{q+1}(\Omega)$ and that there is a constant $C$ independent of $\varepsilon$ such
  that $\|\e\|_\infty \le C$ and $\|\bnabla \e\|_\infty \le C$. Then it
  holds that
  \[\|\D_2\|_{H^q} \le
\frac{1}{\varepsilon^{q+1}} \|\e\|_{H^{q+1}_\varepsilon}
\,.\]
\end{lemma}
\begin{proof}
  First, note that we can use \cref{eq:interpol2} to bound for
  $\ell \le q$,
\begin{equation}
  \label{eq:norm_gradu_bound}
  \|\,|\e|^2\,\|_{H^{\ell}} \le
   C \|\e\|_{L^\infty}
     \|\e\|_{H^{\ell}}, \qquad \|\,|\bnabla \e|^2\,\|_{H^{\ell}} \le
   C \|\bnabla \e\|_{L^\infty}
     \|\bnabla \e\|_{H^{\ell}}.
\end{equation}
Using the fact that $\m = \e + \tilde \m$, we can moreover show that
\begin{align*}
  \D_2
  &= \L \tilde \m (|\e|^2 + 2 (\e \cdot \tilde \m))
 + a^\varepsilon \e (|\bnabla \tilde \m|^2 +  |\bnabla \e|^2)\\ & \qquad + a^\varepsilon \tilde \m (|\bnabla \e|^2
+ 2 (\bnabla \e : \bnabla \tilde \m))
 + 2 a^\varepsilon \e (\bnabla \e : \bnabla \tilde \m) , 
\end{align*}
where the last term satisfies
\[|a^\varepsilon\e (\bnabla \e : \bnabla \tilde \m)| \le C |a^\varepsilon\e |\bnabla \e|^2 + a^\varepsilon\e|\bnabla \tilde \m|^2|. \]
Thus, it holds according to \cref{eq:hepsf_prod} in  \Cref{lemma:prod_infty_bound} 
that
\begin{equation}
  \label{eq:D2_bound_part1}
\|\D_2\|_{H^q} \le \left(\|\D_{21}\|_{H^q} +  \|a^\varepsilon \D_{22}\|_{H^q}\right) 
\le  C \left(\|\D_{21}\|_{H^q} +  \frac{1}{\varepsilon^q} \|\D_{22}\|_{H^q_\varepsilon}\right), 
\end{equation}
where we let
\begin{align*}
  \D_{21} :&= \L \tilde \m |\e|^2 +  \L \tilde \m (\e \cdot \tilde \m), \\
  \D_{22} :&=  \e |\bnabla \tilde \m|^2 + \e |\bnabla \e|^2 + \tilde \m |\bnabla \e|^2 +  \tilde \m (\bnabla \e : \bnabla \tilde \m).
\end{align*}
For the norms of the terms involved in $\D_{21}$, it holds by \Cref{lemma:prod_infty_bound} and \cref{eq:norm_gradu_bound} that
\begin{align*}
  \|\L \tilde \m |\e|^2\|_{H^q}
  &\le C \sum_{j=0}^q \|\L \tilde \m\|_{W^{q-j, \infty}} \||\e|^2\|_{H^j}
    \le C \sum_{j=0}^q \|\L \tilde \m\|_{W^{q-j, \infty}} \|\e\|_\infty \|\e\|_{H^j},
\\
  \|\L \tilde \m (\e \cdot \tilde \m)\|_{H^q}
    &\le C \sum_{j=0}^q \sum_{i=0}^j \|\L \tilde \m\|_{W^{q-j, \infty}} \|\tilde \m\|_{W^{j-i, \infty}} \|\e\|_{H^i},
\end{align*}
which together with the assumption on the boundedness of
$\tilde \m$, \cref{eq:m_tilde_bound_ass}, implies that
\begin{equation}
  \label{eq:D21_bound}
  \|\D_{21}\|_{H^q}
\le C \sum_{j=0}^q \left(\frac{1}{\varepsilon^{q-j+1}} \|\e\|_{H^j} + \sum_{i=0}^{j-1}
\frac{1}{\varepsilon^{q-i}} \|\e\|_{H^i} 
\right)
 \le
 C \frac{1}{\varepsilon^{q+1}} \|\e\|_{H^q_\varepsilon}.
\end{equation}
Again using \Cref{lemma:prod_infty_bound} and
\cref{eq:norm_gradu_bound}, we can furthermore show that
the norms involved in $\D_{22}$ satisfy
\begin{align*}
  \|\e|\bnabla \tilde \m|^2\|_{H^q}
  &
  \le C \sum_{j=0}^q \sum_{i=0}^j \|\bnabla \tilde \m\|_{W^{q-j-i, \infty}}
  \|\bnabla \tilde \m\|_{W^{i, \infty}}  \|\e\|_{H^j},\\
  \|\tilde \m|\bnabla \e|^2\|_{H^q}
  &\le C   \sum_{j=0}^q \| \tilde \m\|_{W^{q-j}, \infty} \||\bnabla \e|^2\|_{H^j}
    \le  C   \sum_{j=0}^q \| \tilde \m\|_{W^{q-j, \infty}} \|\bnabla \e\|_{L^\infty}  \|\bnabla \e\|_{H^j}, \\
  \|\tilde \m (\bnabla \e:\bnabla \tilde \m)\|_{H^q}
  &
  \le  \sum_{j=0}^q \sum_{i=0}^j \| \tilde \m\|_{W^{q-j}, \infty} \|\bnabla \tilde \m\|_{W^{j-i, \infty}} \|\bnabla \e\|_{H^i},
\end{align*}
and finally, as shown in \cite{melcher}, we have as a consequence of
\cref{eq:interpol2} and the boundedness of the gradients of $\m$ and
$\tilde \m$ that
\begin{align*}
  \|\e|\bnabla \e|^2\|_{H^q} \le C\left(\|\e\|_{L^\infty} \|\bnabla \e\|_{L^\infty} \|\bnabla \e\|_{H^q} + \|\bnabla \e\|_{L^\infty}^2 \|\bnabla \e\|_{H^{q-1}} \right) \le  C \|\bnabla \e\|_{H^q}.
\end{align*}
Applying the assumption \cref{eq:m_tilde_bound_ass}, we thus get
\begin{align}
  \|\D_{22}\|_{H_j} &\le C \left(\sum_{i=0}^j \frac{1}{\varepsilon^{j-i}} \|\e\|_{H^i}
+ \sum_{i=0}^j \frac{1}{\varepsilon^{\max(0, j-i-1)}} \|\e\|_{H^{i+1}}
+ \sum_{i=0}^j \frac{1}{\varepsilon^{j-i}} \|\e\|_{H^{i+1}}
+  \|\e\|_{H^{j+1}}
\right) \nonumber \\
&\le C \left(\sum_{i=0}^j \frac{1}{\varepsilon^{j-i}} \|\e\|_{H^i} + \|\e\|_{H^{j+1}}\right)
  \le C \frac{1}{\varepsilon^{j+1}} \|\e\|_{H^{j+1}_\varepsilon}.
  \label{eq:D22_bound}
\end{align}
In total, the combination of \cref{eq:D21_bound} and
\cref{eq:D22_bound} with \cref{eq:D2_bound_part1} and application of \cref{eq:H_eps_bound} in
\Cref{lemma:H_eps} results in
\[\|\D_2\|_{H^q} \le C \left(
\frac{1}{\varepsilon^{q+1}} \|\e\|_{H^q_\varepsilon}
+ \frac{1}{\varepsilon^{q+1}} \|\e\|_{H^{q+1}_\varepsilon} 
\right)
\le
C \frac{1}{\varepsilon^{q+1}} \|\e\|_{H^{q+1}_\varepsilon}
.\]
This completes the proof.
\end{proof}

To continue with the proof of \Cref{thm:error_norm}, consider now
$\bnabla \L^{k} \e$ with $k\ge 0$.  Based on \cref{eq:dte}, we find
using integration by parts that
\begin{align}%
  \frac{1}{2} \partial_t \|\sqrt{a^\varepsilon} \bnabla \L^k \e\|_{L^2}^2
  &= \int_\Omega a^\varepsilon \bnabla \L^k \e : \bnabla \L^k \partial_t  \e dx
  = -\int_\Omega (\L^{k+1} \e)\cdot \L^k \partial_t \e dx \\
  &= -\int_\Omega (\L^{k+1}\e) \cdot \L^k \D_1 dx - \alpha \int_\Omega (\L^{k+1}\e) \cdot  \L^k (\L \e + (\D_2 + \D_3)) dx  \nonumber \\
  & \qquad + \int_\Omega (a^\varepsilon \bnabla \L^k\e) \cdot \bnabla \L^k \eta dx \nonumber.
\end{align}
Similarly, we obtain for $k > 0$ that
\begin{align*}
\frac{1}{2} \partial_t \| \L^k \e\|_{L^2}^2 
 &= - \int_\Omega  a^\varepsilon \bnabla \L^k \e : \bnabla \L^{k-1} \D_1 dx -
   \alpha \int_\Omega  (a^\varepsilon \bnabla \L^k \e) : \bnabla \L^{k-1} (\L \e + (\D_2 + \D_3) dx
  \\ &\qquad + \int_\Omega  (\L^k \e) \cdot \L^k \eta dx.
\end{align*}
It thus holds that for $k \ge 0$,
\begin{align}
 \frac{1}{2} \partial_t \|\sqrt{a^\varepsilon} \bnabla \L^k \e\|^2_{L^2} + \alpha  \|\L^{k+1} \e\|_{L^2}^2
&= - \J_{1,k} - \alpha (\J_{2,k} + \J_{3,k}) + \int_\Omega a^\varepsilon \bnabla \e \cdot \bnabla \eta dx \label{eq:dtw_int}, \\
 \frac{1}{2} \partial_t \|\L^{k+1} \e\|^2_{L^2} + \alpha  \|\sqrt{a^\varepsilon} \bnabla \L^{k+1} \e\|_{L^2}^2
&= - \mathbf{K}_{1,k} - \alpha (\mathbf{K}_{2,k} + \mathbf{K}_{3,k}) + \int_\Omega  \L \e \cdot \L \eta dx \label{eq:dtLe_int},
\end{align}
where
\[\J_{j,k} := \int_\Omega \L^{k+1} \e \cdot \L^k \D_{j} dx, \quad \mathbf{K}_{j,k} := \int_\Omega a^\varepsilon \bnabla \L^{k+1} \e : \bnabla \L^k \D_j dx\,.\]
We now derive bounds for these integrals. In general, the estimates
for the $\J_{j,k}$ and $\mathbf{K}_{j,k}$ integrals are very similar to each
other and only differ regarding details.  We therefore focus mostly on
the $\J_{j,k}$ estimates.

To bound the first terms, $\J_{1,k}$ and $\mathbf{K}_{1, k}$, one can use the fact
that by \Cref{lemma:nonlinear_utility3},
\[\L^k \D_1 = \L^k (\e \times \L\e + \tilde \m \times \L \e + \e \times \L \tilde \m )
= \m \times \L^{k+1} \e + \R_{k, \e} + \R_{k, \tilde \m} +
\L^{k}(\e \times \L \tilde \m).\]
The highest order term here, $\m \times \tilde \L^{k+1} \e$, cancels
in the integral in $\J_{1,k}$ due to orthogonality. Consequently,
application of Cauchy-Schwarz and Young's inequality yields 
\begin{align*}
  |\J_{1,k}| &= \left|\int_\Omega \L^{k+1} \e \cdot (\R_{k, \e} + \R_{k, \tilde \m} + \L^k(\e \times \L \tilde \m)) dx\right| \\
 &\le \frac{\gamma}{2} \|\L^{k+1} \e\|_{L^2}^2 + \frac{1}{2 \gamma} \left(\|\R_{k, \e}\|_{L^2}^2 + \|\R_{k, \tilde \m}\|_{L^2}^2 + \|\L^k(\e \times \L \tilde \m)\|_{L^2}^2\right).
\end{align*}
Making use of \Cref{lemma:nonlinear_utility1},
\Cref{lemma:prod_infty_bound} and the assumption
\cref{eq:m_tilde_bound_ass}, the latter norm can be bounded as
\[\|\L^k(\e \times \L \tilde \m)\|_{L^2} 
\le C \frac{1}{\varepsilon^{2k}} \sum_{i=0}^{2k} \varepsilon^i \|\L \tilde \m\|_{W^{i, \infty}} \|\e\|_{H^{2k-i}_\varepsilon}
\le C \frac{1}{\varepsilon^{2k+1}} \|\e\|_{H^{2k}_\varepsilon}
.
\]
Together with the bounds for $\|\R_{k, \u}\|_{L^2}$ according to
\Cref{lemma:nonlinear_utility3}, we thus
 get
\begin{align}\label{eq:j1est}
  |\J_{1,k}| &\le \frac{\gamma}{2} \|\L^{k+1} \e\|_{L^2}^2 + \frac{C}{2 \gamma}  \frac{1}{\varepsilon^{2(2k+1)}}\|\e\|^2_{H^{2k+1}_\varepsilon}.
\end{align}
We obtain an according estimate for $\mathbf{K}_{1, k}$ by taking
the gradient of $\L^k \D_1$ and proceeding in the same way as for $\J_{1,k}$.
However, we have to consider that
\[\bnabla (\m \times \L^{k+1} \e) = \m \times \bnabla \L^{k+1} \e + \bnabla \m \times \L^{k+1} \e,\]
where only the first term on the right-hand side cancels due to
orthogonality in $\mathbf{K}_{1,k}$. To bound the $L^2$-norm of the
second term, we use the fact that by assumption (A4) we have an
infinity bound on $\bnabla \m$, making it possible to remove it form
the norm. The remaining term can be bounded using
\Cref{lemma:nonlinear_utility1}. In total, this results in
\begin{align}
  |\mathbf{K}_{1,k}|
 &\le \frac{\gamma}{2} \|\sqrt{a^\varepsilon} \bnabla \L^{k+1} \e\|_{L^2}^2 \nonumber
 \\ & \quad+ \frac{1}{2 \gamma} \left(\|\bnabla \m\times \L^{k+1} \e\|_{L^2}^2 + \|\R_{k, \e}\|_{H^1}^2 + \|\R_{k, \tilde \m}\|_{H^1}^2 + \|\bnabla \L^k(\e \times \L \tilde \m)\|_{L^2}^2\right) \nonumber \\
 &\le \frac{\gamma}{2} \|\sqrt{a^\varepsilon} \bnabla \L^{k+1} \e\|_{L^2}^2 + \frac{C}{2 \gamma}
 \frac{1}{\varepsilon^{2(2k+2)}}\|\e\|^2_{H^{2k+2}_\varepsilon}.
 \label{eq:k1est}
\end{align}

For the second kind of integrals, $\J_{2,k}$ and $\mathbf{K}_{2,k}$,
application of Cauchy-Schwarz and Young's inequality, yields
directly that for all constants $\gamma > 0$,
\[  |\J_{2, k}| 
\le
\frac{\gamma}{2} \|\L^{k+1}\e\|^2_{L^2} + \frac{1}{2\gamma} \|\L^k \D_2\|^2_{L^2}.
\]
Using \Cref{lemma:nonlinear_utility1} together with
\Cref{lemma:d2_norm} to go from the norm of $\L^k \D_2$ to an
estimate in terms of $\e$ then gives
\begin{align*}
  \|\L^k \D_2\|_{L^2}
  &\le C \sum_{j=1}^{2k} \frac{1}{\varepsilon^{2k-j}} \|\D_2\|_{H^{j}}
  \le C \frac{1}{\varepsilon^{2k+1}} \sum_{j=1}^{2k} \|\e\|_{H^{j+1}_\varepsilon}
    \le C  \frac{1}{\varepsilon^{2k+1}}\|\e\|_{H^{2k+1}_\varepsilon},
\end{align*}
and it follows that
\begin{equation}
  \label{eq:j2est}
   |\J_{2, k}|
  \le \frac{\gamma}{2} \|\L^{k+1}\e\|^2_{L^2} + \frac{C}{2\gamma}
  \frac{1}{\varepsilon^{2(2k+1)}}\|\e\|^2_{H^{2k+1}_\varepsilon}
\end{equation}
and for $\mathbf{K}_{2,k}$ we obtain similarly,
\begin{equation}
  \label{eq:k2est}
  |\mathbf{K}_{2,k}| \le  \frac{\gamma}{2} \|\sqrt{a^\varepsilon}\bnabla \L^{k+1}\e\|^2_{L^2} + \frac{C}{2\gamma}
 \frac{1}{\varepsilon^{2(2k+2)}}\|\e\|^2_{H^{2k+2}_\varepsilon}.
\end{equation}
Application of \cref{eq:nonlinear_utility4} in \Cref{lemma:ms_elliptic_regularity} to the right-hand side in the estimates
\cref{eq:j1est} and \cref{eq:j2est} then results in
\begin{align*}
  |\J_{1,k}| + \alpha |\J_{2,k}|
  &\le c \frac{\gamma}{2} \|\L^{k+1} \e\|_{L^2}^2 +
  \frac{C}{2\gamma}
 \frac{1}{\varepsilon^{2(2k+1)}}\|\e\|^2_{H^{2k+1}_\varepsilon}
  \\
  &\le c \frac{\gamma}{2} \|\L^{k+1} \e\|_{L^2}^2 + \frac{C}{2\gamma} \left(
    \|\sqrt{a^\varepsilon} \bnabla \L^k \e\|_{L^2}^2 +  \frac{1}{\varepsilon^{2(2k+1)}}\|\e\|_{H^{2k}_\varepsilon}^2
    \right).
\end{align*}
Correspondingly, we find based on \cref{eq:nonlinear_utility4},  \cref{eq:k1est} and \cref{eq:k2est} that
\begin{align*}
  |\mathbf{K}_{1,k}| + \alpha |\mathbf{K}_{2,k}|
  &\le  c \frac{\gamma}{2} \|\sqrt{a^\varepsilon} \bnabla \L^{k+1} \e\|_{L^2}^2 + \frac{C}{2\gamma}  \left(
    \|\L^{k+1} \e\|_{L^2}^2 +
    \frac{1}{\varepsilon^{2(2k+2)}}\|\e\|_{H^{2k+1}_\varepsilon}^2
    \right).
\end{align*}

To do the estimates for $\J_{3,k}$ and $\mathbf{K}_{3,k}$, note
that 
it follows by \Cref{lemma:nonlinear_utility1},
\Cref{lemma:prod_infty_bound} and \cref{eq:m_tilde_bound_ass}, that
\begin{align*}
  \|\L^k (L|\tilde \m|^2 \tilde \m)\|_{H^q}
&\le C \frac{1}{\varepsilon^{q + 2k}} \|L |\tilde \m|^2 \tilde \m\|_{H^{q+2k}_\varepsilon}
\le C \frac{1}{\varepsilon^{q + 2k}} \sum_{j=0}^{q+2k} \varepsilon^j \|\tilde \m\|_{W^{j, \infty}}  \|L |\tilde \m|^2\|_{H^{q+2k-j}_\varepsilon} \\
&\le C \frac{1}{\varepsilon^{q + 2k}} \|L |\tilde \m|^2\|_{H^{q+2k}_\varepsilon}
  \le C \frac{1}{\varepsilon^{q + 2k+1}} \|\nabla |\tilde \m|^2 \|_{H^{q+2k+1}_\varepsilon}.
\end{align*}
Hence, we find for $\J_{3,k}$ and $\mathbf{K}_{3,k}$ that
\begin{subequations}
  \begin{align}
      |\J_{3,k}| &= \left|\frac{1}{2}\int_\Omega \L^{k+1}\e \cdot \L^k(\tilde \m L |\tilde \m|^2) dx \right|
 \le \frac{\gamma}{4} \|\L^{k+1} \e\|^2_{L^2} + \frac{C}{2 \gamma} \frac{1}{\varepsilon^{2(2k+1)}} \|\nabla|\tilde \m|^2 \|_{H^{2k+1}_\varepsilon}^2\,, \label{eq:j3est}\\
 |\mathbf{K}_{3,k}| 
&\le \frac{\gamma}{4} \|\sqrt{a^\varepsilon} \bnabla \L \e\|^2_{L^2}
+ \frac{C}{2 \gamma} \frac{1}{\varepsilon^{2(2k+2)}} \|\nabla|\tilde \m|^2 \|_{H^{2k+2}_\varepsilon}^2
 \label{eq:k3est}\,.
  \end{align}
\end{subequations}
The remaining integrals in \cref{eq:dtw_int} and \cref{eq:dtLe_int},
involving $\boldsymbol \eta$, are bounded using Cauchy-Schwarz and
Young in the same way as in \cref{eq:eta_remainder}, which together with \Cref{lemma:H_eps} and \Cref{lemma:nonlinear_utility1} results in
\begin{align*}
  \left|\int_\Omega a^\varepsilon \bnabla \L^k \e : \bnabla \L^k \boldsymbol \eta dx \right|
  &\le \frac{\gamma}{2} \|\sqrt{a^\varepsilon} \bnabla \L^k \e\|_{L^2}^2 + \frac{C}{2 \gamma} \|\bnabla \L^k \boldsymbol \eta \|_{L^2}^2 \\
  &\le \frac{\gamma}{2} \|\sqrt{a^\varepsilon} \bnabla \L^k \e\|_{L^2}^2 + \frac{C}{2 \gamma} \frac{1}{\varepsilon^{2(2k+1)}} \| \boldsymbol \eta \|_{H^{2k+1}_\varepsilon}^2,
\end{align*}
and correspondingly for \cref{eq:dtLe_int}.
Thus, it holds in total
that when choosing $\gamma$ small enough in the estimates for
$\J_{1,k}$, $\J_{2,k}$ and $\J_{3,k}$, we get from \cref{eq:dtw_int}
\begin{align*}
  \frac{1}{2} \partial_t \|\sqrt{a^\varepsilon} \bnabla \L^k \e\|^2_{L^2}
  &\le C \left(\|\sqrt{a^\varepsilon} \bnabla \L^k \e\|_{L^2}^2 + \J_{R,k}(t)\right),\qquad 0\leq t \leq T^\varepsilon,
\end{align*}
for some $C$ independent of $\varepsilon$ and $t$, and where
$\J_{R,k}$ only depends on lower order norms of $\e$ as well as
terms involving $\boldsymbol \eta$ and $\tilde \m$ ,
\begin{align*}
  \J_{R,k}(t) :=
  \frac{1}{\varepsilon^{2(2k+1)}} \left(\|\e\|_{H^{2k}_\varepsilon}^2 
  + \|\nabla|\tilde \m|^2 \|_{H^{2k+1}_\varepsilon}^2
  +   \|\boldsymbol \eta\|_{H^{2k+1}_\varepsilon}^2 \right).
\end{align*}
Assume now that \cref{eq:eHq_est} holds up to $q = 2k$. This is true for
$k = 0$ according to the estimate in the previous section,
\cref{eq:eL2_est}. Then
\begin{align*}
  \|\e\|_{H^{2k}_\varepsilon}^2 = C \sum_{j=0}^{2k} \varepsilon^{2j} \|\e\|_{H^j}^2
  &\le  C t \sup_{0 \le s \le t} \left( \|\boldsymbol \eta(\cdot, s)\|_{H^{2k}_\varepsilon}^2
    + \|\nabla |\tilde \m(\cdot, s)|^2\|_{H^{2k}_\varepsilon}^2\right)
\end{align*}
and thus
\begin{align*}
  \J_{R,k}(t)
&\le \frac{C(t + 1)}{\varepsilon^{2(2k+1)}} \sup_{0 \le s \le t}\left(\|\boldsymbol \eta(\cdot, s)\|_{H^{2k+1}_\varepsilon}^2 + \|\nabla|\tilde \m(\cdot, s)|^2  \|_{H^{2k+1}_\varepsilon}^2\right).
\end{align*}
Since moreover $\bnabla \L^k \e(0,x)=0$, we have by Gr\"onwall's
inequality, that for $0 \le t \le T^\varepsilon$,
\begin{equation*}
   ||\sqrt{a^\varepsilon} \bnabla \L^k \e(\cdot, t)||_{L^2}^2
   \leq C\int_0^t  \frac{Ct + 1}{\varepsilon^{2(2k+1)}} \sup_{0 \le s \le t}\left(\|\boldsymbol \eta(\cdot, s)\|_{H^{2k+1}_\varepsilon}^2 + \|\nabla|\tilde \m(\cdot, s)|^2 \|_{H^{2k+1}_\varepsilon}^2\right) ds.
 \end{equation*}
 where, as in \cref{eq:eL2_est},
 the prefactor is independent of $\varepsilon$.
  It then holds for $k = 0$ that
  \begin{align*}
   \|\e(\cdot, t)\|_{H^{1}}^2
   &\le C \left(\|\e\|_{L^2}^2 +  \|\sqrt{a^\varepsilon} \bnabla \e\|_{L^2}^2\right)
   \le \frac{C t}{\varepsilon^{2}} \sup_{0 \le s \le t}
\left(\|\boldsymbol \eta(\cdot, s)\|_{H^{1}_\varepsilon}^2 +   \|\nabla|\tilde \m(\cdot, s)|^2  \|_{H^{1}_\varepsilon}^2\right),
 \end{align*}
 while it follows by elliptic regularity,  \cref{eq:elliptic_reg}, that for $k \ge 1$. 
 \begin{align*}
   \|\e(\cdot, t)\|_{H^{2k+1}}^2
   &\le C \left(\|\e\|_{L^2}^2 + \frac{1}{\varepsilon^{2(2k+1)}} \|\e\|_{H^{2k}_\varepsilon}^2 + \|\sqrt{a^\varepsilon} \bnabla L^k \e\|_{L^2}^2\right) \\
   &\le C t\frac{1}{\varepsilon^{2(2k+1)}} \sup_{0 \le s \le t}
\left(\|\boldsymbol \eta(\cdot, s)\|_{H^{2k+1}_\varepsilon}^2 +   \|\nabla|\tilde \m(\cdot, s)|^2  \|_{H^{2k+1}_\varepsilon}^2\right),
 \end{align*}
 which shows the estimate in \Cref{thm:error_norm} for odd $q$ given
 that it holds up to $q-1$.  Finally, we obtain in the same way when
 combining the estimates \cref{eq:k1est}, \cref{eq:k2est},
 \cref{eq:k3est}, for $\mathbf{K}_{1,k}$, $\mathbf{K}_{2,k}$ and
 $\mathbf{K}_{3,k}$, with \cref{eq:dtLe_int}, and again using
 ellitpic regularity \cref{eq:elliptic_reg} and applying
 Gr\"onwall's inequality, that for $0\leq t\leq T^\varepsilon$,
 \begin{align*}
  \|\e\|_{H^{2k+2}}^2 &\le C \left(\|\e\|_{L^2}^2 + \frac{1}{\varepsilon^{2(2k+2)}} \|\e\|_{H^{2k+1}_\varepsilon}^2 +  \|\L^{k+1} \e\|^2_{L^2} \right)\\
                      &\le  C t \frac{1}{\varepsilon^{2(2k+2)}} \sup_{0 \le s \le t}
\left(\|\boldsymbol \eta(\cdot, s)\|_{H^{2k+2}_\varepsilon}^2 +   \|\nabla|\tilde \m(\cdot, s)|^2 \|_{H^{2k+2}_\varepsilon}^2\right),
\end{align*}
which shows the estimate in \Cref{thm:error_norm} for even $q > 0$
given that it holds for $q-1$. This completes the proof.

\section{Estimates of homogenized solution and correction terms}\label{sec:application}

In this section, we provide estimates for the norms of the
correction terms $\m_j$, $j \ge 1$. To obtain these, we use a
theorem for general linear equations of the form as
\cref{eq:higher_order_pde}, which is presented in the next
subsection. We moreover derive bounds for the remaining quantities
involved in the stability estimate \Cref{thm:error_norm}.

\subsection{Linear equation.}\label{sec:lin_eq}
%
%
First, we consider solutions $\m$ to the inhomogeneous
linear equation
\begin{subequations} \label{eq:lin_prob}
\begin{align}
  \partial_\tau \m(x,y,t,\tau) &= \LL \m(x,y,t,\tau) + \F(x,y,t,\tau)\,, \\
  \m(x, y, t, 0) &= \g(x, y, t),
\end{align}
\end{subequations}
with periodic boundary conditions in $y$ and up to some fixed final time $T > 0$.
The linear operator $\LL$ is defined
as in \cref{eq:main_ll}.
It depends on the material coefficient $a$ and on
the solution of the homogenized equation
$\m_0$.

We note that since $\LL$ has a non-trivial null space and $\alpha>0$
this is a degenerate parabolic equation in $(y,\tau)$.
In the following, it will be beneficial to split the solution $\m$, initial data
$\g$ and forcing $\F$ in \cref{eq:lin_prob} into a part that
lies in the null-space, and a part that is orthogonal to it.
To this means,
we introduce the
matrix $\M$ corresponding to the orthogonal projection
onto $\m_0$ and the
averaging operator $\AA$,
\begin{align}\label{eq:aa}
  \M(x, t) :=  \m_0 \m_0^T\,, \qquad  \AA \m := \int_Y \m(x, y, t, \tau) dy \,.
\end{align}
and then define projections
\begin{align}\label{eq:pp_qq}
  \PP\m :=
 (\I - \M) (\I - \AA) \m\,
\qquad\text{and} \qquad
  \QQ \m := \M\m + (\I-\M) \AA \m \,,
\end{align}
which means that $\QQ = \I - \PP$.
According to this definition, $\PP \m$ is orthogonal to $\m_0$ and has zero
average in $y$, while $\QQ \m$ consists of the average of $\m$ and
the contribution to $\m - \AA \m$ that is parallel to $\m_0$. In
particular, $\QQ$ is a projection onto the null-space of $\LL$.
Note that $\PP$ and $\QQ$ depend on $(x, t)$, but not on $(y,\tau)$.
Then we have the following theorem about the size of the two parts
of the solution

\begin{theorem}\label{thm:inhom_lin}
Assume (A1), (A3) and (A5) hold.
If
  $\partial_t^\ell\F(\cdot,\cdot,t,\cdot) \in C(\Real^+; H^{q-2\ell,\infty})$
 and $\partial^\ell_t\g(\cdot,\cdot,t)\in H^{q-2\ell,\infty}$ for $0\leq 2\ell\leq 2k\leq q\leq r$
 and $0\leq t \leq T$,
then $\partial^k_t\m(\cdot,\cdot,t,\cdot)  \in C(\Real^+; H^{q-2k,\infty})$
when $t\in [0,T]$
and
 for each integer $p\geq 0$,
 there are constants $C$ and $\gamma>0$, independent of $\tau\geq 0$, $t\in[0,T]$, $\F$ and $\g$,
 such that
 \begin{align}\label{eq:inhom_linPtime}
    \|\partial^k_t \PP\m(\cdot, \cdot, t, \tau)\|_{H^{q-2k, p}}  &\le C\sum_{\ell=0}^k
 \Big(e^{- \gamma \tau} \|\partial^\ell_t\PP \g(\cdot, \cdot,t)\|_{H^{q-2\ell, p}}
\nonumber\\
&\ \ \ \qquad
+ \int_0^\tau e^{-\gamma(\tau - s)}
\|\partial^\ell_t \PP \F(\cdot, \cdot, t, s)\|_{H^{q-2\ell, p}}
ds\Bigr),\\
    \|\partial^k_t \QQ \m(\cdot, \cdot, t, \tau)\|_{H^{q-2k, p}}  & \le
    \|\partial^k_t \QQ \g(\cdot, \cdot,t)\|_{H^{q-2k, p}}
 + \int_0^\tau \|\partial^k_t \QQ \F(\cdot, \cdot, t, s) \|_{H^{q-2k, p}} ds.
 \label{eq:inhom_linQtime}
  \end{align}
\end{theorem}
This is proved in
\if\longversion0
\cite{longversion}

\else
  \Cref{append:lin_eq}
\fi
. The proof uses standard energy
estimates in which the precise growth rates of the different solution
parts are carefully analyzed.
Note that, since $\tau$ represents the fast scale, sharp bounds
on the growth in $\tau$ are necessary.

\if\longversion0
In \cite{longversion}
\else
In \Cref{append:lin_eq}
\fi
we also prove a few
properties of $\PP(\cdot,t)$ and $\QQ(\cdot,t)$, in particular that
they are bounded operators on $H^{q,p}$ for $0\leq q\leq r$ and
$p\geq 0$, uniformly in $t\in [0,T]$.  The following lemma gives the
more general result.
\begin{lemma}\label{lemma:projections1}
  Assume (A5) holds.  Suppose
  $\partial_t^\ell\v(\cdot,\cdot,t)\in H^{q-2\ell,p}$ for
  $0\leq 2\ell\leq 2k \leq q\leq r$ and $p\geq 0$. Then
    \begin{subequations}      \label{eq:PQtbounded}
          \begin{align}
    \|\partial_t^k\PP\v(\cdot,\cdot,t)\|_{H^{q-2k, p}}&\leq
       C\sum_{\ell=0}^k
   \|\partial_t^{\ell}\v(\cdot,\cdot,t)\|_{H^{q-2\ell,p}}, \\
    \|\partial_t^k\QQ\v(\cdot,\cdot,t)\|_{H^{q-2k, p}}&\leq
   C\sum_{\ell=0}^k
   \|\partial_t^{\ell}\v(\cdot,\cdot,t)\|_{H^{q-2\ell,p}},
    \end{align}
         \end{subequations}
    for $0\leq t\leq T$, where
    $C$ is independent of $\v$ and $t$.
  \end{lemma}

Note that this lemma shows that the projected initial data functions
$\PP\g,\QQ\g\in H^{q,\infty}(\Omega)$ if the unprojected function
$\g\in H^{q,\infty}(\Omega)$ and similar for
the forcing function $\F$ and the time derivatives of the functions.
This justifies why we only ask for
smoothness of the unprojected functions in \Cref{thm:inhom_lin}.

\subsection{Correction terms.}\label{sec:correction_estimates}
We now apply \Cref{thm:inhom_lin} to the correctors $\m_j$ in the
asymptotic expansion \cref{eq:asymptotic_expansion} in order to
obtain estimates for their norms.  Here and throughout the rest of
this section we suppose that the assumptions (A1)-(A5) are true.
We recall
from \cref{eq:higher_order_pde} that the correction terms satisfy
linear PDEs,
\begin{align}\label{eq:mj_pde}
  \partial_\tau \m_j(x, y, t, \tau) &= \LL \m_j(x, y, t, \tau)+\F_{j}(x, y, t, \tau), \\
  \m_j(x,y,t,0) &= 0,\nonumber
\end{align}
where $\F_j$ is defined in \cref{eq:forces}.
Additionally, we 
consider the term $\v$ in the definition of $\m_1$
\cref{eq:m1} to get a better understanding of the behavior of $\m_1$.
As given in \cref{eq:v_total}, $\v$ satisfies
\begin{align}\label{eq:v_pde}
  \partial_\tau \v(x, y, t, \tau) &= \LL \v(x, y, t, \tau), \\
  \v(x,y,t,0) &= - \bnabla_x \m_0(x,t) \boldsymbol\chi(y)\,,\nonumber
\end{align}
where $\boldsymbol\chi$ is the solution to the cell problem,
\cref{eq:cell_problem}.  We then obtain the following result.
\begin{theorem}\label{thm:mjestimate}
  For $0\leq t\leq T$ and $0\leq 2k\leq r-j$ we have
\begin{align}\label{eq:mjregularity}
\partial_t^k\m_j(\cdot,\cdot,t,\cdot)  \in C(\Real^+; H^{r-j-2k,\infty}), \qquad \partial_t^k \v(\cdot, \cdot, t, \cdot) \in C(\Real^+, H^{r-1-2k, \infty}),
\end{align}
and there are constants $\gamma > 0$ and $C$ independent of
$\varepsilon, \tau \ge 0$ and $0 \le t \le T$ such that when $p\geq 0$,
\begin{subequations}  \label{eq:mj_estimate}
\begin{align}
  \|\partial^k_t\v(\cdot, \cdot, t, \tau)\|_{H^{r-1-2k, p}} &\le C e^{-\gamma \tau}\,, \label{eq:v_estimate}\\
  \label{eq:Pmj_estimate}
  \|\partial_t^k\PP\m_j(\cdot, \cdot,  t, \tau)\|_{H^{r-j-2k, p}} &\le C\,,   \\
  \label{eq:Qmj_estimate}
  \|\partial_t^k\QQ\m_j(\cdot, \cdot,  t, \tau)\|_{H^{r-j-2k, p}} &\le C\left(1+\tau^{\max(0,j-2)}\right)\,,  \\
  \|\partial_t^k\m_j(\cdot, \cdot,  t, \tau)\|_{H^{r-j-2k, p}} &\le C\left(1+\tau^{\max(0,j-2)}\right)\,.  \label{eq:mjtot_estimate}
\end{align}
\end{subequations}
Moreover, it holds for $\tau\geq 0$ and $0\leq t\leq T$ that
\begin{equation}
  \label{eq:m1vorth}
\m_1\perp \m_0, \qquad \v\perp\m_0,\qquad
\PP\m_1=\m_1, \qquad \PP\v=\v.
\end{equation}
\end{theorem}
This theorem entails in particular the following.
\begin{itemize}
\item The first corrector $\m_1$ has zero average, is orthogonal to $\m_0$ and stays bounded for all $\tau\geq 0$.
\item As the first one, the second corrector $\m_2$ is uniformly bounded in $\tau$,
but it is neither orthogonal nor parallel to $\m_0$.
\item Higher order correctors are not bounded in $\tau$ but grow algebraically.
\end{itemize}
To prove the theorem, we use induction. We consider first the base
cases, norms of $\v$ and $\m_1$ and their time derivatives, which in
turn makes it possible to bound the norms of
$\partial_t^k \QQ \F_2$. Then we provide a utility lemma giving
estimates for the quantities involved in higher order $\F_j$,
$j \ge 3$. Finally, we conclude the proof with an induction step
showing \cref{eq:mj_estimate} for general $j$.

\subsubsection{$\m_1$ and $\v$ estimates.}\label{sec:m1v}
To begin with, we show \cref{eq:m1vorth}. For $\m_1$, the forcing
term $\F_1$ only depends on $\m_0$.  In fact, as
$\Z_0=\L_1\m_0=\nabla_x\m_0 \nabla_y a$ the expression for $\F_1$ in
\cref{eq:forces1} can be written as
\begin{align*}
 \F_1(x, y, t, \tau)
&= -\m_0(x, t) \times [\nabla_x \m_0(x, t) + \alpha \m_0(x, t) \times \nabla_x \m_0(x, t)] \nabla_y a(y)
\,,
\end{align*}
which shows that $\F_1$ is orthogonal to $\m_0$. Moreover,
since the averaging operator $\AA$ commutes with differentiation in $y$,
\begin{align*}
  \AA \F_1 &=
 -\m_0(x, t) \times [\nabla_x \m_0(x, t) + \alpha \m_0(x, t) \times \nabla_x \m_0(x, t)] \AA\nabla_y a(y)=0\,,
\end{align*}
and consequently $ \QQ\F_1 = 0$ and $\PP\F_1 = \F_1$.
For $\v$ it holds at the initial time $\tau = 0$,
\[\g(x,y,t)=\v(x,y,t,0)=-\nabla_x\m_0(x,t)\boldsymbol\chi(y).\]
Since we choose $\boldsymbol\chi$ such that $\AA \boldsymbol\chi = 0$ and
due to the fact that $\m_0$ is orthogonal to its gradient,
$\m_0\cdot \bnabla_x \m_0=\nabla|\m_0|^2/2 = 0$, we have $\QQ \g=0$
and $\PP \g=\g$.  It thus follows from \Cref{thm:inhom_lin} that
$\QQ\m_1=\QQ\v=0$ and consequently for all $\tau \ge 0$,
\[\PP \m_1 = \m_1, \qquad \PP \v = \v.\]
Hence, \cref{eq:m1vorth} holds.  Next, we consider the norms of
$\partial_t^k \v$ and $\partial_t^k \m_1$.  \Cref{lemma:m0bound}
shows that for $0\leq \ell\leq k$,
\begin{align*}
  \|\partial^\ell_t \F_1 \|_{H^{r-1-2\ell,p}} &\leq C\sum_{s=0}^\ell
  \|\nabla_x \partial^s_t \m_0\nabla_y a\|_{H^{r-1-2s,p}}
   \leq
  C\sum_{s=0}^\ell\|\partial^s_t \m_0\|_{H^{r-2s,p}} \|a\|_{H^{p+1}}
   \leq C,
\end{align*}
and similarly,
\begin{align*}
  \|\partial^\ell_t\g\|_{H^{r-1-2\ell,p}} \leq
\|\nabla_x\partial^\ell_t\m_0\|_{H^{r-1-2\ell}}
\|\boldsymbol\chi\|_{H^p} \le
C\|\partial^\ell_t\m_0\|_{H^{r-2\ell}}\le C \,.
\end{align*}
Since $\F_1$ and $\g$ coincide with their $\PP$-projections,
we have
\[  \|\partial^\ell_t \PP \F_1 \|_{H^{r-1-2\ell,p}} =
  \|\partial^\ell_t \F_1 \|_{H^{r-1-2\ell,p}}, \quad \|\partial^\ell_t\PP\g\|_{H^{r-1-2\ell,p}}=
  \|\partial^\ell_t\g\|_{H^{r-1-2\ell,p}},\]
and thus obtain from  \Cref{thm:inhom_lin} that
\begin{align}
  \|\partial^k_t\v\|_{H^{r-1-2k, p}} &=
  \|\partial^k_t\PP\v\|_{H^{r-1-2k, p}} \leq
  C  \sum_{\ell=0}^k  e^{-\gamma \tau} \|\partial^\ell_t\PP\g\|_{H^{r-1-2\ell,p}}
\le C e^{-\gamma \tau}, \label{eq:v_est1}\\
    \|\partial^k_t \m_1 \|_{H^{r-1-2k, p}}&=
        \| \partial^k_t\PP\m_1 \|_{H^{r-1-2k, p}}
      \le C \sum_{\ell=0}^k  \int_0^\tau  e^{-\gamma(\tau - s)} \|\partial^\ell_t\PP \F_1\|_{H^{r-1-2\ell, q}} ds \nonumber \\
      &\le C   \int_0^\tau  e^{-\gamma(\tau - s)}ds
\le C \label{eq:m1_est1}
  \,.
\end{align}
This shows \cref{eq:mj_estimate} for $\v$ and the first corrector $\m_1$.

Consider now $\F_2$ as defined in \cref{eq:F2}, which consists only of quantities involving $\v$ and $\m_1$. Combining \cref{eq:AF2} with the definition
of the homogenized solution \cref{eq:m0_eq} shows that the average
of $\F_2$ is $\AA \F_2 = - (\E_1 + \E_2)$, with $\E_1$ and $\E_2$ given in
\cref{eq:E1} and \cref{eq:E2},
\[\E_1 = \AA(\R_1 + \alpha \S_1), \qquad \E_2 = \m_0 \times \AA(\L_1 \v) + \m_0 \times \m_0 \times \AA(\L_1 \v),\]
where $\R_1=\m_1 \times \L_2 \v$ and
$\S_1=\m_1 \times \m_0 \times \L_2 \v$, defined according to
\cref{eq:T_notation} and \cref{eq:S_notation}, are parallel to
$\m_0$ due to the orthogonality given in \cref{eq:m1vorth}. This implies that $\E_1$ is
parallel to $\m_0$ as well, while $\E_2$ is orthogonal to $\m_0$ as it is of the form
$\m_0 \times \cdot$. Hence, it holds that
\[(\I-\M)(\E_1 + \E_2) = \E_2.\]
Again using the fact that terms of the form
$\m_0 \times \cdot$, as well as $\partial_t \m_0$, are orthogonal to
$\m_0$, we thus can show that  application of the operator $\QQ$ to $\F_2$  yields
$$
  \QQ \F_2 = \M\F_2 + (\I-\M)\AA \F_2 =
  -\M(\R_1+\alpha\S_1) - (\I-\M)(\E_1+\E_2),
  =
 -\R_1-\alpha\S_1 -\E_2.
$$
Using \Cref{lemma:prodindex}, with $j=2$ and $m=m'=1$ as
$\partial_t^k\m_1,\partial_t^k\v \in H^{r-1-2k,p}$, therefore yields
together with \cref{eq:v_est1} and \cref{eq:m1_est1},
\begin{equation*}
  \begin{split}
  \|\partial_t^k \R_1 \|_{H^{r-2-2k, p}}&\leq
  C\sum_{\ell=0}^k
  \|\partial_t^{k-\ell} \m_1 \|_{H^{r-1-2k+2\ell, p+2}}
  \|\partial_t^\ell \L_2\v \|_{H^{r-1-2\ell, p}} \\ &\leq
  C\sum_{\ell=0}^k
  \|\partial_t^\ell \v \|_{H^{r-1-2\ell, p+2}}
  \leq Ce^{-\gamma\tau},
\\  \|\partial_t^k \S_1 \|_{H^{r-2-2k, p}}&\leq
  C\sum_{\ell=0}^k
  \|\partial_t^\ell (\m_0\times\L_2\v) \|_{H^{r-1-2\ell, p}}\leq
  C\sum_{\ell=0}^k
  \|\partial_t^\ell\L_2\v \|_{H^{r-1-2\ell, p}} \\ &\leq
  C\sum_{\ell=0}^k
  \|\partial_t^\ell\v \|_{H^{r-1-2\ell, p+2}}
  \leq Ce^{-\gamma\tau}.
      \end{split}
    \end{equation*}
    Finally, using \Cref{lemma:m0bound} with $\f = \AA \L_1 \v$ gives
    \[
  \|\partial_t^k \E_2 \|_{H^{r-2-2k, p}}\leq
  C\sum_{\ell=0}^k
  \|\partial_t^\ell \AA\L_1\v\|_{H^{r-2-2\ell, p}}\leq
  C\sum_{\ell=0}^k
  \|\partial_t^\ell\v\|_{H^{r-1-2\ell, p+1}}
  \leq Ce^{-\gamma\tau}.
  \]
We thus conclude that
\begin{equation}
  \label{eq:dtk_QF2}
  \|\partial_t^k \QQ \F_2\|_{H^{r-j-2k, p}} \le C e^{-\gamma \tau}\,.
\end{equation}

\subsubsection{Higher order $\m_j$-estimate.}
%
%
We now consider $\m_j$ with $j\geq 2$.  First, note that due to
assumption (A1), it holds in general for $\m_j$, $j \ge 0$, that
\begin{align*}
  \|\L_k \m_j\|_{H^{q, p}} \le C \|\m_j\|_{H^{q+2-k, p+k}}\,, \qquad k = 0, 1, 2,
\end{align*}
for $p, q \ge 0$, whenever the norms are bounded.
This can be used to prove a lemma providing upper bounds for norms of all the intermediate quantities involved in the forcing term $\F_j$ in \cref{eq:mj_pde}.
\begin{lemma}\label{lemma:rstvzestimate}
  Suppose that \cref{eq:mjregularity} and \cref{eq:mj_estimate} hold
  for $1\leq j\leq J$, $0 \le 2k \le r-j$ and $0 \le t \le
  T$. 
Then, for $p\geq 0$,
\begin{align}
\|\partial^k_t\Z_j\|_{H^{r-j-1-2k,p}}&\leq  C\left(1+\tau^{\max(0,j-2)}\right),\quad 0\leq j\leq J,\quad 0\leq 2k\leq r-j-1, \\
\|\partial^k_t\X_j\|_{H^{r-j-2k,p}}&\leq C\left(1+\tau^{\max(0,j-2)}\right),\quad 1\leq j\leq J,
\quad 0\leq 2k\leq r-j, \label{eq:Xjeq}
\end{align}
where $\X_j$ is any of $\R_j$, $\S_j$, $\T_j$ and $\V_j$ as defined in \cref{eq:V_notation}, \cref{eq:T_notation} and \cref{eq:S_notation} and the constant $C$ is independent of $t$ and $\tau$.
\end{lemma}
\begin{proof}
Let $1\leq j\leq J$, $p\geq 0$ and $0\leq 2k\leq r-j-1$.
For $\Z_j$ we have
\begin{align*}
\|\partial_t^k\Z_j\|_{H^{r-j-1-2k,p}}&\leq
\|\partial_t^k\L_0\m_{j-1}\|_{H^{r-j-1-2k,p}}+
\|\partial_t^k\L_1\m_{j}\|_{H^{r-j-1-2k,p}}\\
&\leq
C(\|\partial_t^k\m_{j-1}\|_{H^{r-j+1-2k,p}}+
\|\partial_t^k\m_{j}\|_{H^{r-j-2k,p+1}})\\
&\leq C (1+\tau^{\max(0,j-3)}+\tau^{\max(0,j-2)})
\leq C\left(1+\tau^{\max(0,j-2)}\right).
\end{align*}
 Since by definition, $\m_{-1}\equiv 0$, the result still holds true for $j=0$.
For $\V_j$ we then have accordingly, when $0\leq 2k\leq r-j$,
\begin{align*}
\|\partial_t^k\V_j\|_{H^{r-j-2k,p}}&\leq
\|\partial_t^k\L_2\m_{j}\|_{H^{r-j-2k,p}}+
\|\partial_t^k\Z_{j-1}\|_{H^{r-j-2k,p}}\\ &\leq
C\|\partial_t^k\m_{j}\|_{H^{r-j-2k,p+2}}+
\|\partial_t^k\Z_{j-1}\|_{H^{r-j-2k,p}}\\
&\leq C (1+\tau^{\max(0,j-2)}+\tau^{\max(0,j-3)})
\leq C\left(1+\tau^{\max(0,j-2)}\right).
\end{align*}
This shows \Cref{lemma:prodindex} for $\X_j = \V_j$.  Suppose now
that that $\Y_m$ satisfies \cref{eq:Xjeq} for $1\leq m\leq j$ and
that $m' \in \{m-1, m\}$. Since $j-m'+m \le j+1$, we then find by \Cref{lemma:prodindex}
that
\begin{align*}
\left\|\partial_t^k (\m_{j-m'}\times\Y_m)\right\|_{H^{r-j-2k,p}}
&\leq C
\sum_{\ell=0}^k
\left\|\partial_t^{k-\ell} \m_{j-m'}\right\|_{H^{r-j+m'-2k+2\ell,p+2}}
\left\|\partial_t^\ell \Y_m\right\|_{H^{r-m-2\ell,p}}\\
&\leq
C(1+\tau^{\max(0,j-m'-2)})(1+\tau^{\max(0,m-2)}) \\
&\leq
C(1+\tau^{\max(0,j-m'-2,m-2,j + m - m'-4)}) \\
&\leq C(1+\tau^{\max(0,j-2,j-3,j-4)}) \leq C(1+\tau^{\max(0,j-2)}).
\end{align*}
When using this result for $\Y_m = \V_m$ and $m' = m$, we get
\begin{align*}
\|\partial_t^k\T_j\|_{H^{r-j-2k,p}}
&\leq
\sum_{m=1}^{j}\left\|\partial_t^k(\m_{j-m}\times\V_m\right)\|_{H^{r-j-2k,p}}
\leq C\left(1+\tau^{\max(0,j-2)}\right).
\end{align*}
Similarly, when $m' = m-1$, we find by choosing $\Y_m$ to be $\V_m$ and $\T_m$ respectively,
\begin{align*}
\|\partial_t^k\R_j\|_{H^{r-j-2k,p}}
&\leq
\sum_{m=1}^{j}\left\|\partial_t^k(\m_{j+1-m}\times\V_m)\right\|_{H^{r-j-2k,p}}
\leq C\left(1+\tau^{\max(0,j-2)}\right),\\
\|\partial_t^k\S_j\|_{H^{r-j-2k,p}}
&\leq
\sum_{m=1}^{j}\left\|\partial_t^k(\m_{j+1-m}\times\T_m)\right\|_{H^{r-j-2k,p}}
\leq C\left(1+\tau^{\max(0,j-2)}\right).
\end{align*}
This proves the lemma.
\end{proof}

We now have the tools necessary to conclude the induction step for
\Cref{thm:mjestimate}.
For $j = 1$, we have already shown \cref{eq:mjregularity} and
\cref{eq:mjtot_estimate}. Assume now that they hold up to some $j$
with $1 \le j \le r-2k-1$. We then show in the following that they
also hold for $j+1\leq r-2k$.  To this means, suppose
$0\leq 2k\leq r-j-1 =: q'$ and $p\geq 0$. From the definition of
$\F_j$ according to \cref{eq:forces} it follows using
\Cref{lemma:m0bound} and \Cref{lemma:rstvzestimate}
 that
\begin{align*}
&\|\partial_t^k\F_{j+1}\|_{H^{q'-2k,p}}
  \leq  \|\partial_t^{k+1} \m_{j-1}\|_{H^{q'-2k,p}}
    +
\|\partial_t^k\mathbf{R}_{j}\|_{H^{q'-2k,p}}+
    \|\partial_t^k(\m_0\times\Z_{j})\|_{H^{q'-2k,p}} \\
    &\qquad \qquad + \alpha \left(\|\partial_t^k(\m_0 \times \mathbf{R}_{j})\|_{H^{q'-2k,p}}
+ \|\partial_t^k(\m_0\times\m_0\times\Z_{j})\|_{H^{q'-2k,p}}
+ \left\|\partial_t^k\S_j\right\|_{H^{q'-2k,p}}\right) \\
&\quad\leq
  C\sum_{\ell=0}^k \left(\|\partial_t^\ell\Z_{j}\|_{H^{q'-2\ell,p}}+
  \|\partial_t^\ell\mathbf{R}_{j}\|_{H^{q'-2\ell,p}}\right)  + \alpha\left\|\partial_t^k\S_j\right\|_{H^{q'-2k,p}}
    +  \|\partial_t^{k+1} \m_{j-1}\|_{H^{q'-2k,p}}\\
&\quad\leq C\left(1+\tau^{\max(0,j-2)}+\tau^{\max(0,j-3)}\right)\leq
C\left(1+\tau^{\max(0,j-2)}\right).
\end{align*}
By \cref{eq:PQtbounded} the same estimate holds for $\partial_t^k\PP\F_{j+1}$
and $\partial_t^k\QQ\F_{j+1}$. However, for the latter we have due to \cref{eq:dtk_QF2},
$$
\|\partial_t^k \QQ\F_{j+1} \|_{H^{r-j-1-2k, p}}\leq C
\begin{cases}
 e^{-\gamma\tau}, & j=1,\\
  1+\tau^{j-2}, & j\geq 2.
\end{cases}
$$
The estimate \cref{eq:Pmj_estimate} with $j+1$ then follows from
\Cref{thm:inhom_lin} as
\begin{align*}
    \|\partial_t^k \PP\m_{j+1} \|_{H^{r-j-1-2k, p}}
&\leq
C\sum_{\ell=0}^k \int_0^\tau  e^{-\gamma(\tau - s)} \|\partial_t^\ell \PP \F_{j+1}\|_{H^{r-j-1-2\ell, p}}ds
\\ &
\leq
C\int_0^\tau  e^{-\gamma(\tau - s)}(1 + \tau^{\max(0,j-2)})ds\leq
C,
\end{align*}
and accordingly,
\begin{align*}
    \|\partial_t^k \QQ\m_{j+1} \|_{H^{r-j-1-2k, p}}
&\leq
C \int_0^\tau  \|\partial_t^k \QQ \F_{j+1}\|_{H^{r-j-1-2k, p}}ds
\leq
C \int_0^\tau
\begin{cases}
 e^{-\gamma\tau}, & j=1,\\
 1+\tau^{j-2}, & j\geq 2,
\end{cases}
\ ds\\
&\leq
C
\begin{cases}
1, & j=1,\\
 1+\tau^{j-1}, & j\geq 2,
\end{cases}
\leq
C(1+\tau^{\max(0,j-1)}),
\end{align*}
which yields \cref{eq:Qmj_estimate} with $j+1$.
Finally, \cref{eq:mjtot_estimate} is obtained using the triangle inequality.
This concludes the induction step and the proof of \Cref{thm:mjestimate}.


\subsection{Approximations $\tilde \m_J$ and $\tilde \m^\varepsilon_J$.}
In this section we consider the approximation
$\tilde \m^\varepsilon_J$ to $\tilde \m$ as defined in
\cref{eq:main_m_tilde_expanded} and correspondingly,
\begin{equation*}
  \tilde{\m}_J(x,y,t,\tau;\varepsilon) = \m_0(x, t) + \sum_{j=1}^J \varepsilon^j \m_j(x, y, t, \tau), \qquad \tilde \m_J^\varepsilon(x, t) = \tilde \m_J(x, x/\varepsilon, t, t/\varepsilon^2).
\end{equation*}
We are interested in different aspects of
$\tilde{\m}^\varepsilon_J$ and
$\tilde{\m}_J$ up to time $T^\varepsilon$ as given in \cref{eq:Tepsdef},
\cref{eq:Teps}
, here repeated for convenience of the reader,
\begin{equation}
  \label{eq:Teps2}
  T^\varepsilon := \varepsilon^\sigma T \qquad \text{with} \qquad
\begin{cases}
   0 \le \sigma \le 2, & J \le 2, \\
   1 - \frac{1}{J-2} \le \sigma \le 2, & J \ge 3.
 \end{cases}
\end{equation}
Up to this final time, we have
\[1 + \tau \le C(1 + \varepsilon^{- (2 - \sigma)}) \le C
  \varepsilon^{- (2 - \sigma)}.\]
As a consequence, we can simplify
the estimate \Cref{eq:mjtot_estimate} for final time
$T^\varepsilon$. Under the assumptions in \Cref{thm:mjestimate}, it
holds for $0 \le t \le T^\varepsilon$, $0 \le \tau \le T^\varepsilon/\varepsilon^2$, that
\begin{equation}
  \label{eq:mjeps_estimate}
  \|\partial_t^k \m_j(\cdot, \cdot, t, \tau)\|_{H^{r-j-2k, p}} \le C \varepsilon^{-(2-\sigma)\max(0, j-2)}, \qquad 0 \le 2k \le r-j\,.
\end{equation}

\subsubsection{Norms of approximations.}
We start by estimating the approximations
$\tilde \m_J$ and $\tilde \m_J^\varepsilon$ in different Sobolev norms.
We obtain the following theorem.
\begin{theorem}\label{thm:mJnorms}
For $0\leq t\leq T$,
\begin{align}\label{eq:mJregularity}
\tilde{\m}_J(\cdot,\cdot,t,\cdot;\varepsilon)  \in C(\Real^+; H^{r-J,\infty}(\Omega, \Real^3)),\qquad
\tilde{\m}^\varepsilon_J(\cdot,t)  \in H^{r-J}(\Omega).
\end{align}
Moreover, 
consider $T^\varepsilon$ as given in \cref{eq:Teps2}, then
for any $p\geq 0$ 
\begin{align}\label{eq:mtilde_estimate}
  \|\tilde{\m}_J(\cdot, \cdot,  t, \tau;\varepsilon)\|_{H^{r-J, p}} &\le  C,\qquad
  0\leq t\leq T^\varepsilon,\quad 0\leq \tau\leq \frac{T^\varepsilon}{\varepsilon^2}.
\end{align}
Additionally, for $0\leq q \leq r-J$ and $0\leq q'\leq r-J-2$ 
\begin{align}\label{eq:mtildeeps_estimate}
  \|\tilde{\m}^\varepsilon_J(\cdot,  t)\|_{H^{q}} \le  C\varepsilon^{\min(0,1-q)},\qquad
  \|\tilde{\m}^\varepsilon_J(\cdot,  t)\|_{W^{q',\infty}} \le  C\varepsilon^{\min(0,1-q')},\quad
  0\leq t\leq T^\varepsilon.
\end{align}
All constants denoted $C$ are independent of $\tau$, $t$ and $\varepsilon$,
but depend on $T$.
\end{theorem}
\begin{proof}
  The simplification \cref{eq:mjeps_estimate} of the estimate in
  \Cref{thm:mjestimate} gives for fixed $t$, $\tau$ as in \cref{eq:mtilde_estimate} and
  $0\leq j\leq J$,
\begin{align*}
   \varepsilon^j\|\m_j\|_{H^{r-j, p}}
  \leq C \varepsilon^{j - (2 - \sigma)\max(0,j-2)} = C
    \begin{cases}
      \varepsilon^j, & 0 \le j \le 2, \\
      \varepsilon^{j - (2-\sigma)(j-2)} & 3 \le j \le J\,,
    \end{cases}
    \le C
   \begin{cases}
      1, & j = 0, \\
      \varepsilon, &j \ge 1\,,
      \end{cases}
\end{align*}
where we used for the second case that
\[j - (2-\sigma)(j-2) = 2 + (j-2)(\sigma-1) \ge 2 - \frac{j-2}{J-2} \ge 2 - \frac{J-2}{J-2} = 1\,.\]
This shows \cref{eq:mtilde_estimate}, as
\begin{align*}
  \|\tilde{\m}_J\|_{H^{r-J, p}}
  &\leq \sum_{j=0}^J \varepsilon^j \|{\m_j}\|_{H^{r-J, p}}\leq \sum_{j=0}^J \varepsilon^j \|{\m_j}\|_{H^{r-j, p}} \le C.
\end{align*}
For the second statement, we use \Cref{lemma:multiscaleest} and the fact that $\m_0$ does not depend on $y$, which yields
\begin{align*}
\|\tilde{\m}^\varepsilon_J(\cdot,  t)\|_{H^{q}}
&\leq
\|\m_0(x, t)\|_{H^{q}} + \sum_{j=1}^J \varepsilon^j \|\m_j(\cdot , \cdot /\varepsilon, t, t/\varepsilon^2)\|_{H^{q}} \\
&\leq
\|\m_0(x, t)\|_{H^{q}} + C \sum_{j=1}^J \varepsilon^{j-q} \|\m_j(\cdot , \cdot, t, t/\varepsilon^2)\|_{H^{q,q+2}}.
\end{align*}
Proceeding similarly to before, we then get for $j \ge 1$,
\[ \|\m_j\|_{H^{q,q+2}}
\le  \|\m_j\|_{H^{r-J,q+2}}
\le  \|\m_j\|_{H^{r-j,q+2}}
\le C \varepsilon^{(2-\sigma)\max(0, j-2)}
\le C \varepsilon^{1-j}.
\]
Therefore,
$$
\|\tilde{\m}^\varepsilon_J(\cdot,  t)\|_{H^{q}}
\leq
C (1 + \varepsilon^{1-q}).
$$
Finally, by \Cref{lemma:multiscaleest},
\[
\begin{split}
\|\tilde{\m}^\varepsilon_J(\cdot,  t)\|_{W^{q',\infty}}
&\le \|\m_0(x, t)\|_{W^{q',\infty}} + C \sum_{j=1}^J \varepsilon^{j-q'}
 \|\m_j(\cdot, \cdot, t, t/\varepsilon^2)\|_{H^{q' + 2, q'+2}}\,,
\end{split}
\]
where
 \[ \|\m_j\|_{H^{q' + 2, q'+2}} \le \|\m_j\|_{H^{r-J, q'+2}}\,,\]
from which \cref{eq:mtildeeps_estimate} follows in the same way as above. This completes the proof.

\end{proof}

\subsubsection{Residual.}
The truncated approximation $\tilde \m_J^\varepsilon$ satisfies the
differential equation \cref{eq:main_prob} for the original
$\m^\varepsilon$ only up to a certain residual
$\boldsymbol \eta^\varepsilon_J$. In the following, we derive an
expression for this residual $\boldsymbol \eta_J^\varepsilon$ that is then used
to obtain a bound for its $\|\cdot\|_{H^q}$-norm.

\begin{theorem}\label{thm:m_pert}
Let 
the residual $\boldsymbol \eta^\varepsilon_J$ be defined as
\begin{align}\label{eq:eta_eq}
  \boldsymbol \eta^\varepsilon_J := \partial_t \tilde \m_J^\varepsilon + \tilde \m_J^\varepsilon \times \L \tilde \m_J^\varepsilon + \alpha \tilde \m_J^\varepsilon \times
  \tilde \m_J^\varepsilon \times \L \tilde \m_J^\varepsilon\,
\end{align}
and suppose $2\leq J\leq r-2$ and $0 \le t \le T^\varepsilon$ with $T^\varepsilon$ as in \cref{eq:Teps2}. 
Then for $0\leq q \leq r-J-2$,
\begin{align*}
  \|{\boldsymbol \eta}^\varepsilon_J(\cdot, t)\|_{H^{q}} &\le  C\varepsilon^{1+(\sigma-1)(J-2)-q},\qquad \|\boldsymbol \eta^\varepsilon_J(\cdot, t)\|_{H^q_\varepsilon} \le C \varepsilon^{1 + (\sigma - 1)(J-2)}.
\end{align*}
The constant  $C$ is independent of $t$ and $\varepsilon$,
but depends on $T$.
\end{theorem}
\begin{proof}
Using the notation given in \cref{eq:V_notation} and \cref{eq:T_notation}, we find along the  same steps as in \Cref{sec:homogenization}
that the expression corresponding to \cref{eq:Lm_expansion} for the
truncated expansion $\tilde \m_J^\varepsilon$ becomes
\begin{align*}
  \L \tilde \m^\varepsilon_J &= \sum_{j=0}^J \varepsilon^{j-2} \L_{2} \m_j
+ \sum_{j=0}^J \varepsilon^{j-1} \L_{1} \m_j
+ \sum_{j=0}^J \varepsilon^{j} \L_{0} \m_j
= \sum_{j=1}^J \varepsilon^{j-2} \V_{j} + \varepsilon^{J-1} \boldsymbol \mu_1 \,,
\end{align*}
where
\[\boldsymbol \mu_1 :=  \L_1 \m_{J} + \L_0 \m_{J-1} + \varepsilon \L_0 \m_J \,.\]
To obtain an expanded expression for the precession term, we then take
the cross product of $\tilde \m^\varepsilon_J$ with the expanded expression for $\L \tilde \m^\varepsilon_J$ which results in
\begin{align}\label{eq:m_eps_precession}
  \tilde \m^\varepsilon_J \times \L \tilde \m^\varepsilon_J
  &= \sum_{j=0}^J \varepsilon^j \m_j \times
    \left(\sum_{k=1}^J \varepsilon^{k-2} \V_{k} + \varepsilon^{J-1} \boldsymbol \mu_1\right)
    =\sum_{j=1}^J \varepsilon^{j-2} \T_{j} + \varepsilon^{J-1} \boldsymbol \eta_1,
\end{align}
where
\begin{align*}
  \boldsymbol \eta_1 :&=
  \boldsymbol \mu_2 +  \tilde\m_J \times \boldsymbol \mu_1
\qquad \text{ and }
\qquad
\boldsymbol \mu_2 := \sum_{j=0}^{J-1} \varepsilon^{j} \sum_{k=j+1}^J \m_{J+1+j-k} \times \V_{k}\,.
\end{align*}
Taking one more cross product by $\tilde \m^\varepsilon$ yields an
expanded form of the damping term,
\begin{align}\label{eq:m_eps_damping}
  \tilde \m^\varepsilon_J \times \tilde \m^\varepsilon_J \times \L \tilde \m^\varepsilon_J
  =  \sum_{j=1}^J \varepsilon^{j-2}  \sum_{k=1}^j \m_{j-k} \times  \T_{k} + \varepsilon^{J-1} \boldsymbol \eta_2
\,,
\end{align}
where
\begin{align*}
  \boldsymbol \eta_2 :&= \boldsymbol \mu_3 +  \tilde\m_J \times \boldsymbol \eta_1
\qquad \text{ and }
\qquad
\boldsymbol \mu_3 := \sum_{j=0}^{J-1} \varepsilon^{j} \sum_{k=j+1}^{J} \m_{J+1+j-k} \times \T_{k} \,.
\end{align*}
Moreover, it holds for the time derivative of $\tilde \m_J^\varepsilon$ that
\if\longversion1
\begin{align}
  \label{eq:dt_m_eps}
\partial_t \tilde \m^\varepsilon_J = \sum_{j=0}^J (\varepsilon^j \partial_t \m_j + \varepsilon^{j-2} \partial_\tau \m_j) =
\sum_{j=0}^J \varepsilon^{j-2} (\partial_t \m_{j-2} + \partial_\tau \m_j)
+ \varepsilon^{J-1} \partial_t \m_{J-1} + \varepsilon^J \partial_t \m_J\,.
\end{align}
\else
\begin{align}
  \label{eq:dt_m_eps}
\partial_t \tilde \m^\varepsilon_J =
\sum_{j=0}^J \varepsilon^{j-2} (\partial_t \m_{j-2} + \partial_\tau \m_j)
+ \varepsilon^{J-1} \partial_t \m_{J-1} + \varepsilon^J \partial_t \m_J\,.
\end{align}
\fi
Putting the expanded expressions as given in
\cref{eq:m_eps_precession}, \cref{eq:m_eps_damping} and
\cref{eq:dt_m_eps} into the definition of $\boldsymbol \eta^\varepsilon$ that is
given by the differential equation, \cref{eq:eta_eq}, yields
together with \cref{eq:higher_order_pde} that
\begin{align}\label{eq:eta}
  \boldsymbol \eta_J^\varepsilon(x, t) = \boldsymbol \eta_J(x, x/\varepsilon, t, t/\varepsilon^2),
  \quad\text{where}\quad \boldsymbol \eta_J=\varepsilon^{J-1} (\partial_t \m_{J-1} + \varepsilon \partial_t \m_J
+  \boldsymbol \eta_1 +  \alpha \boldsymbol \eta_2)\,.
\end{align}
This implies that in
order to get a bound for the $H^q$-norm in space of $\boldsymbol \eta^\varepsilon$
we have to consider both $x$ and $y$ in the expanded form.
By \Cref{lemma:multiscaleest} it holds that
\begin{align}\label{eq:eta_h1_start}
   \|\boldsymbol \eta_J^\varepsilon(\cdot, t)\|_{H^q}
\le \frac{C}{\varepsilon^{q}}\|\boldsymbol \eta_J(\cdot, \cdot, t, t/\varepsilon^2)\|_{H^{q, q+2}}\,.
\end{align}

Using the explicit form of $\boldsymbol \eta_J$ given in \cref{eq:eta}, one can
obtain an upper bound on the norm of $\boldsymbol \eta_J$.
To begin with, let $q' := r - J - 2$, then we have
\begin{align*}
  \|\boldsymbol \eta_J(\cdot,  \cdot, t, \tau)\|_{H^{q',p}} 
&\le C\varepsilon^{J-1} \left(\|\partial_t \m_{J-1}\|_{H^{q',p}} + \varepsilon \|\partial_t \m_J\|_{H^{q',p}} + \|\boldsymbol \eta_1\|_{H^{q',p}}  + \|\boldsymbol \eta_2\|_{H^{q',p}}\right)\,.
\end{align*}
For the first two terms we get from \cref{eq:mjeps_estimate},
 as $J\geq 2$,
\begin{align*}
\varepsilon^{J-1}\|\partial_t \m_{J-1}\|_{H^{q',p}} + \varepsilon^J \|\partial_t \m_J\|_{H^{q',p}}
&\leq C\varepsilon^{J-1+(\sigma-2)\max(0,J-3)}+
C\varepsilon^{J+(\sigma-2)(J-2)}\\
&= C\varepsilon^{1+(\sigma-1)(J-2)}
(\varepsilon^{(2-\sigma)(J-2-\max(0,J-3))}+
\varepsilon)\\
&\leq C\varepsilon^{1+(\sigma-1)(J-2).}
\end{align*}
Note that by the assumptions on $J$ and $\sigma$ the exponent for $\varepsilon$ here is positive.
To get an estimate for the norms of $\boldsymbol \eta_1$ and $\boldsymbol \eta_2$, consider first the
norms of the perturbation terms $\boldsymbol \mu_i$, $i = 1, 2, 3$,
individually. By \cref{eq:mjeps_estimate} and since $J \ge 2$, it holds that
\begin{align*}
  \|\boldsymbol \mu_1\|_{H^{r-J-2,p}}
&\le   C \left(\|\m_{J} \|_{H^{r-J-1, p+1}} + \|\m_{J-1} \|_{H^{r-J, p}} + \varepsilon \| \m_J\|_{H^{r-J, p}} \right) \\
&\le   C \left(\|\m_{J} \|_{H^{r-J, p+1}} + \|\m_{J-1} \|_{H^{r-J+1, p}} + \varepsilon \| \m_J\|_{H^{r-J, p}} \right) \\
&\le C (\varepsilon^{-(2-\sigma) \max(0,J-3)} + (1 + \varepsilon) \varepsilon^{-(2-\sigma) (J-2)})
\le C \varepsilon^{-(2-\sigma)(J-2)},
\end{align*}
and therefore we can bound $\varepsilon^{J-1} \|\boldsymbol \mu_1\|_{H^{r-J-2,p}}$ in the same way as the terms above,
\begin{align*}
  \varepsilon^{J-1}\|\boldsymbol \mu_1\|_{H^{r-J-2,p}}
\leq C\varepsilon^{J-1-(2-\sigma)(J-2)}=
 C\varepsilon^{1+(\sigma-1)(J-2)}.
\end{align*}

Consider now the cross-products $\m_{J+1+j-k} \times \V_{k}$ when
$j+1\leq k\leq J$ and $0\leq j\leq J-1$, which appear in the definition of $\boldsymbol \mu_2$.
By \Cref{lemma:rstvzestimate} we have that
\[\|\V_k\|_{H^{r-k, p}} \le C (1+\tau^{\max(0,k-2)}) \le C \varepsilon^{-(2-\sigma)\max(0, k-2)}\,.\]
We then use \cref{eq:prodest}
in
\Cref{lemma:bilinearest} with $q_0=r-J-1$, $q_1=r-J-1-j+k$ and
$q_2=r-k$ for the cross-product.  This choice is valid since $q_0\leq \min(q_1,q_2)$ and
$$
q_1+q_2 = r-J-1-j+r = q_0 + r-j\geq q_0 + J+2-j\geq q_0+3,
$$
which satisfies the left condition in \cref{eq:qjcond}.
Together with
\cref{eq:mjeps_estimate}, we thus get
\begin{align*}
\|\m_{J+1+j-k} \times \V_{k}\|_{H^{r-J-1,p}}
&\leq
C\|\m_{J+1+j-k}\|_{H^{r-J-1-j+k,p+2}} \|\V_{k}\|_{H^{r-k,p}}\\
&\leq C \varepsilon^{-(2-\sigma)\max(0,J+j-k-1)}
\varepsilon^{-(2-\sigma)\max(0, k-2)} \\
& \leq C \varepsilon^{-(2-\sigma)\max(0,J-2,J+j-3)}. 
\end{align*}
Exploiting the fact that \[-(2-\sigma)\max(0, J-2, J+j-3) = -(2-\sigma) (J-2 + \max(0, j-1)),
  \]
we hence find for the norm of $\boldsymbol \mu_2$ that
\begin{align*}
  \|\boldsymbol \mu_2\|_{H^{r-J-1,p}}
  & \le C \sum_{j=0}^{J-1} \sum_{k={j+1}}^J \varepsilon^j \|\m_{J+1+j-k} \times \V_k\|_{H^{r-J-1,p}}
    \\
& = C \varepsilon^{-(2-\sigma)(J-2)}
  \sum_{j=0}^{J-1}\varepsilon^{j - (2 - \sigma)\max(0, (j-1))}
  \,,
\end{align*}
and therefore obtain
\begin{align*}
  \varepsilon^{J-1}\|\boldsymbol \mu_2\|_{H^{r-J-1,p}}
  &\le C\varepsilon^{1+(\sigma-1)(J-2)}\left(1
  +\sum_{j=1}^{J-1}\varepsilon^{1+(\sigma-1)(j-1)}\right)
  \leq C\varepsilon^{1+(\sigma-1)(J-2)},
\end{align*}
where the last step is valid since, for $J\geq 3$,
$$
1+(\sigma-1)(j-1)\geq
1-\frac{j-1}{J-2}\geq
1-\frac{J-2}{J-2}= 0.
$$
We get the same estimate for  $\boldsymbol \mu_3$ upon considering instead $\m_{J+1+j-k} \times \T_{k}$.
Finally, note that multiplication by $\tilde{\m}_J$ does
not affect the results.
We can therefore use \Cref{lemma:bilinearest}
with the right condition in \cref{eq:qjcond}
together with 
\cref{eq:mtilde_estimate}
in \Cref{thm:mJnorms}, which yields
$$
  \|\tilde\m_J\times \boldsymbol \mu_1\|_{H^{r-J-2,p}}\leq C
  \|\tilde\m_J\|_{H^{r-J,p+2}}
  \|\boldsymbol \mu_1\|_{H^{r-J-2,p}}\leq
  C\|\boldsymbol \mu_1\|_{H^{r-J-2,p}},
$$
and thus
\begin{align*}
  \varepsilon^{J-1}\|\boldsymbol \eta_1\|_{H^{r-J-2, p}} &\le \varepsilon^{J-1}\left(\|\boldsymbol \mu_2\|_{H^{r-J-2, p}} + \|\tilde \m_J \times \boldsymbol\mu_1\|_{H^{r-J-2, p}}\right) \\
  &\le \varepsilon^{J-1}\left(\|\boldsymbol\mu_2\|_{H^{r-J-1, p}} + \|\tilde \m_J \times \boldsymbol\mu_1\|_{H^{r-J-2, p}}\right)
  \le C \varepsilon^{1 + (\sigma-1)(J-2)}.
\end{align*}
For the remaining terms we proceed similarly.

The $\|\cdot\|_{H^q_\varepsilon}$-norm estimate follows immediately
from the $\|\cdot\|_{H^q}$-estimate using \cref{eq:H_eps_bound} in
\Cref{lemma:H_eps}.
\end{proof}

\subsubsection{Length variation.}
While by assumption (A2), $|\m^\varepsilon| \equiv 1$ in space and
constant in time due to the norm preservation property of the
Landau-Lifshitz equation, \cref{eq:norm_pres}, the norm of the
approximation $\tilde \m^\varepsilon_J$ is not constant in time
since it does not satisfy \cref{eq:main_prob} exactly.  We now
consider the length of $\tilde \m_J^\varepsilon$ and obtain
an upper bound for its deviation from one, the length of $\m^\varepsilon$.

\begin{lemma}\label{lemma:m_tilde_length}
  Suppose $2\leq J\leq r-2$ and let $T^\varepsilon$ be defined as in
  \cref{eq:Teps2}.  Then for $0\leq t\leq T^\varepsilon$ and
  $0\leq q \leq r-J-1$,
\begin{align*}
  \| |\tilde \m_J^\varepsilon(\cdot,t)|^2-1\|_{H^q} &\le  C\varepsilon^{3+(\sigma-1)(J-2)-q},\qquad
                                                    \| |\tilde \m_J^\varepsilon(\cdot,t)|^2-1\|_{H^q_\varepsilon} &\le  C\varepsilon^{3+(\sigma-1)(J-2)},
\end{align*}
where the constant  $C$ is independent of $t$ and $\varepsilon$,
but depends on $T$.
\end{lemma}

This lemma implies by \cref{eq:H_eps_grad} in \Cref{lemma:H_eps} that
for $0 \le q \le r-J-2$ and $0 \le t \le T^\varepsilon$,
\begin{equation}\label{eq:m_tilde_length_alt}
\|\nabla |\m_J^\varepsilon|^2\|_{H^q_\varepsilon}
  = \|\nabla (|\m_J^\varepsilon|^2 - 1)\|_{H^q_\varepsilon}
  \le \varepsilon^{-1} \| |\m_J^\varepsilon|^2 - 1 \|_{H^{q+1}_\varepsilon}
  \le C \varepsilon^{2 + (\sigma-1)(J-2)}.
\end{equation}
\begin{proof}
We note first that since $\tilde{\m}_J^\varepsilon$ satsifies \cref{eq:eta_eq},
\[\boldsymbol{\eta}_J^\varepsilon \cdot \tilde \m_J^\varepsilon = \partial_t \tilde \m_J^\varepsilon \cdot \tilde \m_J^\varepsilon\,,\]
which together with  \cref{eq:eta} implies that
\begin{align}\label{eq:length_tderiv}
  \partial_t|\tilde{\m}_J^\varepsilon|^2 = 2\tilde{\m}_J^\varepsilon\cdot \boldsymbol{\eta}^\varepsilon_J
  &= \left.2\varepsilon^{J-1}\tilde{\m}_J^\varepsilon\cdot(\partial_t \m_{J-1} + \varepsilon \partial_t \m_J
+ \boldsymbol \eta_1 +  \sigma \boldsymbol\eta_2)\right|_{y=x/\varepsilon,\tau=t/\varepsilon^2} \nonumber \\
&=:\sum_{j=J-1}^{J'}\varepsilon^j d_j(x,x/\varepsilon,t,t/\varepsilon^2),
\end{align}
for some functions $d_j$ and integer $J'$. On the other hand, we can expand
$|\tilde \m_J^\varepsilon|^2$ as
\begin{align}\label{eq:length_expand}
  |\tilde \m_J^\varepsilon(x,t)|^2 = \left|\m_0(x,t) + \sum_{j=1}^J \varepsilon^j \m_j(x,x/\varepsilon,t,t/\varepsilon^2)\right|^2 =: \sum_{j=0}^{2J}
  \varepsilon^j c_j(x,x/\varepsilon,t,t/\varepsilon^2),
\end{align}
where
$$
   c_j = \sum_{k=\max(0,j-J)}^{\min(j, J)} \m_{k} \cdot \m_{j-k}\,.
   $$
   In particular, $c_0=|\m_0|^2\equiv 1$ and
   $c_1 = 2 \m_0 \cdot \m_1 \equiv 0$ due to the orthogonality of $\m_0$
   and $\m_1$ shown in \cref{eq:m1vorth} in \Cref{thm:mjestimate}.
By \cref{eq:length_tderiv}, the full time derivative of the first $J-2$ terms vanishes,
since
$$
   \left(\frac{\partial}{\partial t}+\varepsilon^{-2}\frac{\partial}{\partial \tau}\right)
   \sum_{j=0}^{2J}
  \varepsilon^j c_j = \sum_{j=J-1}^{J'}\varepsilon^j d_j.
$$
As this identity is valid for all $\varepsilon$, it holds that
$$
   \partial_t c_j + \partial_\tau c_{j+2}=0, \qquad j=0,\ldots, J-2.
$$
We claim that this implies that $c_j\equiv 0$ for $j=1,\ldots J$.
For $j = 1$ this is true due to \cref{eq:m1vorth}
in \Cref{thm:mjestimate} as shown above. Assume now that the claim
holds up to $j\leq J-1$. Then $j-1\leq J-2$,  and we thus have 
$$
  \partial_\tau c_{j+1}= -\partial_t c_{j-1} = 0,
$$
which is true also for $j=1$ since $\partial_tc_0=0$ as
$c_0 = |\m_0|^2 \equiv 1$ for all time by \cref{eq:norm_pres}.
Moreover, at time $\tau = 0$, $c_j(x,y,t,0)=0$ for $j\geq 1$ and all
$t\geq 0$, since this is true for the correctors $\m_j$.  Hence,
$c_{j+1}\equiv 0$. By induction we thus obtain
\begin{align*}
  |\tilde \m_J^\varepsilon(x,t)|^2 = 1 + \varepsilon^{J+1}\sum_{j=0}^{J-1}
  \varepsilon^j \tilde{c}_j(x,x/\varepsilon,t,t/\varepsilon^2),\qquad
\tilde{c}_j = c_{j+J+1} = \sum_{k=j+1}^{J} \m_{k} \cdot \m_{j+J+1-k}\,.
\end{align*}
%
Using \Cref{lemma:multiscaleest} it then follows that
\begin{align*}
  \| |\tilde \m_J^\varepsilon(\cdot,t)|^2-1\|_{H^q} \leq \varepsilon^{J+1-q}\sum_{j=0}^{J-1}
  \varepsilon^j \|\tilde{c}_j(\cdot,\cdot,t,t/\varepsilon^2)\|_{H^{q,q+2}}.
\end{align*}
We have left to estimate $\tilde{c}_j$ and note that it is of the same
type as the terms in the sum definining $\boldsymbol \mu_2$ in the proof of \Cref{thm:m_pert}.
Therefore, with the same steps as in that proof, we obtain
$$
\varepsilon^{J-1}\sum_{j=0}^{J-1}
  \varepsilon^j \|\tilde{c}_j(\cdot,\cdot,t,t/\varepsilon^2)\|_{H^{r-J-1,p}}
  \leq C\varepsilon^{1+(\sigma-1)(J-2)}.
$$
This finally gives
$$
\| |\tilde \m_J^\varepsilon(\cdot,t)|^2-1\|_{H^q}\leq
C\varepsilon^{3+(\sigma-1)(J-2)-q},
$$
for $0\leq q\leq r-J-1$ and the corresponding
$\|\cdot\|_{H^q_\varepsilon}$-norm estimate follows by
\cref{eq:H_eps_bound} in \Cref{lemma:H_eps}.
\end{proof}

\bibliographystyle{acm}
\bibliography{paper1}  

\if\longversion1
\appendix


\section{Estimates linear equation}\label{append:lin_eq}


In this section, we consider solutions $\m$ to the inhomogeneous
linear PDE given in
\cref{eq:lin_prob}
repeated here for convenience,
\begin{subequations} \label{eq:lin_prob_app}
\begin{align}
  \partial_\tau \m(x,y,t,\tau) &= \LL \m(x,y,t,\tau) + \F(x,y,t,\tau)\,, \\
  \m(x, y, t, 0) &= \g(x, y, t).
\end{align}
\end{subequations}
The PDE has periodic boundary conditions in $y$ and up to some fixed final time $T > 0$.
The linear operator $\LL$ is defined
as in \cref{eq:main_ll} as
$$  \LL \m_j := - \m_0 \times \L_2 \m_j - \alpha \m_0 \times \m_0 \times \L_2 \m_j, \qquad \text{where} \qquad \L_2 := \bnabla_y \cdot (a(y) \bnabla_y)\,.
$$
It depends on the material coefficient $a$ and on
the solution of the homogenized equation
$\m_0$.
We assume (A1), (A3) and (A5) hold.

Since $\LL$ has a non-trivial null space and $\alpha>0$
this is a degenerate parabolic equation in $(y,\tau)$.
It can also been viewed as a mixed parabolic hyperbolic system
where the hyperbolic part has zero advection.
As $\m_0$ is independent of $(y,\tau)$,
 standard theory
 ensures the existence of a unique smooth solution
 for all time
$$
 \m(x,\cdot,t,\cdot) \in C^1(\Real^+; H^{\infty}(Y)) ,
$$
 to \cref{eq:lin_prob}, if
$$
\F(x,\cdot,t,\cdot) \in C(\Real^+; H^\infty(Y)),\qquad
\g(x,\cdot,t) \in H^{\infty}(Y).
$$
See e.g. \cite[Chapter 6]{kreisslorenz}.

In this appendix we prove \Cref{thm:inhom_lin}
where we
obtain energy estimates
for $\m$ and its derivatives with respect to $x$, $y$
and $t$. The energy method can easily be used to show
that these derivatives grow at most exponentially fast in $\tau$.
However, since $\tau$ represents that fast scale, we need sharper bounds
on the growth.
The energy method must therefore be applied with more care.
Below we prove energy estimates that show the precise growth rate
in $\tau$.

As in \Cref{sec:lin_eq} we use the projections $\PP$ and $\QQ$ to
split the solution into a part that gets damped away with time
($\PP\m$) and a part that is invariant ($\QQ\m$).
We write again the definitions of the projections
\begin{align}\label{eq:MAdef}
  \M(x, t) :=  \m_0 \m_0^T\,, \qquad &\qquad  \AA \m := \int_Y \m(x, y, t, \tau) dy \,,\\
  \label{eq:PQdef}
  \PP\m :=
 (\I - \M) (\I - \AA) \m\,
\qquad & \qquad 
  \QQ \m := \M \m + (\I-\M) \AA \m \,.
\end{align}
It is easy to check that $\AA^2 = \AA$, the average of the average is equal to the average,
and $\M^2 = \m_0 \m_0^T \m_0 \m_0^T = \m_0 |\m_0|^2 \m_0^T = \M$, from which it directly follows that
$\PP^2 = \PP$ and $\QQ^2 = \QQ$.


To prove the theorem, we proceed as follows: in
\Cref{lemma:inhom_norm_estimate} we show that the first statement in \Cref{thm:inhom_lin}
holds for $p = q = 0$ when $\m$ is independent of $t$. In
\Cref{lemma:dy_norm_estimate} we extend this to general $p$ and in \Cref{lemma:dx_norm_estimate}
to general $q$. We then complete the proof by also taking the $t$-dependence of $\m$ into account.

\subsection{Projections.}
%
We first take a closer look at the properties of the projections $\PP$
and $\QQ$ as defined by \cref{eq:PQdef} and \cref{eq:MAdef}. We note that they depend
on $(x,t)$ but we mostly suppress this dependence in the notation.

We start with a lemma showing
that $\M$  belongs to the same space as $\m_0$, and that $\AA$ is bounded on $H^{q,p}$.

\begin{lemma}\label{lemma:projectionsApp}
If (A5) holds, then
\begin{equation}\label{eq:Mregularity}
  \partial_t^k\M\in C(0,T; H^{r-2k}(\Omega)), \qquad 0\leq 2k\leq r.
\end{equation}
The averaging operator $\AA$ is bounded on $H^{q,p}$ for all $q,p \ge 0$,
    \begin{equation}\label{eq:Abounded}
      ||\AA\v||_{H^{q, p}}\leq
      ||\v||_{H^{q, 0}}\leq
      ||\v||_{H^{q, p}},\qquad \forall \v\in H^{q,p}.
    \end{equation}
\end{lemma}
\begin{proof}
Since $\partial_t^k\m_0\in C(0,T; H^{r-2k}(\Omega))$
with $r\geq 5$ by (A5) 
it follows from \Cref{lemma:bilinearest}
with $q_0 = r - 2k, q_1 = r - 2l \ge q_0$ and $q_2 = r - 2k + 2l \ge q_0$ which is valid by \cref{eq:qjcond} since $q_1 + q_2 = r + q_0 \ge 5$,
 that
\begin{align*}
   ||\partial_t^k\M(\cdot,t)||_{H^{r-2k}}&\leq C
   \sum_{\ell=0}^k
   ||\partial_t^\ell\m_0(\cdot,t)||_{H^{r-2\ell}}
   ||\partial_t^{k-\ell}\m_0(\cdot,t)||_{H^{r-2k+2\ell}}\leq C\,.
\end{align*}
This shows \cref{eq:Mregularity}.
Next, we get \cref{{eq:Abounded}} since
\begin{align*}
||\AA\v||^2_{H^{q,p}}
&= \sum_{|\beta|\leq q,|\gamma|\leq p}\int_{\Omega}\int_Y \left|\partial_x^\beta
\partial_y^\gamma \int_Y \v(x,y)dy\right|^2 dxdy
= \sum_{|\beta|\leq q}\int_{\Omega}\left|\int_Y \partial_x^\beta\v(x,y)dy\right|^2 dx
\\
&\leq  \sum_{|\beta|\leq q}\int_{\Omega}\int_Y\left|\partial_x^\beta\v(x,y)\right|^2 dxdy=\|\v||^2_{H^{q,0}}.
\end{align*}
\end{proof}

Using this lemma we can now prove \Cref{lemma:projections1} and the
boundedness of the projections in $H^{q,p}$.  We use
\cref{eq:Abounded} and \Cref{lemma:bilinearest} with
$q_0 = q - 2k, q_1 = r-2k+2\ell \ge q_0, q_2 = q - 2\ell \ge q_0$
for $\ell \le k$, which again is valid by \cref{eq:qjcond} since
$q_1 + q_2 = q_0 + r$. This gives
\begin{align*}
||\partial_t^k\PP\v||_{H^{q-2k, p}}&=
   ||(\I-\AA) \partial_t^k (\I-\M)\v||_{H^{q-2k,p}}\leq C
   ||\partial_t^k (\I-\M)\v||_{H^{q-2k,p}}\\
   &\leq
   ||\partial_t^k\v||_{H^{q-2k,p}}
   + C\sum_{\ell=0}^k
   ||\partial_t^{k-\ell}\M(\cdot,t)||_{H^{r-2k+2\ell}}
   ||\partial_t^{\ell}\v(\cdot,t)||_{H^{q-2\ell,p}} \\
   &\leq
   C\sum_{\ell=0}^k
   ||\partial_t^{\ell}\v(\cdot,t)||_{H^{q-2\ell,p}}.
\end{align*}
The $\QQ$ part is given directly as $\QQ=\I-\PP$.
\Cref{lemma:projections1} is proved.

Further properties of the projections $\PP$ and $\QQ$, in particular
regarding commutations with derivatives, are given in the next
lemma.

\begin{lemma}\label{lemma:projections2}
  Let $\AA$ be the averaging operator defined by \cref{eq:MAdef} and $\PP$ and $\QQ$ be given by
\cref{eq:PQdef}. On the domain $H^1(Y)$, they have the following properties:
\begin{enumerate}
\item
The derivative of the average and the average of the derivative are zero,
\begin{align}\label{eq:dy_aa_zero}
  \partial_y \AA = \AA \partial_y = 0\,,
\end{align}
and
\begin{align}\label{eq:PQA}
  \PP \AA = \AA \PP = 0, \qquad \QQ \AA = \AA \QQ = \AA\,.
\end{align}
\item $\PP$ and $\QQ$ commute with $\partial_y$,
   \begin{align}\label{eq:pp_qq_commute}
     \partial_y \PP = \PP \partial_y\,, \qquad \partial_y \QQ = \QQ \partial_y \,,
   \end{align}
   and consequently for $\m\in H^2(Y)$,
   \begin{align}\label{eq:pp_ll}
     \PP \LL \m &= \LL \PP \m = \LL \m, &\quad
     \PP \L_2 \m &= \L_2 \PP \m, \\
     \QQ \LL \m &= \LL \QQ \m = \boldsymbol 0,&\quad
     \QQ \L_2 \m &= \L_2 \QQ \m \,.
   \end{align}
 \item Applying one of the projections $\PP$ and $\QQ$ to the
   $x$-derivative of the other results in an expression involving
   only the $x$-derivative of $\M$ and $\PP$,
   \begin{align*}
      \QQ \partial_x \PP = - (\partial_x \M) \PP  \qquad \text{and} \qquad
      \PP \partial_x \QQ = \PP (\partial_x \M)\,.
    \end{align*}
\end{enumerate}
\end{lemma}
\begin{proof}
We show the points in turn:
\begin{enumerate}
\item  To begin with, note that when $\m \in H^1(Y)$ it is $Y$-periodic in $y$,
and it follows that $\AA \partial_y \m = 0$.
 Moreover,
$\partial_y \AA \m = 0$ as well,
as $\AA \m$ is constant with respect to $y$.
Furthermore,
\begin{align*}
  \AA \PP \m &= (\I - \M) \AA (\I - \AA) \m =
  (\I - \M)(\I - \AA) \AA \m =
  \PP \AA \m \\ &=  (\I - \M)(\AA - \AA^2) \m = 0\,,
\end{align*}
since $(\I - \M)$ is independent of $y$. Similarly,
\begin{align*}
  \QQ \AA \m = \M \AA \m + (\I - \M) \AA^2 \m = \AA \m
\end{align*}
and $\AA \QQ \m = \AA \m$ follows in the same way.

\item  To prove the second claim, we use the fact that $\M$ is independent of
$y$. It therefore follows by \cref{eq:dy_aa_zero} that $\PP$
commutes with derivatives in $y$ and with $\L_2$,
\begin{align*}
  \PP \partial_y \m 
 &= (\I - \M) \partial_y \m - (\I - \M) \AA \partial_y \m
  =   (\I - \M) \partial_y\m \\&= \partial_y(\I - \M) (\I - \AA) \m = \partial_y \PP \m,
\end{align*}
and
\begin{align*}
\PP\L_2 \m &= (\I-\M)\L_2 \m
-(\I-\M)\AA\L_2 \m
= (\I-\M)\L_2 \m \\&=
(\I-\M) \L_2 (\I - \AA) \m =
  \L_2 \PP \m.
\end{align*}
The same holds for $\QQ$ since $\QQ=\I-\PP$.
Moreover, for any vector $\u \in \Real^3$, the definition of $\M$ implies that
\[\m_0 \times \M \u = 0, \qquad \M (\m_0 \times \u) = 0\,. \]
This property combined with the commutation above
entails that
\begin{align*}
  \m_0 \times \L_2 \PP \m =
  \m_0 \times \PP\L_2 \m =
  \m_0 \times (\I-\M) (\I - \AA) \L_2 \m = \m_0 \times \L_2 \m,
\end{align*}
and
\begin{align*}
 \PP( \m_0 \times \L_2\m) =
(\I-\M)  \m_0 \times  (\I - \AA) \L_2 \m = \m_0 \times \L_2 \m.
\end{align*}
A similar argument for the damping term implies that $\PP$ and
$\LL$ commute as stated in \cref{eq:pp_ll}.  As $\QQ = \I - \PP$, it
then follows directly that $\QQ \LL \m = \LL \QQ \m = 0$.

\item To show the last statement, note that
\begin{align}\label{eq:M_x_der}
  \M (\partial_x \M) = \partial_x \M^2 - (\partial_x \M) \M = \partial_x \M (\I - \M)\,.
\end{align}
As $\AA$ is independent of $x$, \cref{eq:M_x_der} yields together with
\cref{eq:PQA} and
$\QQ \PP \m = 0$, that
\begin{align*}
  \QQ \partial_x (\PP \m) &= 
 \QQ \left[\PP \partial_x \m - (\partial_x \M)(\I - \AA) \m\right]
 = - \M (\partial_x \M)(\I - \AA) \m
 = - (\partial_x \M) \PP \m
\end{align*}
and, since $\PP = \PP (\I-\AA)$ by \cref{eq:PQA},
\begin{align*}
  \PP \partial_x (\QQ \m) &= 
\PP [\QQ \partial_x \m + (\partial_x\QQ) \m]  = \PP (\partial_x \M)(\I - \AA) \m
  = \PP (\partial_x \M) \m\,.
\end{align*}
\end{enumerate}
\end{proof}

\subsection{Estimates for $\QQ$ part.}
%
Suppose $\m$ satsifies \cref{eq:lin_prob_app}.
By \cref{eq:pp_ll} in \Cref{lemma:projections2}
we then have for $\QQ\m$,
$$
   \partial_\tau (\QQ\m) = \QQ\LL\m + \QQ\F =\QQ\F, \qquad \QQ\m(x,y,t,0) =\QQ\g(x,y,t).
$$
Hence,
\[
\QQ \m(x,y, t,\tau) = \QQ \m(x,y, t,0) + \int_0^\tau \QQ \F(x,y, t,s) ds
= \QQ \g(x,y,t) + \int_0^\tau \QQ \F(x,y, t,s) ds \,.
\]
Therefore, \cref{eq:inhom_linQtime} follows directly.

\subsection{Estimates for $\PP$ part: $y$-derivatives.}
%
In this section we consider $\PP\m$ assuming that $\F$ and $\g$ only
depend on the fast variables.  We start by proving $L^2$-norm energy
estimates for $\PP\m(y, \tau)$ and its gradient in this setting.
Those results can subsequently be used to prove estimates for higher
derivatives. The proof of the lemma involves the following two
variations of the Poincar\'e-inequality.  Let $\u_p$ be an arbitrary, $Y$-periodic
function in $H^p(Y;\Real^3)$. Then the we have
\begin{equation}\label{eq:poincare}
  \|\u_1 - \bar  \u_1 \|_{L^2} \le C_P \|\bnabla \u_1 \|_{L^2},\qquad
     \|\bnabla \u_2\|_{L^2} \le \frac{C_P}{a_\mathrm{min}} \|\L_2 \u_2\|_{L^2}\,,
\end{equation}
where $\bar \u_1$ is the average of $\u_1$ over $Y$,
$a_\mathrm{min}=\inf_{y\in Y}a(y)$ and $C_P$ only depends on
$Y$. The left statement in \cref{eq:poincare} is the standard
Poincar\'e inequality. The right one follows since due to periodicity,
\begin{align*}
  \| \bnabla \u_2\|_{L^2}^2
  &\le \frac{1}{a_\mathrm{min}} \int_Y a \bnabla \u_2 : \bnabla \u_2 dx
    = - \frac{1}{a_\mathrm{min}} \int_Y \u_2 \cdot \L_2 \u_2 dx \\
  &= -  \frac{1}{a_\mathrm{min}} \int_Y (\u_2 - \bar \u_2) \cdot \L_2 \u_2 dx
    \le \frac{1}{a_\mathrm{min}} \|\u_2 - \bar \u_2\|_{L^2} \|\L_2 \u_2\|_{L^2}\,,
\end{align*}
which, upon application of the left Poincar\'e inequality, yields the right statement.

\begin{lemma}\label{lemma:inhom_norm_estimate}
  Assume that $\m\in C^1(\Real^+; H^{\infty}(Y))$ is a solution to
  \begin{subequations} 
  \begin{align}
    \partial_\tau \m(y,\tau) &=  \LL \m(y, \tau) + \F(y,\tau)\,, \\
    \m(y, 0) &= \g(y) \,,
  \end{align}
  \end{subequations}
where  $\F \in C(\Real^+; H^\infty(Y))$ and
$\g \in H^{\infty}(Y)$.
Then it holds for some constant $\gamma > 0$ that
  \begin{align}\label{eq:m_inhom_est0}
    \|\PP\m(\cdot, \tau)\|_{L^2} \le e^{-\gamma\tau} \|\PP\g\|_{L^2}
    + \int_0^\tau  e^{-\gamma(\tau - s)}\|\PP \F(\cdot, s)\|_{L^2}ds\,.
  \end{align}
Moreover,  there is a constant $C$ and $\tilde\gamma\geq \gamma$ such that
  \begin{align}\label{eq:m_inhom_est_deriv}
 \| \bnabla_y \PP \m(\cdot, \tau)\|_{L^2} \le
 C\left(e^{-\tilde\gamma \tau} \| \bnabla_y  \PP\g\|_{L^2}
    + \int_0^\tau e^{-\tilde\gamma(\tau - s)}\| \bnabla_y  \PP\F(\cdot, s)\|_{L^2} ds
    \right)\,.
 \end{align}
The constants $C$, $\gamma$ and $\tilde\gamma$ are independent of $\tau$, $x$, $t$, $\F$ and $\g$.
\end{lemma}

\begin{proof}
  Suppose first that $\F\equiv 0$.  From \cref{eq:pp_ll} in
  \Cref{lemma:projections2} it follows that $\m_\perp := \PP \m$
  then satisfies
 \begin{equation}\label{eq:m_perp1}
   \begin{split}
    \partial_\tau \m_\perp(y, \tau)  &= \LL \m_\perp(y, \tau) , \\
    \m_\perp(y, 0) &= \PP \g(y) \,.
   \end{split}
\end{equation}
 Multiplication of \cref{eq:m_perp1} by $\m$ and integration over $Y$ yields,
  due to the triple product identites \cref{eq:first_triple_product_identity} and
  \cref{eq:second_triple_product_identity} and $\m_0$ being independent of $y$,
  \begin{align*}
    \int_Y \m_\perp \cdot \partial_\tau \m_\perp dy
    &=  \m_0 \cdot \int_Y \m_\perp \times \L_2 \m_\perp dy - \alpha \int \L_2(\m_0 \cdot \m_\perp) (\m_0 \cdot  \m_\perp) dy \\ &\qquad+ \alpha  \int  \m_\perp \cdot \L_2 \m_\perp dy\,.
  \end{align*}
The first two terms here become zero, due to the integral identity \cref{eq:int_m_lm_0} and since
$\m_\perp$ is orthogonal to $\m_0$,  $\m_\perp \cdot \m_0 = 0$.
Using integration by parts and the Poincar\'e inequality \cref{eq:poincare} then gives
\begin{align*}
  \frac{1}{2} \partial_\tau \|\m_\perp\|_{L^2(Y)}^2
  &=  - \alpha  \int a \nabla_y \m :  \nabla_y \m dy \le - \alpha a_\mathrm{min}\|\nabla_y \m_\perp\|_{L^2(Y)}^2
    \le
    -\gamma \|\m_\perp\|^2_{L^2(Y)},
\end{align*}
where $\gamma = \alpha a_\mathrm{min}/C_P^2$. Here we also used the fact
that the average of $\m_\perp$ is zero.
Via the Gr\"onwall inequality 
we then get
  \begin{align}\label{eq:Stau1}
    \|S^\tau \PP\g\|_{L^2(Y)}^2= \|\m_\perp(\cdot, \tau)\|^2_{L^2(Y)} \le e^{-2\gamma\tau} \|\m_\perp(\cdot, 0)\|^2_{L^2(Y)}  =  e^{-2\gamma\tau} \|\PP \g\|^2_{L^2(Y)}\,,
  \end{align}
  where $S^\tau$ is the solution operator of \cref{eq:m_perp1}.

To show the second statement, \cref{eq:m_inhom_est_deriv}, we can differentiate \cref{eq:m_perp1} and get
 \begin{equation*}
    \begin{split}
      \partial_\tau \bnabla_y \m_\perp(y, \tau) &= \bnabla_y \LL \m(y, \tau),
 \\
      \bnabla_y \m_\perp(y, 0) &= \bnabla_y \PP\g(y) \,.
    \end{split}
 \end{equation*}
We multiply this by
$a \bnabla_y \m_\perp$ and integrate over $Y$. Using integration by
parts as in \cref{eq:int_m_lm_0} as well as   \cref{eq:first_triple_product_identity} and
  \cref{eq:second_triple_product_identity}, then yields
 \begin{align*}
   \frac{1}{2} \partial_\tau  \|a^{1/2} \bnabla_y \m_\perp\|^2_{L^2}
&= \int_Y \L_2 \m_\perp \cdot  (\m_0 \times \L_2 \m_\perp) dy +
\alpha \int_Y \L_2 \m_\perp  \cdot (\m_0 \times \m_0 \times \L_2 \m_\perp) dy \\
&= \int_Y \m_0 \cdot  (\L_2 \m_\perp \times \L_2 \m_\perp) dy +
\alpha \int_Y (\m_0 \cdot \L_2 \m_\perp)^2 - |\L_2 \m_\perp|^2 dy \\&= - \alpha \|\L_2 \m_\perp\|^2_{L^2(Y)}\,,
  \end{align*}
 due to the orthogonality of $\m_\perp$ and $\m_0$.
 We can now use the second version of the Poincar\'e inequality in \cref{eq:poincare}
 to show that
   \begin{align*}
    \frac12 \partial_\tau \|a^{1/2} \bnabla_y \m_\perp\|^2 \le -\frac{\alpha a_{\min}^2}{C_P^2} \|\bnabla_y \m_\perp\|^2 \le
-\frac{\alpha a_{\min}^2}{C_P^2 a_\mathrm{max}} \|a^{1/2} \bnabla_y \m_\perp\|^2\,
  \end{align*}
  and obtain by Grönwall's inequality,
   \begin{align*}
    \|a^{1/2} \bnabla_y \m_\perp\|^2 \le  e^{-2\tilde\gamma \tau} \|a^{1/2} \PP \bnabla_y \g\|^2\,, \qquad \tilde\gamma =  \frac{a_{\min}}{a_\mathrm{max}}\gamma.
\end{align*}
The resulting estimate for the gradient of $\m_\perp$ is then
  \begin{align}\label{eq:Stau2}
\|\bnabla_y S^\tau \PP\g\|_{L^2} &= \|\bnabla_y \m_\perp(\cdot, \tau)\|_{L^2} \le \sqrt{\frac{a_\mathrm{max}}{a_\mathrm{min}}} e^{-\tilde\gamma \tau} \|\bnabla_y \PP \g\|_{L^2}\,.
  \end{align}
  The estimates \cref{eq:m_inhom_est0} and
  \cref{eq:m_inhom_est_deriv} now follow from \cref{eq:Stau1},
  \cref{eq:Stau2} and Duhamel's principle
\[\m(y, \tau) = S^\tau \g + \int_0^\tau S^{\tau-s} \F(y, s) ds.\]
This concludes the proof.
\end{proof}

We now continue to show the estimates for an arbitrary number of
$y$-derivatives of $\m$, from which the following $H^p$-norm
estimate follows.

\begin{lemma}\label{lemma:dy_norm_estimate}
  Assume that $\m\in C^1(\Real^+; H^{\infty}(Y))$ is a solution to
  \begin{subequations} \label{eq:m_inhom_y}
  \begin{align}
    \partial_\tau \m(y,\tau) &=  \LL \m(y, \tau) + \F(y,\tau)\,, \\
    \m(y, 0) &= \g(y) \,,
  \end{align}
  \end{subequations}
where  $\F \in C(\Real^+; H^\infty(Y))$ and
$\g \in H^{\infty}(Y)$.
Then, for each integer $p\geq 0$,
 there are constants $C$ and $\gamma>0$, independent of $\tau$, $x$, $t$, $\F$ and $\g$,
 such that
  \begin{align}\label{eq:m_inhom_est}
    \|\PP\m(\cdot, \tau)\|_{H^p} \le
    C\left(e^{-\gamma\tau} \|\PP\g\|_{H^p}
    + \int_0^\tau  e^{-\gamma(\tau - s)}\|\PP \F(\cdot, s)\|_{H^p} ds\,\right).
  \end{align}
\end{lemma}

\begin{proof}
  Suppose $p=2k+\ell$ with $\ell\in \{0,1\}$ and let
  $\u := \L_2^k \m$.  Then $\u$ satisfies a differential equation
  with the same structure as in \cref{eq:m_inhom_y},
  \begin{subequations} \label{eq:m_inhom_L}
    \begin{align}
    \partial_\tau \u &= \LL \u  + \L_2^k \F, \\
    \u(y, 0) &= \L_2^k \g(y) \,.
  \end{align}
  \end{subequations}
  One therefore obtains from \Cref{lemma:inhom_norm_estimate} that
 \begin{align}\label{eq:m_inhom_est_L}
\|\PP\L_2^k\m(\cdot, \tau)\|_{L^2}=\|\PP\u(\cdot, \tau)\|_{L^2} \le e^{-\gamma\tau} \|\PP\L_2^k\g\|_{L^2}
    + \int_0^\tau  e^{-\gamma(\tau - s)}\|\PP \L_2^k\F(\cdot, s)\|_{L^2}ds\,.
  \end{align}
By \cref{eq:pp_ll} in \Cref{lemma:projections2} we know that
$\PP$ and $\L_2^k$ commute, which together with (A2), $a \in C^\infty$, entails that
  there is a constant $C$ such that for any $\v\in H^{2k}(Y)$,
  \begin{align*}
    \|\PP \L_2^k \v\|_{L^2} = \|\L_2^k \PP \v\|_{L^2} \le C \|\PP \v\|_{H^{2k}}\,.
  \end{align*}
  Hence, the right-hand side in \cref{eq:m_inhom_est_L} is bounded
  by the right-hand side in \cref{eq:m_inhom_est}.  By elliptic
  regularity, it holds for an arbitrary function $\u_p$ in
  $H^p(Y;\Real^3)$ that
\begin{equation}\label{eq:elliptic_reg_a}
  \|\u_{2p}\|_{H^{2p}} \le C \sum_{j=0}^p \|\mathcal{L}_2^j \u_{2p}\|_{L^2}\,,
\qquad  \|\u_{2p+1}\|_{H^{2p+1}} \le C \Big(\|\mathcal{L}_2^{p}\u_{2p+1}\|_{H^1} + \sum_{j=0}^{p-1} \|\mathcal{L}_2^j \u_{2p+1}\|_{L^2}\Big)\,,
\end{equation}
where $C$ are constants independent of the functions $\u_{2p}$ and $\u_{2p+1}$.

For even $p$ (when $\ell=0$), it then follows from the elliptic regularity
\cref{eq:elliptic_reg_a}, that
  \begin{align*}
    \|\PP\m(\cdot, \tau)\|_{H^{2k}} \le C\sum_{j=0}^k \|\L_2^j \PP\m(\cdot,\tau)\|_{L^2},
  \end{align*}
  and consequently, we obtain \cref{eq:m_inhom_est}.  For odd $p$,
  when $\ell=1$, we use \cref{eq:m_inhom_est_deriv} in
  \Cref{lemma:inhom_norm_estimate} applied to \cref{eq:m_inhom_L},
  which yields
    \begin{align*}
      \| \bnabla_y \PP\L_2^k\m(\cdot, \tau)\|_{L^2} &=\| \bnabla_y\PP\u(\cdot, \tau)\|_{L^2} \\
   &\le
   C\left(e^{-\tilde\gamma \tau} \|\bnabla_y  \PP \L_2^k\g\|_{L^2}
    + \int_0^\tau e^{-\tilde\gamma(\tau - s)}\|\bnabla_y  \PP \L_2^k\F(\cdot, s)\|_{L^2} ds
\right)\,.
 \end{align*}
This gives the extra $H^1$-norm
 needed to use the second variant of elliptic regularity
in \cref{eq:elliptic_reg_a}. Combined with the fact that
$\|\PP \bnabla \L_2^k \v\|_{L^2} \le C \|\PP \v\|_{H^{2k+1}}$
for $\v\in H^{2k+1}(Y)$ we obtain \cref{eq:m_inhom_est} also for odd $p$.
\end{proof}


\subsection{Estimate for $\PP$ part: $x$-derivatives.}
%
%
%
To obtain an estimate for the norms of $x$-derivatives of $\m$
satsfying \cref{eq:lin_prob_app}, we again aim to split the function we
want to estimate. In the process, the order in which partial
derivatives are taken matters. To handle this, let
$\beta = [\beta_1, \beta_2, \ldots, \beta_n]$ be an
{\it ordered} multi-index with $n$ elements. Each of these elements
$\beta_j$, $j = 1, ..., n$, specifies one of the coordinate
directions. The length of any ordered multi-index is the number of
elements it contains, $|\beta| = n$.  We call $\eta$ an ordered
subset of $\beta$, denoted $\eta \subset \beta$, if it contains any
selection of elements in $\beta$ where the order of elements is not
changed compared to their order in $\beta$. Each $\eta$ is again an
ordered multi-index.

Moreover, we want to be able to extend a given ordered multi-index
with additional coordinate directions. This is denoted as follows:
if $\beta$ is as above and
we add an additional coordinate direction $k$, then we write
\[\bar\beta = [\beta; k] := [\beta_1, \beta_2, ..., \beta_n, k],\qquad
|\bar\beta|=n+1.
\]
When using the ordered multi-index for denoting partial derivatives we define
for $\f\in H^{|\beta|}(\Omega)$,
\begin{align*}
  \partial_x^{\beta} \f := \partial_{x_{\beta_n}}\partial_{x_{\beta_{n-1}}} \cdots\ \partial_{x_{\beta_1}} \f.
\end{align*}
We also introduce the operator $D$, where the partial
derivatives are interlaced with the projection $\PP$,
$$
D^\beta \f :=
\PP \partial_{x_{\beta_n}} \PP \partial_{x_{\beta_{n-1}}} \cdots\ \PP \partial_{x_{\beta_1}}\PP \f, \qquad D^0 \f:= \PP\f.
$$
This operator is useful since if $\m$ satisfies \cref{eq:lin_prob_app} then
 $D^\beta \m$
satisfies the same PDE where $\F$ and $\g$
are replaced by $D^\beta \F$ and $D^\beta \g$. We can then
proceed to estimate $D^\beta \m$
using \Cref{lemma:dy_norm_estimate}. We get the following lemma.
\begin{lemma}\label{lemma:dx_norm_estimate}
  Let $\beta$ be an ordered multi-index.
 Suppose $D^\beta\F(x,\cdot,t,\cdot)\in C(\Real^+; H^\infty(Y))$
 and $D^\beta\g(x,\cdot,t)\in H^\infty(Y)$.
 Then, for each integer $p\geq 0$,
 there are constants $C$ and $\gamma>0$, independent of $\tau$, $x$, $t$, $\F$ and $\g$,
 such that
  \begin{align*}
     \| D^\beta \m(x, \cdot, t, \tau)\|_{H^p(Y)}
     &\le
     C \bigg(e^{-\gamma \tau} \|D^\beta\g(x, \cdot, t)\|_{H^p(Y)} \\&\qquad \quad \left. + \int_0^\tau e^{-\gamma (\tau - s)}\|D^\beta \F(x, \cdot, t, s)\|_{H^p(Y)} ds \right) \,.
  \end{align*}
\end{lemma}
\begin{proof}
We first show that $\w_\beta(x, y, t, \tau) :=D^\beta\m(x, y, t, \tau)$ satisfies
  \begin{subequations}\label{eq:w_beta_LL}
\begin{align}
      \partial_\tau \w_\beta(x, y, t, \tau) &= \LL \w_\beta(x, y, t, \tau) + D^\beta \F(x, y, t, \tau), \\
      \w_\beta(x,y, t, 0) &= D^\beta \g(x, y, t),
\end{align}
  \end{subequations}
given that $\m$ satisfies \cref{eq:lin_prob_app}. We use induction and
consider first the case
$\beta = 0$. Then based on \Cref{lemma:projections2}, we have
\[
\partial_\tau \w_0 =
 \partial_\tau \PP \m = \LL \PP\m + \PP \F = \LL \w_0 + \PP \F \,.
\]
In general, assume that \cref{eq:w_beta_LL} holds for some $\beta$ and consider $\bar\beta := [\beta; k]$. As
$\w_{\bar\beta} = \PP \partial_{x_k} \w_\beta$, it satisfies
\begin{align*}
  \partial_\tau \w_{\bar\beta}
  &= \PP \partial_{x_k} \partial_\tau  \w_\beta
    = \PP \partial_{x_k} (\LL \w_\beta + D^\beta \F) \\
  &= \LL \PP\partial_{x_k} \w_\beta - \PP\left(\partial_{x_k} \m_0 \times \L_2 \w_\beta\right)  \\ & \qquad + \alpha \PP\left( \partial_{x_k} \m_0 \times \m_0 \times \L_2 \w_\beta
    + \m_0 \times \partial_{x_k} \m_0 \times \L_2 \w_\beta \right)  + D^{\bar\beta} \F,
\end{align*}
as $\LL$ and $\PP$ commute, as shown in \Cref{lemma:projections2}.
 Since $\m_0 \perp \partial_{x_k} \m_0$
for any $k$ given that $|\m_0| \equiv 1$, and by the definition of $\PP$,
$\m_0 \perp \L_2 \w_\beta$, we can conclude that $\partial_{x_k} \m_0 \times \L_2 \w_\beta = \kappa \m_0$
for some $\kappa \in \Real$ and thus
$\PP (\partial_{x_k} \m_0 \times \L_2 \w_\beta) = \kappa \PP \m_0 =0$.
Moreover,
\begin{align*}
  \m_0 \times \partial_{x_k} \m_0 \times \L_2 \w_\beta = \kappa \m_0 \times \m_0 = 0 \,,
\end{align*}
and
by the triple product identity \cref{eq:second_triple_product_identity},
\begin{align*}
  \PP (\partial_{x_k} \m_0  \times \m_0 \times  \L_2 \w_\beta) =
  \PP (\partial_{x_k} \m_0 \cdot \L_2 \w_\beta) \m_0 - \PP(\m_0 \cdot \partial_{x_k} \m_0) \L_2 \w_\beta = 0\,.
\end{align*}
This proves the claim that $\w_\beta$ satisfies \cref{eq:w_beta_LL} by induction.

Since $\w_\beta$ satisfies \cref{eq:w_beta_LL} and $\QQ \w_\beta = 0$,
it follows from \Cref{lemma:dy_norm_estimate} that
\begin{align*}
  \|\w_\beta\|_{H^p(Y)}
 &\le C \left(e^{-\gamma \tau} \|D^\beta\g\|_{H^p(Y)} + \int_0^\tau e^{-\gamma (\tau - s)}\|D^\beta \F\|_{H^p(Y)} ds \right).
\end{align*}
This concludes the proof.
\end{proof}

For the final estimate of $\m$
we  also need to exploit the fact that the usual
Sobolev norms of $\PP\m$ are equivalent to
Sobolev norms based on $D$, as follows.

\begin{lemma}\label{lemma:m_dx_sum}
  Assume \cref{eq:m0regularity} and let $0\leq q \leq r$.
  Then for all $\f\in H^q(\Omega;\Real^3)$ and $0\leq t \leq T$,
  $$
     C_0 ||\PP \f||^2_{H^{q}(\Omega)}\leq
     \sum_{|\beta|\leq q}||D^\beta \f||^2_{L^2(\Omega)}
     \leq
     C_1 ||\PP \f||^2_{H^{q}(\Omega)},
  $$
  where $C_0$ and $C_1$ are independent of $t$.
  \end{lemma}
\begin{proof}
In the first step we prove that
for
 each ordered multi-index $\beta$ with $|\beta|\leq q\leq r$
  there are $\c_{\eta, \beta}\in C(0,T; H^{r-|\beta|+|\eta|}(\Omega))$
such that
\begin{align}\label{eq:dx_w_sum}
    \partial_x^{\beta} \PP\f = D^\beta \f +
     \sum_{\substack{\eta \subset \beta \\ |\eta| < |\beta|}}
    \c_{\eta, \beta} D^\eta \f.
\end{align}
We again use induction.
The statement is trivially true for $\beta=0$.
Assume now that
\cref{eq:dx_w_sum} holds for all $\beta$ with $0\leq |\beta|<q$ and that the corresponding coefficient functions $\c_{\eta, \beta} \in C(0, T; H^{r - |\beta| + |\eta|}(\Omega))$.
Let $|\beta|=q'<q$ and consider the ordered multi-index $\bar\beta=[\beta;j]$.
From \cref{eq:dx_w_sum} we then get, using \Cref{lemma:projections2}, that
\begin{align*}
  \partial_x^{\bar\beta} \PP\f
  &= \partial_{x_j}\left(
  D^\beta \f +      \sum_{\substack{\eta \subset \beta \\ |\eta| < |\beta|}}
    \c_{\eta, \beta} D^\eta \f\right)
  = \PP\partial_{x_j}
  D^\beta \f +
  \QQ\partial_{x_j}
  D^\beta \f
   +      \sum_{\substack{\eta \subset \beta \\ |\eta| < |\beta|}}
    (\partial_{x_j}\c_{\eta, \beta}) D^\eta \f
    \\
&\hspace{15 mm}   + \sum_{\substack{\eta \subset \beta \\ |\eta| < |\beta|}}
    \c_{\eta, \beta} \PP\partial_{x_j}D^\eta \f
   + \sum_{\substack{\eta \subset \beta \\ |\eta| < |\beta|}}
    \c_{\eta, \beta} \QQ\partial_{x_j}D^\eta \f\\
  &=
  D^{\bar\beta} \f -  (\partial_{x_j} \M)
  D^\beta \f
   + \sum_{\substack{\eta \subset \beta \\ |\eta| < |\beta|}}
    (\partial_{x_j}\c_{\eta, \beta}) D^\eta \f\\
&\hspace{15 mm}    + \sum_{\substack{\eta \subset \beta \\ |\eta| < |\beta|}}
    \c_{\eta, \beta} D^{[\eta;j]} \f
   - \sum_{\substack{\eta \subset \beta \\ |\eta| < |\beta|}}
    \c_{\eta, \beta} (\partial_{x_j} \M)D^\eta \f.
\end{align*}
We note that all terms to the right of $D^{\bar\beta} \f$
involve ordered multi-indices of length at most $|\beta|$,
which are subsets of $\bar\beta$.
Moreover, the order of their elements is the same as
their order in $\bar\beta$. Thus, $\partial_x^{\bar \beta} \PP \f$ satisfies \cref{eq:dx_w_sum}
for some coefficient functions $\c_{\eta, \bar \beta}$.

Finally, we need to show that the coefficient functions
belong to $C(0,T; H^{r-|\bar\beta|+|\eta|}(\Omega))$. This is the case since
\begin{align*}
\partial_{x_j} \M &\in C(0,T; H^{r-1}(\Omega))=C(0,T; H^{r-|\bar\beta|+|\beta|}(\Omega)),\\
\partial_{x_j}\c_{\eta, \beta}
&\in C(0,T; H^{r-|\beta|+|\eta|-1}(\Omega))=C(0,T; H^{r-|\bar\beta|+|\eta|}(\Omega)),\\
\c_{\eta, \beta}
&\in C(0,T; H^{r-|\beta|+|\eta|}(\Omega))=C(0,T; H^{r-|\bar\beta|+|[\eta;j]|}(\Omega)).
\end{align*}
For the last term we use
\Cref{lemma:bilinearest}, with $q_0 = r - |\bar \beta| + |\eta|$, $q_1 = r - |\beta| + |\eta|$  and $q_2 = r - 1$. This is a valid choice
since $|\beta| < |\bar \beta|$, which implies that
$q_0\leq\min(q_1,q_2)$ and
 $q_1+q_2=q_0+|\bar\beta|-|\beta|+r-1\geq q_0+r$.
We get
$$
||\c_{\eta, \beta}(\cdot,t) \partial_{x_j} \M(\cdot,t)||_{H^{r-|\bar\beta|+|\eta|}}
\leq C
||\c_{\eta, \beta} (\cdot,t)||_{H^{r-|\beta|+|\eta|}}
||\partial_{x_j}\M(\cdot,t)||_{H^{r-1}} \leq C.
$$
This proves the first claim.

We next note that the boundedness of $\PP$ in \cref{eq:PQtbounded} implies that
for $\eta=[j]$ and $q'\leq q$,
$$
  ||D^\eta\f||^2_{H^{q'-1}}=
  ||\PP\partial_{x_j}\PP\f||_{H^{q'-1}}\leq
  C||\partial_{x_j}\PP\f||_{H^{q'-1}}\leq
  C||\PP\f||^2_{H^{q'}}.
$$
By induction it follows that for general $|\eta|\leq q'\leq q$,
\begin{equation}\label{eq:Dest}
||D^\eta\f||_{H^{q'-|\eta|}}\leq C ||\PP\f ||_{H^{q'}}.
\end{equation}
This shows the right inequality in the lemma upon taking $q'=|\eta|$.

Furthermore, from \cref{eq:dx_w_sum} and \cref{eq:Dest} we get
\begin{align*}
||\PP \f||^2_{H^{q}}
&=\sum_{|\beta|\leq q} ||\partial_x^{\beta} \PP\f||^2_{L^2}
\leq C\sum_{|\beta|\leq q} \left(||D^\beta \f||^2_{L^2} +
     \sum_{\substack{\eta \subset \beta \\ |\eta| < |\beta|}}
    ||\c_{\eta, \beta} D^\eta \f||^2_{L^2}
    \right)
    \\
    &\leq C
    \sum_{|\beta|\leq q} ||D^\beta \f||^2_{L^2}+
    C\sum_{|\beta|\leq q}
     \sum_{\substack{\eta \subset \beta \\ |\eta| < |\beta|}}
    ||\c_{\eta, \beta}||^2_{H^{r-|\beta|+|\eta|}}
    ||D^\eta \f||^2_{H^{|\beta|-|\eta|-1}}
    \\
    &\leq C
    \sum_{|\beta|\leq q} ||D^\beta \f||^2_{L^2}+
    C
     \sum_{|\beta|\leq q}
    ||\PP \f||^2_{H^{|\beta|-1}}\leq
        C\sum_{|\beta|\leq q} ||D^\beta \f||^2_{L^2}
+C||\PP \f||^2_{H^{q-1}}.
\end{align*}
To bound the $L^2$-norm of $\c_{\eta, \beta} D^\eta \f$
we used \Cref{lemma:bilinearest} with $q_0 = 0$,
$q_1=r-|\beta|+|\eta| \ge 0$ and $q_2=|\beta|-|\eta|-1 \ge 0$, which is valid
since $q_1 + q_2 = r - 1 \ge q_0+4$.
The left inequality in the lemma now follows by induction.
\end{proof}

We have left to put everything together to show \Cref{thm:inhom_lin}
for $k = 0$, which is the estimate
 \begin{align}\label{eq:inhom_linP}
   \| \PP\m(\cdot, \cdot, t, \tau)\|_{H^{q, p}} &\le C \left(e^{- \gamma \tau} \|\PP \g(\cdot, \cdot, t)\|_{H^{q, p}}
+ \int_0^\tau e^{-\gamma(\tau - s)} \|\PP \F(\cdot, \cdot, t, s)\|_{H^{q, p}}ds \right).
\end{align}
First we
note that as a consequence of \Cref{lemma:m_dx_sum}, it holds for
$\f\in H^{q,p}$, when $0\leq q\leq r$ and $p\geq 0$, that
\begin{align}\label{eq:PDest}
C_0 ||\PP\f||_{H^{q,p}}^2
\leq \sum_{|\beta|\leq q} ||D^\beta\f||_{H^{0,p}}^2
\leq C_1 ||\PP\f||_{H^{q,p}}^2.
\end{align}
This is true since
\begin{align*}
C_0||\PP\f||_{H^{q,p}}^2 &=
C_0\sum_{|\alpha|\leq p}\int_Y ||\PP\partial_y^\alpha\f(\cdot, y)||_{H^{q}}^2dy
\leq
\sum_{|\alpha|\leq p}
\sum_{|\beta|\leq q} \int_Y ||D^\beta\PP\partial_y^\alpha \f(\cdot, y)||_{L^2}^2 dy \\&=
\sum_{|\beta|\leq q} ||D^\beta\f||_{H^{0,p}}^2,
\end{align*}
where we used the fact that $D\partial_y^\alpha = \partial_y^\alpha D$.
The second inequality follows in the same way.
Then by \Cref{{lemma:dx_norm_estimate}},
\begin{align*}
\lefteqn{||\PP\m(\cdot,\cdot,t,\tau)||_{H^{q,p}}^2
\leq C
\sum_{|\beta|\leq q} ||D^\beta\m(\cdot,\cdot,t,\tau)||_{H^{0,p}}^2 = C
\sum_{|\beta|\leq q}\int_\Omega ||D^\beta\m(x,\cdot,t,\tau)||_{H^{p}}^2dx}
\hskip 2 mm &  \\
     &\le C
     \sum_{|\beta|\leq q}\int_\Omega
     e^{-2\gamma \tau} \|D^\beta\g(x, \cdot, t)\|^2_{H^p} +
          C\sum_{|\beta|\leq q}\int_\Omega
\left(\int_0^\tau e^{-\gamma (\tau - s)}\|D^\beta \F(x, \cdot, t, s)\|_{H^p} ds \right)^2
     dx.
  \end{align*}
  For the first term we have from \cref{eq:PDest},
\begin{align*}
\sum_{|\beta|\leq q}\int_\Omega
     e^{-2\gamma \tau} \|D^\beta\g(x, \cdot, t)\|^2_{H^p}dx
     &=
     e^{-2\gamma \tau}
          \sum_{|\beta|\leq q}
 \|D^\beta\g(\cdot, \cdot, t)\|^2_{H^{0,p}}\leq C     e^{-2\gamma \tau}
 \|\PP\g(\cdot, \cdot, t)\|^2_{H^{q,p}}.
  \end{align*}
  For the second term we use the general estimates
  $$
  \sum_n ||f_n||_{L^1}^2 \leq
    \left\|\left(\sum_n |f_n(x)|^2 \right)^{1/2}\right\|_{L^1}^2,\qquad
        \int ||f(x,\cdot)||_{L^1}^2dx \leq
    \left( \int ||f(\cdot,y)||_{L^2}dy\right)^2,
  $$
  which, again together with \cref{eq:PDest}, give
\begin{align*}
\lefteqn{
\sum_{|\beta|\leq q}
\int_\Omega
\left(\int_0^\tau e^{-\gamma (\tau - s)}\|D^\beta \F(x, \cdot, t, s)\|_{H^p} ds \right)^2dx}\hskip 15 mm & \\
     &\leq
\int_\Omega\left(\int_0^\tau e^{-\gamma (\tau - s)}
\left(\sum_{|\beta|\leq q}
\|D^\beta \F(x, \cdot, t, s)\|^2_{H^p}\right)^{1/2} ds \right)^2dx\\
     &\leq
\left(\int_0^\tau e^{-\gamma (\tau - s)}
\left(\sum_{|\beta|\leq q}
\|D^\beta \F(\cdot, \cdot, t, s)\|^2_{H^{0,p}}\right)^{1/2} ds \right)^2
\\
      &\leq C
\left(\int_0^\tau e^{-\gamma (\tau - s)}
\|\PP\F(\cdot, \cdot, t, s)\|^2_{H^{q,p}} ds \right)^2.
  \end{align*}
This finally shows \cref{eq:inhom_linP}.

\subsection{Estimate for $\PP$ part: time derivative.}
In this final part of the proof of \Cref{thm:inhom_lin}, an estimate
for $\|\partial_t^k \PP\m(\cdot, \cdot, t, \tau)\|_{H^{q, p}}$ is
derived. We proceed in a similar way as in the previous
section and define
$$
  D_t^k = (\PP\partial_t)^k \PP, \qquad k \ge 0,
$$
and let $\w_k = D_t^k\m$.
Then
\begin{align*}
  \partial_\tau \w_k(x, y, t, \tau) &= \LL \w_k(x, y, t, \tau) + D_t^k \F(x, y, t, \tau)\,, \\
  \w_k(x,y,t, 0) &= D_t^k \g(x,y,t)\,.
\end{align*}
This can be shown
by an argument similar to the one in the proof of \Cref{lemma:dx_norm_estimate}.
Since $\QQ \w = 0$,
\Cref{eq:inhom_linP} then gives
%
 \begin{align}\label{eq:num}
     \| \w_k(\cdot, \cdot, t, \tau)\|_{H^{q, p}}  &\le C \left(e^{- \gamma \tau}
     \|D^k_t \g(\cdot, \cdot)\|_{H^{q, p}} + \int_0^\tau e^{-\gamma(\tau - s)} \|D_t^k \F(\cdot, \cdot, t, \tau)\|_{H^{q, p}} ds\right).
 \end{align}
We also need the following lemma,
which corresponds to \Cref{lemma:m_dx_sum} in the previous section.

\begin{lemma}\label{lemma:m_dt_sum}
  Assume \cref{eq:m0regularity} and
  suppose $\partial_t^\ell\f \in H^{q-2\ell}(\Omega;\Real^3)$
  for $0\leq 2\ell\leq 2k\leq q\leq r$ and $p\geq 0$. Then
  $$
     C_0\sum_{\ell=0}^k ||\partial_t^\ell\PP \f||_{H^{q-2\ell,p}}\leq
     \sum_{\ell=0}^k||D^\ell_t \f||_{H^{q-2\ell,p}}
     \leq
     C_1\sum_{\ell=0}^k ||\partial_t^\ell\PP \f||_{H^{q-2\ell,p}},
     \qquad t\in[0,T],
  $$
  where $C_0$ and $C_1$ are independent of $t$.
  \end{lemma}
\begin{proof}
In the first step we prove that
for
 each $k$ with $0\leq q\leq r-2k$
  there are functions $\c_{\ell, k}$ satisfying
  $$
\partial_t^s\c_{\ell, k}\in C(0,T; H^{r-2k+2\ell-2s}(\Omega)),\qquad r-2k+2\ell-2s\geq 0,
  $$
for which
\begin{align}\label{eq:dt_w_sum}
    \partial_t^{k} \PP\f = D^k_t \f +
     \sum_{\ell=0}^{k-1}
    \c_{\ell, k} D_t^\ell \f.
\end{align}
We use induction.
The statement is trivially true for $k=0$.
Assume that
\cref{eq:dt_w_sum} holds upto $k$
and that $2(k+1) \leq r$.
From \cref{eq:dt_w_sum} we then get
\begin{align*}
  \partial_t^{k+1} \PP\f
  &= \partial_{t}\left(
  D_t^k \f +          \sum_{\ell=0}^{k-1}
    \c_{\ell, k} D_t^\ell \f\right)\\
  &= \PP\partial_{t}
  D^k_t \f +
  \QQ\partial_{t}
  D^k_t \f
   +      \sum_{\ell=0}^{k-1}
    (\partial_{t}\c_{\ell, k}) D^\ell_t \f
   + \sum_{\ell=0}^{k-1}
    \c_{\ell, k} \PP\partial_{t}D^\ell_t \f
   + \sum_{\ell=0}^{k-1}
    \c_{\ell, k} \QQ\partial_{t}D^\ell_t \f\\
  &=
  D^{k+1}_t \f -  (\partial_{t} \M)
  D^k_t \f
   +\sum_{\ell=0}^{k-1}
    (\partial_{t}\c_{\ell, k}) D^\ell_t \f
   + \sum_{\ell=0}^{k-1}
    \c_{\ell, k} D^{\ell+1} \f
   - \sum_{\ell=0}^{k-1}
    \c_{\ell, k} (\partial_{t} \M)D^\ell_t \f.
\end{align*}
For the coefficient functions
we have by \Cref{lemma:projectionsApp}
\begin{align*}
\partial_{t}^{s+1} \M &\in C(0,T; H^{r-2-2s}(\Omega))=C(0,T; H^{r-2(k+1)+2k-2s}(\Omega)),\\
\partial_{t}^{s+1}\c_{\ell, k}
&\in C(0,T; H^{r-2k+2\ell-2(s+1)}(\Omega))=C(0,T; H^{r-2(k+1)+2\ell-2s}(\Omega)),\\
\partial_{t}^{s}\c_{\ell, k}
&\in C(0,T; H^{r-2k+2\ell-2s}(\Omega))=C(0,T; H^{r-2(k+1)+2(\ell+1)-2s}(\Omega)),
\end{align*}
and by \Cref{lemma:bilinearest},
\begin{align*}
||\partial_t^s(\c_{\ell, k}(\cdot,t) \partial_{t} \M(\cdot,t))||_{H^{q_0}}
&\leq C\sum_{\eta=0}^s
\left\|[\partial_t^{\eta}\c_{\ell, k}(\cdot,t)] [\partial^{s-\eta+1}_{t} \M(\cdot,t)]
\right\|_{H^{q_0}}\\
&\leq C\sum_{\eta=0}^s
\left\|\partial_t^{\eta}\c_{\ell, k}(\cdot,t)\right\|_{H^{q_{1,\eta}}}
\left\|\partial^{s-\eta+1}_{t} \M(\cdot,t)
\right\|_{H^{q_{2,\eta}}},
\end{align*}
where
$$
q_0 = r-2(k+1)+2\ell-2s,\qquad
 q_{1,\eta}=r-2k+2\ell-2\eta, \qquad q_{2,\eta}=r-2s+2\eta-2.
$$
To check that \Cref{lemma:bilinearest} can indeed be used we note that
$$
q_0= q_{1,\eta}-2(s-\eta+1)=q_{2,\eta}-2(\eta+k-\ell)
\leq \min(q_{1,\eta},q_{2,\eta}),
\quad \text{and} \quad
q_1 + q_2 = q_0 + r\,.
$$
This proves  \cref{eq:dt_w_sum}.
We next show that  for $0\leq \ell\leq k$,
\begin{align}\label{eq:DtPt}
||D_t^\ell \f||_{H^{q-2\ell,p}}
&\leq
C\sum_{j=0}^{\ell} ||\partial_t^{j} \PP\f||_{H^{q-2j,p}}.
\end{align}
This is true for $\ell=0$, and if true for $\ell<k$, then by
\cref{eq:PQtbounded},
\begin{align*}
||D_t^{\ell+1} \f||_{H^{q-2-2\ell,p}}
&\leq
C\sum_{j=0}^{\ell} ||\partial_t^{j} D_t\f||_{H^{q-2-2j,p}}
=
C\sum_{j=0}^{\ell}  ||\partial_t^{j} \PP \partial_t \PP\f||_{H^{q-2-2j,p}}
\\
&\leq
C\sum_{j=0}^{\ell} \sum_{s=0}^{j}  ||\partial_t^{s+1} \PP\f||_{H^{q-2-2s,p}}
\leq
C\sum_{j=0}^{\ell+1} ||\partial_t^{j} \PP\f||_{H^{q-2j,p}}.
\end{align*}
Hence, \cref{eq:DtPt} follows by induction. This shows the right inequality
in the lemma upon summing over $\ell$.
Next, \cref{eq:DtPt} and \cref{eq:dt_w_sum}
give us the following estimate
for $0\leq 2k\leq q$,
\begin{align*}
    ||\partial_t^{k} \PP\f||_{H^{q-2k,p}} &\leq  ||D^k_t\f||_{H^{q-2k,p}}  +
     \sum_{\ell=0}^{k-1}
    ||\c_{\ell, k}||_{H^{r-2k+2\ell}} ||D_t^\ell \f||_{H^{q-2\ell,p}}
    \\
    &\leq ||D^k_t\f||_{H^{q-2k,p}}  +
     \sum_{\ell=0}^{k-1}
    ||D_t^\ell \f||_{H^{q-2\ell,p}} \\ &\leq
    ||D^k_t\f||_{H^{q-2k,p}}  +
    C\sum_{\ell=0}^{k-1} ||\partial_t^{\ell} \PP\f||_{H^{q-2\ell,p}},
\end{align*}
where \Cref{lemma:bilinearest} was used again, this time with
$q_0 = q - 2k$, $q_1 = r - 2k + 2l \ge q_0$, $q_2 = q - 2l \ge q_0$
for $l < k$, and $q_1+q_2=r+q_0$. The left inequality in the lemma
follows by induction.
\end{proof}

Two applications of \Cref{lemma:m_dt_sum} and \cref{eq:num}
now lead to
 \begin{align*}
 \|\partial_t^k\PP\m(\cdot, \cdot, t, \tau)\|_{H^{q-2k, p}}&\leq
C
\sum_{\ell=0}^{k}
  \| \w_\ell\|_{H^{q-2\ell, p}}  \\
  &\le C
  \sum_{\ell=0}^{k}
   \left(e^{- \gamma \tau}
     \|D^\ell_t \g\|_{H^{q-2\ell, p}} + \int_0^\tau e^{-\gamma(\tau - s)} \|D_t^\ell \F\|_{H^{q-2\ell, p}} ds\right)\\
  &\le C
  \sum_{\ell=0}^{k}
   \left(e^{- \gamma \tau}
     \|\partial_t^\ell\PP\g\|_{H^{q-2\ell, p}} + \int_0^\tau e^{-\gamma(\tau - s)} \|\partial_t^\ell\PP\F\|_{H^{q-2\ell, p}} ds\right).
 \end{align*}
\Cref{thm:inhom_lin} is proved.
\section{Energy estimates nonlinear equation}

In this appendix, we derive energy estimates for the Landau-Lifshitz
equation with a material coefficient. We consider first a
matrix-valued, constant coefficient, which is relevant for the
homogenized equation \cref{eq:main_hom}, and then the case of a
highly oscillatory, scalar coefficient, as in
\cref{eq:main_prob}. These estimates are generalizations of an
estimate in \cite{melcher} and can be seen as a step on the way to
prove existence of solutions to the Landau-Lifshitz equation with
material coefficient.

\subsection{Constant matrix-valued coefficient.} \label{appendix1}
Here we derive a priori estimates for the solution of the Landau--Lifshitz Gilbert
equation when
$$
\H(\m) = {\mathcal L}\m, \qquad Lu =
\nabla\cdot(\A\nabla u), \qquad \L \m =
\begin{bmatrix}
  L m^{(1)}, L m^{(2)}, L m^{(3)}
\end{bmatrix}^T,
$$
with $\A$ being a constant, symmetric positive definite matrix.
To obtain the estimates we assume that an $L^\infty$ bound on the
gradient $\bnabla\m$ is given. 

To simplify notation in this section, we define a norm on
$\Real^{3 \times n}$, weighted by the coefficient matrix
$\A \in \Real^{n \times n}$. For a matrix-valued function $\B \in \Real^{3 \times n}$, let
 \begin{equation*}
   |\B|^2_A := \B : (\B \A)\,,
 \end{equation*}
 and the corresponding $L^2$-norm on $\Omega \to \Real^{3 \times n}$ is given by
\begin{align*}
  \|\B\|^2_{0,A} := \int_\Omega |\B(x)|^2_A dx = \int_\Omega \B : (\B \A) dx\,.
\end{align*}
We furthermore consider several identities that are useful in the
following proof.  First, note that since the cross product of a
vector by itself is zero, the cross product of a vector $\v$ and
$\L \v$ can be rewritten as
\begin{equation}
    \label{eq:identity1_abis}
    \v \times {\mathcal L}\v  = \nabla \cdot (\v \times (\nabla \v \A))\,.
  \end{equation}
  For the special case of a vector-function $\v$ with constant
  length throughout the domain, $|\v(x)| \equiv \const$, the dot
  product between $\v$ and its gradient is zero,
  $\v \cdot \bnabla \v = \boldsymbol 0$.  Then it follows from the
  vector triple produce identity
  \cref{eq:second_triple_product_identity} that
\begin{equation}
    \label{eq:identity2_abis}
    - \v \times \v \times {\mathcal L}\v =
    |\v|^2{\mathcal L}\v + (\bnabla \v: \bnabla \A) \v\, .
  \end{equation}
  Moreover, \cref{eq:first_triple_product_identity} implies that
\begin{equation}
    \label{eq:orthobis}
    \B : (\v\times\B)\A = \boldsymbol 0.
  \end{equation}

We can then prove the following energy estimate.
\begin{theorem}\label{thm:nonlinear_matrix}
Suppose $\m \in \mathcal{C}^0((0, T); H^q(\Real^n))$ is a solution of
\begin{equation}
  \label{eq:llg_damping}
  \partial_t \m = - \m \times {\mathcal L} \m - \alpha \m \times \m \times {\mathcal L} \m\,,
\end{equation}
where $0<\alpha\leq 1$ and $L=\nabla\cdot(\A\nabla)$ with a constant,
symmetric positive definite, matrix $\A$.
Given any integer $2 \le \sigma \le q$ and $|\m(x,0)|\equiv 1$,
there exists a constant $c > 0$ such that
\begin{equation}
  \label{eq:der_bound}
  \|\bnabla \m(\cdot,t)\|^2_{H^{\sigma-1}} \le e^{C(t)} \|\bnabla \m(0,\cdot)\|_{H^{\sigma-1}}^2\,,\qquad
  C(t) = c \int_0^t
  \left(\alpha + \frac1\alpha\|\bnabla \m(s,\cdot)\|_{L^\infty}^2\right) ds.
\end{equation}
The constant $c$  is independent of $\m(x,0)$, $\alpha$ and $t$, but
depends on $\sigma$ and $\A$.

\end{theorem}

\begin{proof}
   For any given multi-index $1 \le |\beta| \le \sigma$, let
  \begin{equation}
    \label{eq:highest_der_only}
    \partial^\beta (\m \times \bnabla \m) = \m \times \partial^\beta \bnabla \m + \mathbf{R}\,,
  \end{equation}
  where $\mathbf{R}$ is given by
\[\mathbf{R} =
\sum_{0 < \nu \le \beta}
\binom{\beta}{\nu}  \partial^\nu \m \times \partial^{\beta - \nu} \bnabla \m .
\]
Writing $\nu=\nu_0+\nu_1$ with $|\nu_0|=1$ and using \cref{eq:interpol_ineq}
as in \cite{melcher}, we
obtain the following bound,
\begin{align}
  \label{eq:bound1}
  \|\mathbf R\|_{L^2} &\le c
  \sum_{\stackrel{0\leq \nu_1 \le \beta-\nu_0}{|\nu_0|=1}}
             \|\partial^{\nu_0} \m\|_{L^\infty}\,
             \|\bnabla \m\|_{H^{|\beta|-1}}
+ \|\bnabla \m\|_{L^\infty} \|\partial^{\nu_0}\m\|_{H^{|\beta|-1}} \nonumber \\
&\leq
  c \|\bnabla \m\|_{L^\infty} \|\bnabla \m\|_{H^{\sigma - 1}}\,.
\end{align}


When applying $\partial^\beta$ to \cref{eq:llg_damping} and multiplying by $\partial^\beta \m$,
one gets
\[ \partial^\beta \m \cdot \partial_t \partial^\beta \m =
  - \partial^\beta \m \cdot \partial^\beta (\m \times {\mathcal L}\m +
  \alpha \m \times \m \times {\mathcal L}\m) \,.
\]
  Integration over $\Omega$ and application of the vector identities
  \cref{eq:identity1_abis} and \cref{eq:identity2_abis} then gives
\begin{align*}
   \frac{1}{2} \partial_t \|\partial^\beta \m\|^2_{L^2}
   &= \int_\Omega \partial^\beta \bnabla \m : \partial^\beta (\m \times
         \bnabla \m \A) dx
         - \alpha \int_\Omega \partial^\beta \bnabla \m :
         \partial^\beta  \bnabla \m \A dx  \\ &\qquad+ \alpha \int_\Omega \partial^\beta \m \cdot
         \partial^\beta ( |\bnabla \m|_A^2 \m) dx \,.
\end{align*}
Together with \cref{eq:highest_der_only}
and the fact that $\m \times (\bnabla \m \A) =(\m \times \bnabla \m) \A$,
we obtain due to the orthogonality \cref{eq:orthobis} that
 \begin{align*}
  \frac{1}{2} \partial_t \|\partial^\beta \m\|_{L^2}^2
  &= - \alpha \|\partial^\beta \bnabla \m\|_{0,A}^2 + \int_\Omega
       \partial^\beta \bnabla \m : (\m \times \partial^\beta \bnabla \m +
        \mathbf R )\A dx
        \\& \qquad+ \alpha \int_\Omega \partial^\beta \m \cdot
          \partial^\beta(|\bnabla \m|_A^2 \m) dx \\
  & = - \alpha \|\partial^\beta \bnabla \m\|_{0,A}^2 + \int_\Omega
        \partial^\beta \bnabla \m : \mathbf RA dx
        +\alpha \int_\Omega \partial^\beta \m \cdot
        \partial^\beta(|\bnabla \m|_A^2 \m) dx.
\end{align*}
Application of  Cauchy-Schwarz then yields
\begin{align}\label{eq:after_CS}
   \frac{1}{2} \partial_t \|\partial^\beta \m\|_{L^2}^2 + \alpha \|
   \partial^\beta \bnabla \m\|_{0,A}^2
     &\le   \|\partial^\beta \bnabla \m \|_{L^2} \|\mathbf R \|_{L^2} +
   \alpha \|\partial^\beta \m\|_{L^2} \|\partial^\beta
      (|\bnabla \m|_A^2 \m)\|_{L^2}\,.
\end{align}


An estimate for $\|\partial^\beta (\m |\bnabla \m|_A^2)\|_{L^2}$ can
be achieved similarly to the estimate for $\|\mathbf{R}\|_{L^2}$.
We rewrite the derivative of the product as a sum,
  \begin{align*}
    \partial^\beta(\m |\bnabla \m|_A^2 ) = \m \,\partial^\beta(|\bnabla \m|_A^2) +
    \sum_{0< \nu \le \beta} \binom{\beta}{\nu} \partial^\nu \m  \, \partial^{\beta - \nu} |\bnabla \m|_A^2\,.
  \end{align*}
  Since $|\m| \equiv 1$, the $L^2$-norm of the first term here can
  be bounded in terms of $\|\, |\bnabla \m|_A^2\, \|_{H^\sigma}$,
  and the norm of second term can be estimated in the same way as
  $\mathbf{R}$, which results in a bound in terms of
  $\|\,|\bnabla \m|_A^2\,\|_{H^{\sigma-1}}$.  To estimate the norm
  of $|\bnabla\m|_A^2$ that appears in these two bounds we can use
  \cref{eq:interpol2} to obtain
\begin{equation}\label{eq:nabla_m_square_bound}
  \|\,|\bnabla \m|_A^2\,\|_{H^{\ell}} \le
   C \|\bnabla \m\|_{L^\infty}
     \|\bnabla \m\|_{H^{\ell}}, \qquad 0 \le \ell \le q-1.
\end{equation}
Hence we can bound
  \begin{equation}\label{eq:bound2}
    \begin{split}
    \|\partial^\beta(\m|\bnabla \m|_A^2 )\|_{L^2} &\le
C \left(\|\m\|_{L^\infty}  \|\, |\bnabla \m|_A^2\, \|_{H^\sigma} + \|\bnabla\m\|_{L^\infty}\|\,|\bnabla \m|_A^2\,\|_{H^{\sigma-1}}\right) \\
    &\le C\left( \|\bnabla \m\|_{L^\infty} \|\bnabla \m \|_{H^\sigma} + \|\bnabla \m\|^2_{L^\infty} \|\bnabla \m\|_{H^{\sigma -1}}\right)\,.
    \end{split}
  \end{equation}
Applying the bounds \cref{eq:bound1} and \cref{eq:bound2} to the right-hand side of \cref{eq:after_CS}  results in
\begin{align*}
\frac{1}{2} \partial_t \|\partial^\beta \m\|_{L^2}^2 &+ \alpha \|\partial^\beta \bnabla \m\|_{0,A}^2 \\
&\le
C\left( \|\bnabla \m\|_{L^\infty} \|\bnabla \m\|_{H^{\sigma - 1}}\|\bnabla \m\|_{H^\sigma} +
\alpha \|\bnabla \m \|^2_{L^\infty}  \|\bnabla \m\|^2_{H^{\sigma -1}}\right),
\end{align*}
for $1 \le |\beta| \le \sigma$.
 In order to obtain an estimate for $\partial_t \|\bnabla \m\|^2_{L^2}$, we now sum over all the
multi-indices $1 \le |\beta| \le \sigma$
as well as the zeroth order term $\|\bnabla\m\|_{0,A}^2$,
which yields
\begin{align*}
\frac{1}{2} \partial_t \|\bnabla\m\|_{H^{\sigma-1}}^2 &+ A_\mathrm{ min}\alpha \|\bnabla \m\|_{H^\sigma}^2 \\
&\leq \sum_{|\beta|=1}^\sigma
\left(\frac{1}{2} \partial_t \|\partial^\beta\m\|_{L^2}^2 +
\alpha\|\partial^\beta \bnabla \m\|_{0,A}^2\right)
+\alpha\|\bnabla \m\|_{0,A}^2
\\
&\leq
c_0 \|\bnabla \m\|_{L^\infty} \|\bnabla \m\|_{H^{\sigma - 1}}\|\bnabla \m\|_{H^\sigma} +
c_1\alpha \left(\|\bnabla \m \|^2_{L^\infty}+1\right)  \|\bnabla \m\|^2_{H^{\sigma -1}}.
\end{align*}
Here $A_\mathrm{ min}$ is the lower bound for $\A$, satisfying
$\x^T\A\x\geq A_\mathrm{ min} |\x|^2$ for all $\x\in\Real^n$.
We next apply Young's inequality, for any $\gamma>0$
\[
   \|\bnabla \m\|_{L^\infty} \|\bnabla \m\|_{H^{\sigma - 1}}
   \|\bnabla \m\|_{H^\sigma} \le
   \frac{1}{2\gamma} \|\bnabla \m\|^2_{L^\infty}
    \|\bnabla \m\|_{H^{\sigma - 1}}^2 +
   \frac{\gamma}{2} \|\bnabla \m\|^2_{H^\sigma}\,,
\]
and, upon choosing $\gamma=A_\mathrm{ min}\alpha/c_0$, we get
\begin{align*}
\frac{1}{2} \partial_t \|\bnabla \m\|_{H^{\sigma - 1}}^2 + \frac{1}{2} A_\mathrm{ min}\alpha \|\bnabla \m\|_{H^\sigma}^2
& \le c \left(\alpha+\frac1\alpha \|\bnabla \m\|^2_{L^\infty} \right)
 \|\bnabla \m\|^2_{H^{\sigma -1}}.
\end{align*}
The statement \cref{eq:der_bound} then follows from
 Gr\"onwall's inequality. 
\end{proof}

\subsection{Variable scalar coefficient.}\label{appendix2}
Given $L = \bnabla \cdot (a^\varepsilon \bnabla)$, where
$a^\varepsilon = a(x/\varepsilon)$ is a scalar, variable coefficient
satisfying the assumption (A1),
one can proceed in a
similar way as in the previous section to obtain an energy estimate.
This setup corresponds to \cref{eq:main_prob} and we can therefore
use the lemmas derived in \Cref{sec:utility}. 

\begin{theorem}\label{lemma:l2_m_eps_scalar}
  Consider $\m \in C^1(0, T; H^q(\Omega))$ satisfying \cref{eq:main_prob} under the assumptions $(A1)-(A3)$.
Then it holds that 
\begin{align}
    \|\bnabla \m(\cdot, t)\|_{L^2} \le C  \|\bnabla \m(\cdot, 0)\|_{L^2}, \qquad 0 \le t \le T.
\end{align}
If there is a constant $M$ independent of $\varepsilon$ such that
$\|\bnabla \m(\cdot, t)\|_\infty \le M$ for $0 \le t \le T$, then
it is moreover true that
\begin{align}\label{eq:Lk_bound}
  \|\m(\cdot, t)\|_{H^q} \le C \left(\|\m(\cdot, 0)\|_{L^2} + \frac{1}{\varepsilon^{q-1}} \|\bnabla \m(\cdot, 0)\|_{H^{q-1}_\varepsilon}\right) \le C \left(1 + \varepsilon^{1-q}\right),
\end{align}
where the constant $C$ is independent of $\varepsilon$ and $t$ but depends on $T$ and $M$.
\end{theorem}

\begin{proof}
  To prove this lemma, we first show by induction that for $0 \le j \le q-1$ and $0 \le t \le T$,
  \begin{align}\label{eq:energy1}
    \|\bnabla \m(\cdot, t)\|_{H^j} \le C \frac{1}{\varepsilon^{j}} \|\bnabla \m(\cdot, 0)\|_{H^j_\varepsilon},
  \end{align}
  which by the definition of $\|\cdot\|_{H^j_\varepsilon}$ entails that
  \begin{align}\label{eq:energy2}
    \|\bnabla \m(\cdot, t)\|_{H^j_\varepsilon} \le C \|\bnabla \m(\cdot, 0)\|_{H^j_\varepsilon}.
  \end{align}
  To begin with the proof, consider the $L^2$-norm of $\bnabla
  \m$. By the identity \cref{eq:second_triple_product_identity} it
  holds that
\begin{align*}
  \frac{1}{2} \partial_t \|\sqrt{a^\varepsilon} \bnabla \m\|^2
&= \int_\Omega a^\varepsilon \bnabla \m : \bnabla \partial_t \m dx
= - \int_\Omega \L \m \cdot \partial_t \m dx \\
&=   \int_\Omega \L \m \cdot (\m \times \L \m + \alpha \m \times \m \times \L \m)
=   \alpha \int_\Omega \L \m \cdot (\m \times \m \times \L \m) \\
&=   \alpha \int_\Omega |\L \m \cdot \m|^2 - |\L \m|^2 dx \le 0,
\end{align*}
as $|\L \m \cdot \m|^2 \le |\L \m|^2 |\m|^2 = |\L \m|^2$. Due to the
boundedness of $a(x)$, (A1), this entails
\begin{align}
  \|\bnabla \m(\cdot, t)\|^2_{L^2} \le C \|\bnabla \m(\cdot, 0)\|^2_{L^2}\,,
\end{align}
which shows \cref{eq:energy1} for $q = 0$.
Next, we proceed similarly as in \Cref{sec:nonlinear} and consider
$\|\sqrt{a^\varepsilon} \bnabla \L^k \m\|_{L^2}$ and $\|\L^{k} \m\|_{L^2}$ for
$k \ge 1$. Using the identity
\[-\m \times \m \times \L\m = \L \m + a^\varepsilon |\bnabla \m|^2,\]
which corresponds to \cref{eq:identity2_abis},
and applying integration by parts, we get
\begin{align*}
  \frac{1}{2} \partial_t \|\sqrt{a^\varepsilon} \bnabla \L^k \m\|_{L^2}^2
  &= - \int_\Omega a^\varepsilon \bnabla \L^k \m : \bnabla \L^k (\m
    \times \L \m - \alpha \L \m + \alpha a^\varepsilon \m |\bnabla \m|^2) dx \\
  &= \int_\Omega \L^{k+1} \m \cdot \L^k (\m \times \L \m + \alpha a^\varepsilon \m |\bnabla \m|^2) dx
    - \alpha \|\L^{k+1} \m\|_{L^2}^2
\end{align*}
and similarly,
\begin{align*}
  \frac{1}{2} \partial_t \|\L^{k+1} \m\|_{L^2}^2
  &= - \int_\Omega \L^{k+1} \m \cdot \L^{k+1} (\m \times \L \m - \alpha \L \m + \alpha a^\varepsilon\m|\bnabla \m|^2) dx \\
  &= \int_\Omega a^\varepsilon \bnabla \L^{k+1} \m : \bnabla \L^k (\m \times \L \m + \alpha a^\varepsilon \m |\bnabla \m|^2) dx \\
   &\qquad \qquad
    - \alpha \|\sqrt{a^\varepsilon} \bnabla \L^{k+1} \m\|_{L^2}^2\,.
\end{align*}
We then use \Cref{lemma:nonlinear_utility3} to rewrite
\[\L^k (\m \times \L \m) = \m \times \L^{k+1} \m + \R_k,\]
where the highest order terms cancel in the integral due to
orthogonality. Application of Cauchy-Schwarz and Young's inequality thus gives
that for any $\gamma > 0$,
\begin{align*}
  \int_\Omega \L^{k+1} \m \cdot \L^k (\m \times \L \m) dx
  &
    \le \frac{\gamma}{2} \|\L^{k+1} \m\|^2_{L^2} + \frac{1}{2\gamma} \|\mathbf{R}_k\|_{L^2}^2
  \\
  \int_\Omega a^\varepsilon \bnabla \L^{k+1} \m : \bnabla \L^k (\m \times \L \m) dx
  &
    \le \frac{\gamma}{2}\|\sqrt{a^\varepsilon}  \bnabla \L^{k+1} \m\|_{L^2}^2 \\& \qquad+ \frac{C}{2\gamma} (\|\bnabla \m\|_{L^\infty}^2\|\L^{k+1} \m\|^2_{L^2} +  \|\bnabla \R_k\|_{L^2}^2).
\end{align*}
Similarly, we obtain that
\begin{align*}
  \int_\Omega \L^{k+1} \m \cdot \L^k (a^\varepsilon \m |\bnabla \m|^2) dx &\le \frac{\gamma}{2} \|\L^{k+1} \m\|^2_{L^2} + \frac{1}{2\gamma} \|\L^k (a^\varepsilon \m |\bnabla \m|^2)\|_{L^2}^2, \\
  \int_\Omega a^\varepsilon \bnabla \L^{k+1} \m : \bnabla \L^k (a^\varepsilon \m |\bnabla \m|^2) dx &\le \frac{\gamma}{2}\|\sqrt{a^\varepsilon}  \bnabla \L^{k+1} \m\|_{L^2}^2 + \frac{C}{2\gamma} \|\bnabla \L^k (a^\varepsilon \m |\bnabla \m|^2)\|_{L^2}^2.
\end{align*}
Choosing $\gamma$ sufficiently small and making use of the assumed bound on $\|\bnabla \m\|_{L^\infty}$, it thus follows that
\begin{subequations} \label{eq:B15}
\begin{align}
  \partial_t \|\sqrt{a^\varepsilon} \bnabla \L^k \m\|_{L^2}^2 &\le \frac{C}{2\gamma} \left(\|\mathbf{R}_k\|_{L^2}^2 + \|\L^k (a^\varepsilon \m |\bnabla \m|^2)\|_{L^2}^2\right), \\
  \partial_t \|\L^{k+1} \m\|_{L^2}^2 &\le \frac{C}{2\gamma} \left(M^2\|\L^{k+1} \m\|_{L^2}^2 +  \|\bnabla \mathbf{R}_k\|_{L^2}^2 + \|\bnabla \L^k (a^\varepsilon \m |\bnabla \m|^2)\|_{L^2}^2\right).
\end{align}
\end{subequations}
By \cref{eq:nabla_m_square_bound},
we obtain using \Cref{lemma:nonlinear_utility1} and \Cref{lemma:prod_infty_bound}, that
\[
\|\L^k (a^\varepsilon \m |\bnabla \m|^2) \|_{L^2} 
   \le C \frac{1}{\varepsilon^{2k}} \|\m|\bnabla \m|^2\|_{H^{2k}_\varepsilon}
 \le C  \frac{1}{\varepsilon^{2k}} \|\bnabla \m\|_{L^\infty} \|\bnabla \m\|_{H^{2k}_\varepsilon}.
    \]
    According to \Cref{lemma:nonlinear_utility3}, the $L^2$-norm of
    $\R_k$ can be bounded in the same way.

    Next, we consider slight variations of the elliptic regularity
    estimates \cref{eq:elliptic_reg} and
    \cref{eq:nonlinear_utility4} in
    \Cref{lemma:ms_elliptic_regularity}. For $\ell \in \{0, 1\}$ and $0 < 2k+\ell \le q-1$, it
    holds that
    \begin{align}\label{eq:elliptic_reg_grad}
      \|\bnabla \m\|_{H^{2k+\ell}} \le C \left( \frac{1}{\varepsilon^{2k + \ell}} \|\bnabla \m\|_{H^{2k+\ell -1}_\varepsilon} +
      \begin{cases}
         \|\bnabla \L^k \m\|_{L^2}, & \ell = 0, \\
         \|\L^{k+1} \m\|_{L^2}, & \ell = 1,
      \end{cases}
     \right),
    \end{align}
    which in turn implies that
    \begin{align*}
      \|\bnabla \m\|_{H^{2k+\ell}_\varepsilon} \le C \left( \|\bnabla \m\|_{H^{2k+\ell -1}_\varepsilon} +
      \begin{cases}
        \varepsilon^{2k} \|\sqrt{a^\varepsilon} \bnabla \L^k \m\|_{L^2}, & \ell = 0,\\
        \varepsilon^{2k+1} \|\L^{k+1} \m\|_{L^2}, & \ell = 1
      \end{cases}
     \right).
    \end{align*}
    This can be proved proceeding in the same way as for \cref{eq:elliptic_reg} and \cref{eq:nonlinear_utility4}.
Consequently, the right-hand side in \cref{eq:B15} can be bounded such that we get
\begin{align*}
 \partial_t \|\sqrt{a^\varepsilon} \bnabla \L^k \m\|^2_{L^2}
  &\le \frac{C}{\gamma \varepsilon^{2(2k)}} \|\bnabla \m\|_{L^\infty}^2 \|\bnabla \m\|_{H^{2k}_\varepsilon}^2 \\
  &\le \frac{CM^2}{\gamma}  \left(  \|\sqrt{a^\varepsilon} \bnabla \L^j \m\|_{L^2}^2  + \frac{1}{\varepsilon^{2(2k)}}\|\bnabla \m\|_{H^{2k-1}_\varepsilon}^2 \right)
\end{align*}
and similarly,
\begin{align*}
  \partial_t \|\L^{k+1} \m\|^2_{L^2}
  &\le  \frac{C}{\gamma} \|\bnabla \m\|_{L^\infty}^2 \left(\|\L^{k+1}\m\|_{L^2} +  \frac{1}{\varepsilon^{2k+1}}  \|\bnabla \m\|_{H^{2k+1}}\right)^2 \\
  &\le \frac{C M^2}{\gamma}\left( \|\L^{k+1} \m\|_{L^2}^2  + \frac{1}{\varepsilon^{2(2k+1)}}\|\bnabla \m\|_{H^{2k}_\varepsilon}^2
\right).
\end{align*}

Assume now that \cref{eq:energy1} holds for $0 \le j \le 2k$, which
we already showed for $k = 0$. Then it follows from the above estimates  that
\begin{align*}
\partial_t \|\L^{k+1} \m (\cdot, t)\|^2_{L^2}
  &\le \frac{C M^2}{\gamma} \left(\|\L^{k+1} \m (\cdot, t)\|_{L^2}^2 + \frac{1}{\varepsilon^{2(2k+1)}} \|\bnabla \m (\cdot, 0)\|^2_{H^{2k}_\varepsilon} \right).
  \end{align*}
%
%
  and we obtain using Gr\"onwall's inequality 
and \Cref{lemma:nonlinear_utility1}  that
\begin{align*}
  \|\L^{k+1} \m(\cdot, t)\|^2_{L^2}
  &\le e^{C (M^2/\gamma) t} \left(\|\L^{k+1} \m(\cdot, 0)\|_{L^2}^2 +
    C t \frac{1}{\varepsilon^{2(2k+1)}}  \|\bnabla \m (\cdot, 0)\|^2_{H^{2k}_\varepsilon} \right) \\
  &\le C \frac{1}{\varepsilon^{2(2k+1)}}  \|\bnabla \m (\cdot, 0)\|^2_{H^{2k+1}_\varepsilon}.
\end{align*}
Due to elliptic regularity as given in \cref{eq:elliptic_reg_grad} and \cref{eq:energy2}, we then get
\begin{align*}
  \|\bnabla \m(\cdot, t)\|_{H^{2k+1}} &\le C \left(\frac{1}{\varepsilon^{2k+1}}\|\bnabla \m(\cdot, t)\|_{H^{2k}_\varepsilon} + \|\L^{k+1} \m(\cdot, t)\|_{L^2}\right) \\
                                      &\le  \frac{C}{\varepsilon^{2k+1}}  \|\bnabla \m (\cdot, 0)\|_{H^{2k+1}_\varepsilon}.
\end{align*}
Moreover, it holds similarly that
  \begin{align*}
 \partial_t \|\sqrt{a^\varepsilon} \bnabla \L^{k+1} \m(\cdot, t)\|^2_{L^2}
  &\le C\frac{M^2}{ \gamma}\left(\|\sqrt{a^\varepsilon} \bnabla \L^{k+1} \m(\cdot, t)\|_{L^2}^2  + \frac{1}{\varepsilon^{2(2k+2)}} \|\bnabla \m(\cdot, 0)\|_{H^{2k+1}}^2\right),
\end{align*}
from which it follows in the same way as above that
\begin{align*}
  \|\bnabla \m(\cdot, t)\|_{H^{2k+2}} &\le C \left( \frac{1}{\varepsilon^{2k+2}} \|\bnabla \m(\cdot, t)\|_{H^{2k+1}_\varepsilon} +  \|\sqrt{a^\varepsilon} \bnabla \L^{k+1} \m(\cdot, t)\|_{L^2} \right) \\
                                      &\le C \frac{1}{\varepsilon^{2k+2}} \|\bnabla \m (\cdot, 0)\|_{H^{2k+2}_\varepsilon},
\end{align*}
which proves the claim \cref{eq:energy1} by induction. To complete
the proof of \Cref{lemma:l2_m_eps_scalar}, note that
\begin{align*}
  \|\m\|_{H^q} &\le C  \sum_{|\sigma| = 0}^q \|\partial^\sigma \m\|_{L^2}
  \le C \left( \|\m\|_{L^2} + \sum_{|\sigma| = 0}^{q-1} \|\partial^\sigma \bnabla \m\|_{L^2}\right) \\
  &\le C \left( \|\m\|_{L^2} + \|\bnabla \m\|_{H^{q-1}}\right).
\end{align*}
Hence, \cref{eq:Lk_bound} follows from \cref{eq:energy1} and the
fact that $\|\m(\cdot, t)\|_{L^2} = \|\m(\cdot, 0)\|_{L^2}$ due to
the norm preservation property of the Landau-Lifshitz equation as
shown in (A2).

\end{proof}

\fi

\end{document}